\documentclass[11pt]{amsart}
\usepackage[utf8]{inputenc}
\usepackage{graphicx,amssymb,amsmath,amsthm,bm}
\usepackage{caption}
\usepackage{subcaption}
\usepackage{indentfirst}
\usepackage{cancel}
\usepackage{lipsum}
\usepackage{epstopdf,epsfig}
\usepackage{enumerate,color}

\usepackage{array,ragged2e}
\newcolumntype{C}{>{\Centering\arraybackslash}m{0.14\linewidth}}

\usepackage{accents}
\newcommand*{\dt}[1]{%
   \accentset{\mbox{\large\bfseries .}}{#1}}
\usepackage[overload]{empheq}
\usepackage{soul}
\usepackage[multiple]{footmisc}
\newcommand{\etba}{\beta}

\newcommand{\mc}{\mathcal}
\newcommand{\mb}{\mathbb}

\newcommand{\veps}{\varepsilon}
\newcommand{\oRe}{\operatorname{Re}}
\newcommand{\oIm}{\operatorname{Im}}

\numberwithin{equation}{section}
\everymath{\displaystyle}

\theoremstyle{plain}
\newtheorem{theorem}{Theorem}[section]
\newtheorem{lemma}[theorem]{Lemma}
\newtheorem{proposition}[theorem]{Proposition}
\newtheorem{corollary}[theorem]{Corollary}
\theoremstyle{definition}

\newtheorem*{defi*}{Definition} 
\theoremstyle{remark}
\newtheorem{remark}{Remark}

\theoremstyle{remark}

\usepackage{etoolbox}
\makeatletter
\let\alignts@preamble\align@preamble
\patchcmd{\alignts@preamble}{\displaystyle}{\textstyle}{}{}
\patchcmd{\alignts@preamble}{\displaystyle}{\textstyle}{}{}

\def\alignts{\let\align@preamble\alignts@preamble\start@align\@ne\st@rredfalse\m@ne}

\makeatother

\allowdisplaybreaks

\title[Linear Stability of the Euler-Poisson System]{Linear Stability of solitary waves for the isothermal Euler-Poisson system}

\author[J. Bae]{Junsik Bae}
\address[JB]{Mathematics Division, National Center for Theoretical Sciences, National Taiwan University, No. 1, Sec. 4, Roosevelt Rd., Taipei 10617, Taiwan}
\email{jsbae@ncts.ntu.edu.tw}
 
\author[B. Kwon]{Bongsuk Kwon}
\address[BK]{Department of Mathematical Sciences, Ulsan National Institute of Science and Technology, Ulsan, 44919, Korea}
\email{bkwon@unist.ac.kr}

\date{\today}


\begin{document}
 
\maketitle 

\begin{abstract} 
We study the asymptotic linear stability of a two-parameter family of solitary waves for the isothermal Euler-Poisson system.
When the linearized equations about the solitary waves are considered, the associated eigenvalue problem in $L^2$ space has a zero eigenvalue embedded in the neutral spectrum, i.e., there is no spectral gap. To resolve this issue, use is made of an exponentially weighted $L^2$ norm so that the essential spectrum is strictly shifted into the left-half plane, and this is closely related to the fact that  solitary waves exist in the \emph{super-ion-sonic} regime. Furthermore, in a certain long-wavelength scaling, we show that the Evans function for the Euler-Poisson system converges to that for the Korteweg-de Vries (KdV) equation as an amplitude parameter tends to zero, from which we deduce that the origin is the only eigenvalue on its  \emph{natural domain} with algebraic multiplicity two. We also show that the solitary waves are spectrally stable in $L^2$ space. Moreover, we discuss (in)stability of \emph{large} amplitude solitary waves.\\

\end{abstract}

\section{Introduction}

 We consider the one-dimensional isothermal Euler-Poisson system  in a non-dimensional form:
\begin{equation}\label{EP}
\left\{
\begin{array}{l l}
\partial_{t} n + \partial_{s}((1+n)u) = 0, \\ 
\partial_{t} u  + u \partial_{s} u +  K \partial_{s}\log (1+n) = -\partial_{s}\phi, \\
-\partial_{s}^2\phi = (1+n ) - e^\phi, \ \ s\in\mathbb{R}, t\ge0,
\end{array} 
\right.
\end{equation}
where $1+n$, $u$ and $\phi$ represent   the ion density, the fluid velocity function for ions, and  the electric potential, respectively, and $K=T_i/T_e>0$ is a constant of the ratio of the ion temperature to the electron temperature. The Euler-Poisson system is a fundamental fluid model which describes the dynamics of ions in an electrostatic plasma, and it is often employed to study phenomena of plasma such as plasma sheaths and double layers  (see \cite{Ch,Dav} for more physical backgrounds). In the one-fluid model \eqref{EP}, the electron density $n_e$ is assumed to satisfy the Boltzmann relation $n_e=e^{\phi}$, which can be foramlly derived from the two-fluid model under the massless electron assumption.  A rigorous justification of this zero mass limit is discussed in \cite{GGPS}.

\newcommand{\ve}{\varepsilon}
Among others, the emergence of solitary waves is one of the most interesting phenomena in the dynamics of electrostatic plasma. Motivated by physicists' finding that the KdV equation, originally derived to describe the motion of water waves, can be also derived in the study of hydromagnetic waves in plasmas, see \cite{Gar}, plasma physicists have sought formal connections between the  Euler-Poisson system (pressureless case $K=0$ for the sake of simplicity of analysis) and the KdV equation, see \cite{Sag,Wa}. Later on, the phenomenon was also experimentally observed in \cite{ikezi}.

 Various analytical \cite{BK,Guo} and numerical \cite{HNS,Satt} studies indicate that in certain physical regime, the KdV equation is a good approximation of the Euler-Poisson system \eqref{EP}. Moreover, as solutions of the KdV equation are dominated by their solitary waves  \cite{Benjamin,Bona,GGKM1,GGKM2,Lax1,MM1,MM2,MMT,PW1,ZK}, this gives hope of a similar result for the Euler-Poisson system with more general initial data. 
This motivation naturally leads us to the study of stability of solitary waves for the Euler-Poisson system.

In fact, it is shown in \cite{Cor} that the Euler-Poisson system \eqref{EP}  with the far-field condition
\begin{equation}\label{bdCon x2}
(n, u, \phi) \to (0, 0, 0)  \quad \text{as} \quad s \to \pm\infty
\end{equation}
admits  a two-parameter family of traveling solitary wave solutions\footnote{By letting $c=-c'$ and $u=-u'$, we obtain the traveling waves moving to the left direction.}
 \[
(n, u,\phi)(s,t)=(n_c,u_c,\phi_c)(s-ct+\gamma), \quad   c\in(\sqrt{K+1}, \sqrt{K+1}+\veps_K)
\;  \text{and} \; \gamma \in \mathbb{R},
\]
where $\veps_K>0$ is a critical value (see \eqref{ep0}) and $\sqrt{1+K}$ is called the \emph{ion sound speed} in the context of plasma physics. Furthermore, the authors of this paper show in \cite{BK} that $(n_c, u_c, \phi_c)$ converges to the rescaled solitary wave solution of the associated KdV equation as the amplitude parameter $\veps>0$ tends to zero.
More specifically, in the Gardner-Morikawa scaling (also called as the KdV scaling)
\[
\xi :=\veps^{1/2}x = \veps^{1/2}(s-ct), \quad c=\sqrt{1+K} + \veps,
\]
it is shown that 
\begin{equation*}
\phi_c ( \veps^{-1/2}\xi ) - \ve \Psi_{KdV} (\xi)  = O(\veps^2) \quad \text{as } \veps \to 0,
\end{equation*}
where 
\begin{equation}\label{solutionKdV2}
\Psi_{KdV}(\xi) := \frac{3}{\sqrt{1+K}}\textnormal{sech}^2 \left(\sqrt{{\sqrt{1+K}}/{2}}\,\xi \right), 
\end{equation}
and it satisfies 
\begin{equation}\label{KdV}
- \partial_\xi \Psi_{KdV}  +\sqrt{1+K}\Psi_{KdV}\partial_\xi\Psi_{KdV}  +  \frac{1}{2\sqrt{1+K}} \partial_\xi^3 \Psi_{KdV} =0.
\end{equation}
In this paper we investigate the linear stability of this family of \emph{small} amplitude solitary waves.  



Due to the translation invariance 
and the fact that the traveling speed is a parameter, the linearized Euler-Poisson system around the solitary waves has two linearly independent solutions which do not decay in time (we call these solutions non-decaying modes).
In light of this,  we show that the family of solitary wave solutions to the Euler-Poisson system  \eqref{EP}--\eqref{bdCon x2} is \emph{linearly asymptotically stable modulo} the non-decaying modes. More precisely, for the initial data having no component of the non-decaying modes, the solution to the linearized Euler-Poisson system exponentially decays to zero as $t \to +\infty$. The aymptotic stability result is established in terms of strongly continuous semigroup on exponentially weighted $L^2$-spaces. Introduction of the weighted norm  is closely related to the fact that 
 traveling solitary wave solutions exists in the \emph{super-ion-sonic} regime, i.e.,
$c>\sqrt{1+K}$. 

We also prove that in the usual $L^2$-space, the spectrum of the operator associated with the linearized Euler-Poisson system is precisely the imaginary axis, and hence the family of solitary waves are \emph{spectrally stable}. Our stability results hold for small amplitude solitary waves.



\subsection{Main Results}

In the moving frame $x=s-ct$, 
the linearized system of \eqref{EP} around the solitary wave $(n_c,u_c,\phi_c)(x)$  is given by
\begin{subequations}\label{EP3}
\begin{align}[left = \empheqlbrace\,]
& \partial_t
\begin{pmatrix}
\dot{n} \\
\dot{u}
\end{pmatrix} 
+ L \partial_x \begin{pmatrix}
\dot{n} \\
\dot{u}
\end{pmatrix} + (\partial_xL) \begin{pmatrix}
\dot{n}\\
\dot{u}
\end{pmatrix}  = \begin{pmatrix}
0 \\
-\partial_x \dot{\phi}
\end{pmatrix}, \label{EP3Evol} \\ 
& -\partial_{x}^2\dot{\phi}= \dot{n}  - e^{\phi_c}\dot{\phi}, \label{EP3Poi}
\end{align}
\end{subequations}
where $L=L(x,\veps)$ is the matrix defined by
\begin{equation}\label{DefL}
L:=\begin{pmatrix}
-c+u_c & 1+n_c \\
\frac{K}{1+n_c} & -c + u_c 
\end{pmatrix}.
\end{equation}

We define the exponentially weighted $L^2$ space and the associated Sobolev spaces 
\begin{equation}\label{Def_WeightNorm}
\|f(x)\|_{L_\etba^2(\mathbb{R})}:=\|e^{{\etba} x}f(x)\|_{L^2(\mathbb{R})} \quad \text{and} \quad \|f(x)\|_{H_\etba^s(\mathbb{R})}:=\|e^{{\etba} x}f(x)\|_{H^s(\mathbb{R})},
\end{equation}
where $H^s(\mathbb{R})$ is the usual $L^2$-Sobolev norm, and  $\etba>0, s> 0$. We   sometimes use the notations $L_0^2$ and $H_0^s$ for the unweighted $L^2$ and $H^s$-spaces, respectively, with no confusion.

It can be easily checked that for given $\dot{n} \in L_\beta^2(\mathbb{R})$, where $\beta \in [0,1)$, there exists a unique solution $\dot{\phi}=\dot{\phi}(\dot{n})= (-\partial_x^2 +e^{\phi_c})^{-1} (\dot{n})$ in $H_\beta^2(\mathbb{R})$ to the linear Poisson equation \eqref{EP3Poi} since $0<\phi_c\in L^\infty(\mathbb{R})$. (One can consider the change of variable $(\widetilde{n},\widetilde{\phi}):=(\dot{n},\dot{\phi})e^{\beta x}$.) Hence, we may rewrite \eqref{EP3} in a more compact (nonlocal) form:
 \begin{equation}\label{EP-nonlocal}
(\partial_t -\mathcal{L})(\dot{n},\dot{u})^T =(0,0)^T,
\end{equation}
where  
\begin{equation}\label{Ch3_Def_mcL}
\mathcal{L}\begin{pmatrix}
\dot{n} \\
\dot{u}
\end{pmatrix} := -\partial_x \left[L \begin{pmatrix}
\dot{n} \\
\dot{u}
\end{pmatrix}
 + \begin{pmatrix}
 0 \\
 (-\partial_x^2 +e^{\phi_c})^{-1}(\dot{n})
 \end{pmatrix} \right].
\end{equation}
The eigenvalue problem associated  with \eqref{EP3} is given by
\begin{equation}\label{EigenEP}
(\lambda -\mathcal{L})(\dot{n},\dot{u})^T =(0,0)^T.
\end{equation}

Let us clearly define some terminologies. We say that $\lambda\in\mathbb{C}$ is in the \emph{resolvent set} if $\lambda  - \mathcal{L}$ is has the bounded inverse operator $(\lambda  - \mathcal{L})^{-1}$, which is called the \emph{resolvent}. We call the complement of the resolvent set the \emph{spectrum} of $\mathcal{L}$, denoted by $\sigma(\mathcal{L})$, which can be decomposed in terms of the Fredholm properties of the operator $\lambda  - \mathcal{L}$ as follows:
\begin{enumerate}
\item we say that $\lambda \in \sigma(\mathcal{L})$ is in the \emph{point spectrum}, $\sigma_{\text{pt}}(\mathcal{L})$, if $\lambda  - \mathcal{L}$ is Fredholm with index zero, but it is not invertible;
\item $\sigma_{\text{ess}}(\mathcal{L}):=\sigma(\mathcal{L})\setminus \sigma_{\text{pt}}(\mathcal{L})$ is called the \emph{essential spectrum} of $\mathcal{L}$.
\end{enumerate}
We say that $\lambda\in \sigma(\mathcal{L})$ is an \emph{eigenvalue} of $\mathcal
{L}$ if the kernel of $\lambda  -\mathcal{L}$ is a nontrivial subspace of the domain of $\mathcal{L}$.



Due to the translation invariance and the fact that the speed $c$ is a parameter, $\lambda=0$ is an $L^2$-eigenvalue of the operator $\mathcal{L}$ with algebraic multiplicity at least two.  Indeed, we will see that 
\begin{equation}\label{JordChainforL}
\mathcal{L} \partial_x(n_c, u_c)^T = (0,0)^T, \quad \mathcal{L}\partial_c(n_c,u_c)^T = -\partial_x(n_c,u_c)^T.
\end{equation}
Thus
\[
\partial_x(n_c,u_c)^T(x) \quad \text{and} \quad \partial_c(n_c,u_c)^T(x) - t\partial_x(n_c,u_c)^T(x)
\]
are non-decaying (in time) solutions to \eqref{EP3}.

Since the solitary waves exponentially decay to zero as $|x| \to +\infty$, the essential spectrum of $\mathcal{L}$ in $L^2$-space coincides with the imaginary axis in the complex plane, and the zero eigenvalue is embedded in the essential spectrum. 
Moreover, the point spectrum of $\mathcal{L}$ in $L^2$-space is empty. 

For a Hilbert space $\mc{H}$, we denote $\mathcal{H}\times \mathcal{H}$ by $(\mathcal{H})^2$. We present the result of \emph{spectral stability}. 
\begin{proposition}[Spectrum of $\mathcal{L}$ in  $L^2$-space]\label{SpecStabL2}
Consider the operator $\mathcal{L}: (L^2)^2 \to (L^2)^2$ with dense domain $(H^1)^2$. Then, for all sufficiently small $\veps>0$, there holds
\[
\sigma(\mathcal{L})=\sigma_{\mathrm{ess}}(\mathcal{L})=\{\lambda\in\mathbb{C}: \oRe \lambda=0\}.
\]
\end{proposition}

However, in terms of a standard semigroup approach, the spectral stability itself is not sufficient to conclude the asymptotic linear stability.  
We resolve this issue by employing the weighted space $L^2_\etba(\mathbb{R})$ defined in \eqref{Def_WeightNorm}. For appropriately chosen $\etba>0$, the essential spectrum of $\mathcal{L}$ is strictly shifted into the open left-half plane, while   $\lambda=0$   is still remains as the  eigenvalue of $\mathcal{L}$ (see Figure \ref{Fig2}). Furthermore, $\lambda=0$ is the only $L_\beta^2$-eigenvalue of $\mathcal{L}$ on some closed set containing the closed right-half plane, and its algebraic multiplicity is two. The corresponding eigenvector and the generalized eigenvector  are given by $\partial_x(n_c, u_c)^T$ and $\partial_c(n_c, u_c)^T$, respectively. This idea separating the essential spectrum and the embedded eigenvalue by employing the weighted space was first introduced by \cite{Sattinger} in the study of stability of traveling waves of parabolic system, and it is successfully adopted later to the study of  stability of the KdV solitary waves in \cite{PW1}.  

\begin{figure}[h]
\begin{tabular}{cc}
\resizebox{55mm}{!}{\includegraphics{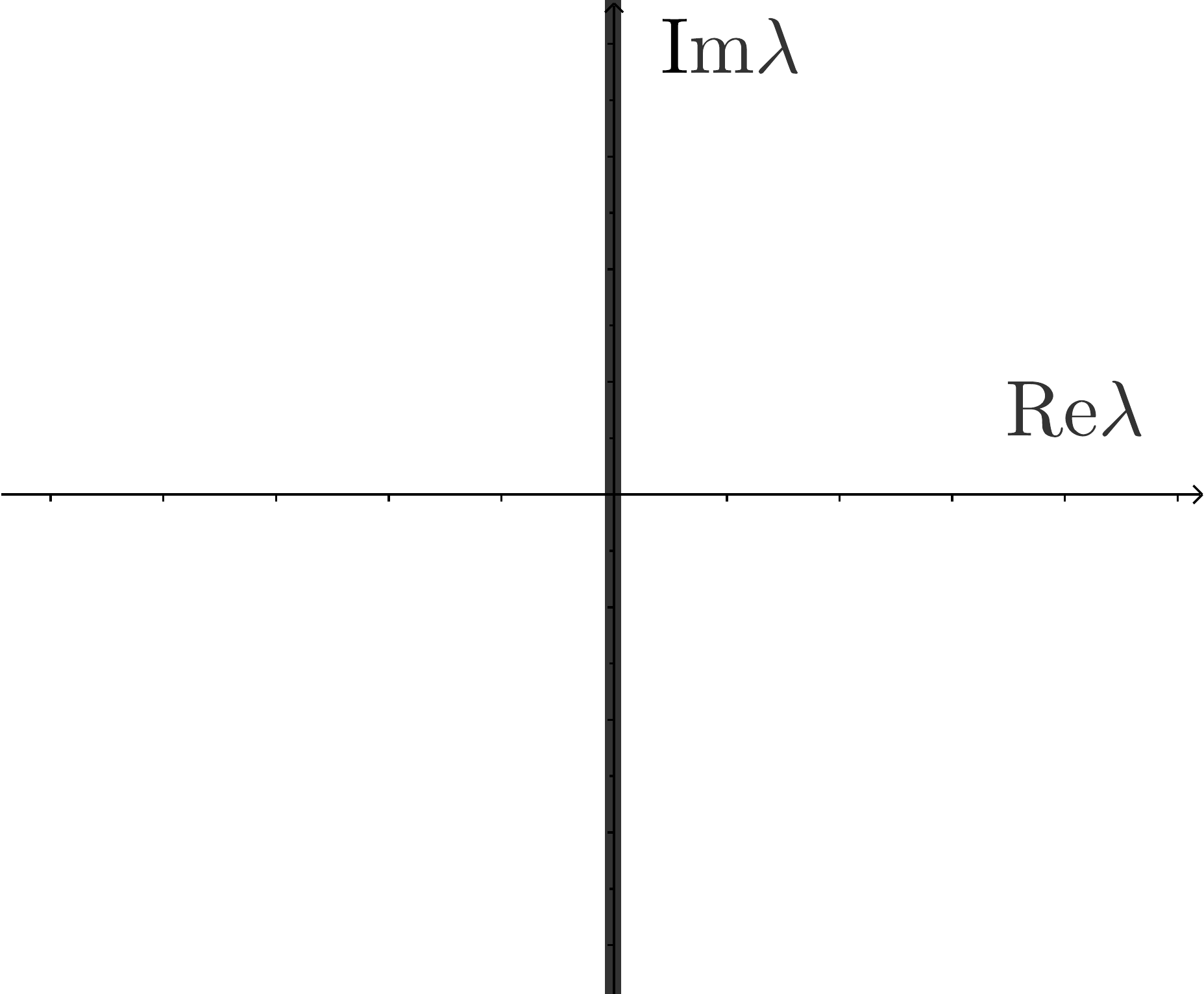}}  & \resizebox{50mm}{!}{\includegraphics{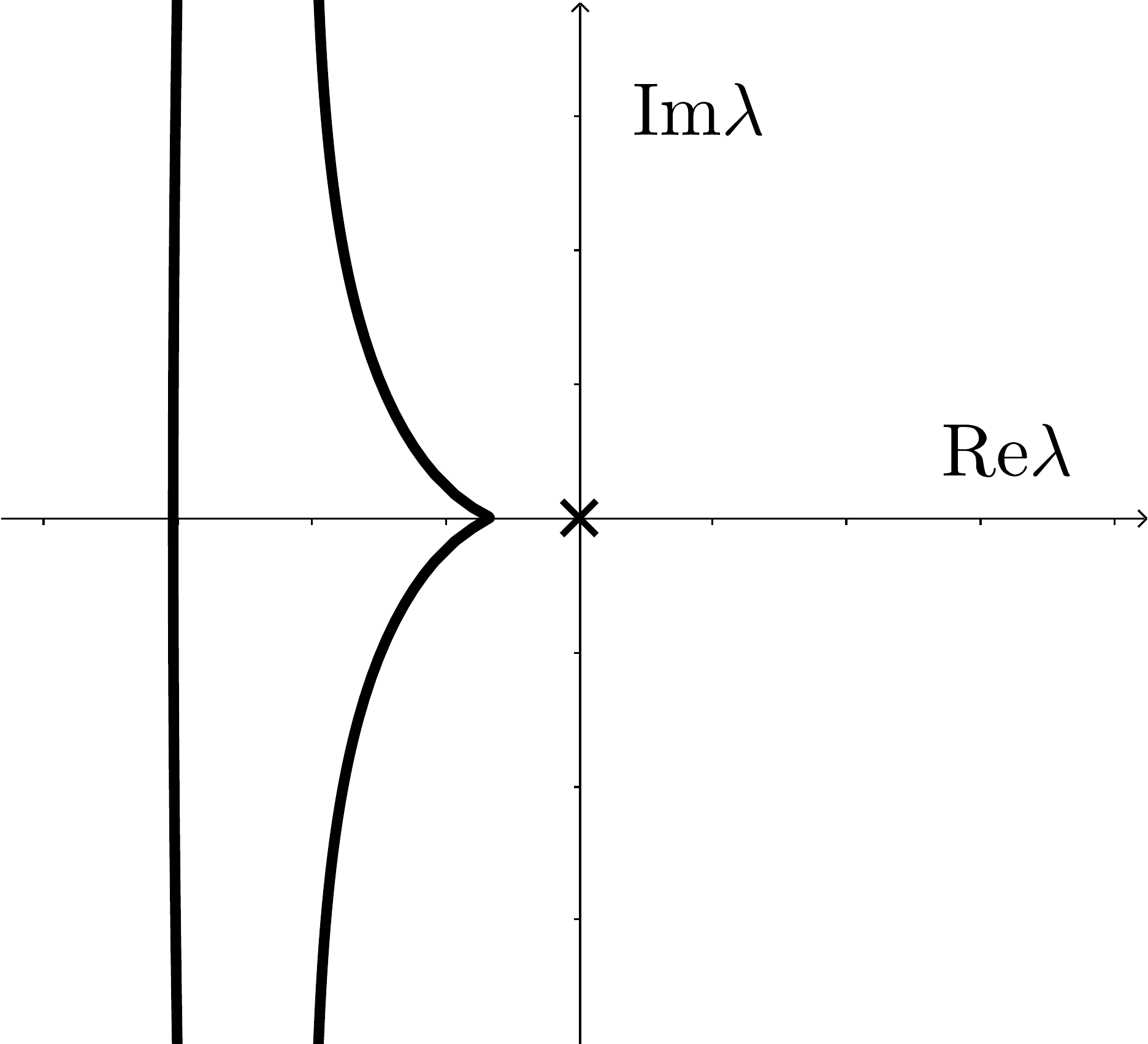}} \\
(a) Spectrum   in $L^2$  & (b) Spectrum in $L^2_\etba$  
\end{tabular}
\caption{The bold curves indicate the essential spectrums of $\mc{L}$. In (b), the zero eigenvalue of  $\mc{L}$ is isolated in   $L^2_\etba$ for sufficiently small $\etba>0$.}
\label{Fig2}
\end{figure}

We first present some preliminary results for the linear asymptotic stability.  
\begin{proposition}\label{MainIngred}
Consider the operator $\mathcal{L}: (L_\etba^2)^2 \to (L_\etba^2)^2$ with dense domain $(H_\etba^1)^2$. For any fixed $c_0\in(0, \sqrt{ 2{\sqrt{1+K}}/3} )$, let $\etba=c_0\veps^{1/2}$. Then there exists $\veps_0>0$ such that for all $\veps \in (0,\veps_0)$, the following holds:
\begin{enumerate}[{(i)}]
\item \label{Semig} $\mc{L}$ generates a $C_0$-semigroup, $e^{\mathcal{L}t}$.
\item \label{AlMul2} $\lambda=0$ is an isolated eigenvalue of $\mathcal{L}$ with algebraic multiplicity two.
\item \label{MainIngEss} $
\sigma_{\mathrm{ess}}(\mathcal{L})\subset \{\lambda: \oRe \lambda < -\veps^{3/2}\eta(c_0)\}$ and $ \sigma_{\mathrm{pt}}(\mathcal{L})\cap \{\lambda: \oRe \lambda \geq  -\veps^{3/2}\eta(c_0)\}=\{0\}$, where $\eta>0$ is a function of $c_0$ defined in  \eqref{ND-eps}. In particular, $\sigma_{\mathrm{ess}}(\mathcal{L})$  is a union of two curves paramterized by
\begin{equation*}
 d_{\pm}(ik-\etba)= (ik-\etba) \left(c \pm \sqrt{\frac{1}{1-(ik-\etba)^2}+ K} \right), \quad k\in \mathbb{R}.
\end{equation*}
\item \label{UnibdRes} $(\lambda-\mc{L})^{-1}$ is uniformly bounded on $\oRe \lambda \geq 0$, outside any small neighborhood of the origin.
\end{enumerate}
\end{proposition}
 
Our main result follows from Proposition \ref{MainIngred} and the result of Pr\"uss \cite{Pruss}.  Let  $\mathcal{P}_0$ be the spectral projection associated with the isolated eigenvalue $\lambda=0$. 
\begin{theorem} \label{LinStab} (Asymptotic linear stability in weighted $L^2$-spaces) Under  the same assumptions as in Proposition \ref{MainIngred}, the following statement holds: for any given $(n_0,u_0)^T \in (L_\etba^2)^2$ satisfying $\mathcal{P}_0(n_0,u_0)^T =0$, there holds
\begin{equation}\label{LinStabEq}
\|e^{\mathcal{L}t}(n_0,u_0)^T\|_{(L_\etba^2)^2} \leq C_1 e^{-C_2 t} \|(n_0,u_0)^T\|_{(L_\etba^2)^2}, \quad \forall t\geq 0,
\end{equation}
for some constants $C_1,C_2>0$ depending on $\veps$.
\end{theorem}


The semigroup estimates \eqref{LinStabEq} holds  for any solution to the linearized Euler-Poisson system \eqref{EP3} with no component of the non-decaying modes. 


 The essential spectrum of $\mathcal{L}$ in exponentially weighted spaces is shifted into the left half plane since the end state of solitary wave solutions lies in a \emph{super-ion-sonic regime}, i.e., $c>\sqrt{1+K}$. More precisely,  
the dispersion relation is  
\[
\omega_\pm(k)= -k\left( c \pm \sqrt{\frac{1}{1+k^2} + K } \right),
\]
from which we find that the group velocities are strictly negative, i.e., $\partial_k\omega_\pm(k) \leq  -\veps$,
 and we have for small $\beta>0$,
\[
\oRe d_\pm(ik-\etba) \approx \partial_k\omega_\pm(k)\beta.
\]
In particular, the real part of the rightmost point of the essential spectrum of $\mathcal{L}$ is $-\veps\beta<0$. In light of this, we see that the system is dissipative at the end state, i.e., near the tail part of solitary waves. However, the linearized system throughout the wave does not have a definite sign, for which one cannot obtain the decay estimate by a standard energy method.

It is worthwhile to remark that the energy method, successfully used to study the stability and quasi-neutral limit of the solutions near the boundary layers representing plasma sheaths for \eqref{EP} in the weighted spaces under the supersonic condition $c>\sqrt{1+K}$ (see \cite{S12, JKS19}), does not work for our current problem since 
our solitary waves are \emph{not small}. More specifically, 
as the solitary wave amplitude $O(\veps)$ gets smaller, the traveling speed $c$ also gets closer to the ion sound speed in $O(\veps)$ order, and vice versa. This leads to a smaller \emph{damping effect} of order $O(\veps)$ due to a combination of the weighted space and transport effect. Hence, rather than a standard energy method, a more detailed spectral analysis is required to study the asymptotic stability of traveling solitary waves.

One of the main tasks in our analysis is to characterize the eigenvalues of $\mathcal{L}$, for which  we employ  the Evans function techniques. The Evans function is an analytic function of the spectral parameter $\lambda$, and on the \emph{natural domain}, the location and order of zeroes of the Evans function coincide with those of  eigenvalues of $\mathcal{L}$. For the associated eigenvalue problem \eqref{EigenEP}, the natural domain of the Evans function is $\{\lambda:\oRe \lambda>0\}$ in $L^2$ space, and the natural domain can be extended to contain $\{\lambda:\oRe \lambda \geq 0\}$ in the exponentially weighted space $L^2_\beta$. 
  The Evans function was first intoduced in \cite{Evans1}--\cite{Evans4}, and successfully developed  in the study of stability of traveling waves in a various contexts, see \cite{AGJ, Jones, PW, PW1, ZH} for their pineering works. We also refer to \cite{Sand,Kapi} and the referneces therein.

%
In general, an \emph{explicit} form of the Evans function is not available except for a few cases.
To overcome this issue, we make use of a specific scale, related to the Gardner-Morikawa transformation,
\begin{equation}\label{GMT}
\xi = \veps^{1/2}x, \quad \lambda = \veps^{3/2}\Lambda,
\end{equation}
and observe that as $\veps$ tends to zero, the rescaled eigenvalue problem  for the Euler-Poisson system can be formally reduced to the eigenvalue problem for the associated KdV equation (see Appendix \ref{Appendix1})
\begin{equation}\label{FormDerivEigEPtoKdV4}
\Lambda \dot{n}_\ast - \partial_\xi \dot{n}_\ast  + \sqrt{1+K}\partial_\xi(\Psi_{KdV} \dot{n}_\ast) + \frac{1}{2\sqrt{1+K}}\partial_\xi^3 \dot{n}_\ast = 0, \quad (\dot{n}_\ast(\xi):=\dot{n}(x)),
\end{equation}
for which an explicit form of the Evans function $D_{KdV}(\Lambda)$ is established in  \cite{PW1}; it vanishes only at $\Lambda=0$ with multiplicity of two.  
In light of this, we show that in the scaling \eqref{GMT},  the Evans function $D(\lambda,\veps)$ for the Euler-Poisson system  converges to that for the associated KdV equation  as $\veps$ tends to zero, and that the convergence is uniform on a domain containing the closed right-half plane.
By this together with Rouch\'e's theorem, we deduce that $\lambda=0$ is a zero of $D(\lambda,\veps)$ with multiplicity two, and there is no other zero satisfying $\oRe \lambda \geq -\veps^{3/2} \eta(c_0)$. This is discussed in Section~\ref{S2}. The relations between the Evans functions and the associated eigenvalue problems are summarized in the following diagram:
\[
\scriptsize{
\begin{array}{ccccccc}
\frac{d\mathbf{y}}{dx} = A(x,\lambda,\veps) \mathbf{y} & 
\Leftrightarrow &
\frac{d\mathbf{p}}{d\xi} = A_\ast(\xi,\Lambda,\veps) \mathbf{p},  & \rightarrow & \frac{d\mathbf{p}}{d\xi} = A_\ast(\xi,\Lambda,0) \mathbf{p} &  \Leftrightarrow & \text{Eq. }\eqref{FormDerivEigEPtoKdV4} \\[0.2cm]
\Big\vert\scriptstyle{\textrm{Prop.\ref{DefEvansFunctDomain}}} & \boxed{\!\begin{aligned}
  & \scriptstyle{\xi = \veps^{1/2}\,x,} \\
  & \scriptstyle{\lambda = \veps^{3/2}\Lambda}
  \end{aligned}} & \Big\vert \scriptstyle{\textrm{Prop.\ref{DefEvansFunctDomain2}}}   & \text{as }\veps \to 0 & \Big\vert\scriptstyle{\textrm{Prop.\ref{DefEvansFunctDomain2}}} &  & \Big\vert\scriptstyle{\textrm{Prop.\ref{KdVSummary}}} \\[0.2cm]
D(\lambda,\veps) &
\underset{\mathrm{Prop.}\ref{RelationDandDtilde}}{=} &
D_\ast(\Lambda,\veps) & \underset{\textrm{Prop.}\ref{ThmEvansConverg}}{\rightarrow}  & D_\ast(\Lambda,0) & \underset{\mathrm{Prop.}\ref{RelationDandDtilde}}{=} & D_{\text{KdV}}(\Lambda)
\end{array}
}
\]

Our strategy for studying the eigenvalue problem is borrowed from the work of \cite{PW2} concerning stability of solitary waves for some Boussinesq systems.
We also refer to \cite{HS1,Li,MW} for similar approaches for various water wave models and \cite{Sche} concerning the solitary waves for the pressureless  $(K=0)$ Euler-Poisson model.

Another key ingredient in the analysis of the linear asymptotic stability  is  establishing the \emph{uniform} (\emph{in} $\lambda$) boundedness of the resolvent operator  $(\lambda-\mathcal{L})^{-1}$, restricted to the complementary spectral projection $I-\mathcal{P}_0$\textcolor{blue}{,} on $\{\lambda:\oRe \lambda \geq 0\}$ (see \cite{Nagel}, Chapter V, p. 304 for a counterexample), and it is  a non-trivial task even in the case of a scalar equation in general (see \cite{Kapi}, Chapter 4.5 or \cite{KapiSand} for instance). 
%
%
%
To obtain the uniform resolvent bounds, we consider Green's function for the first-order ODE system associated with the eigenvalue problem for the Euler-Poisson system and apply a perturbation argument so-called \emph{the roughness of exponential dichotomies} (see \cite{Coppel2}).
We remark that the uniform resolvent bound for the case $K=0$ is obtained in \cite{Sche} by considering a certain change of variables in such a way that the transformed operator can be written as a sum of a constant coefficient first-order operator and a small bounded operator in $L_\etba^2\times H_\etba^1$ space, which is a suitable solution space for the case. However, such an approach is not available in the case $K>0$ since the small perturbation term is not bounded in $L_\etba^2\times L_\etba^2$ space.

In the context of compressible fluid models of \emph{cold plasmas} (also in the study of \emph{sticky} particles), the pressure term in the momentum equation is sometimes ignored for the sake of simplicity.
In a mathematical point of view, 
the absence of the pressure term makes the system weakly coupled,
so one can take an advantage of its simpler structure to analyze some properties of the system.  
For instance, by a simple ODE technique, the formation of singularity can be easily shown for \eqref{EP} with $K=0$, see e.g., \cite{Liu}, while no comparable result is known for \eqref{EP} with $K>0$.   
However, we remark that the pressureless Euler-Poisson system has  qualitatively different properties from those of \eqref{EP} with $K>0$.
 One of the significant differences is the emergence of \emph{delta shocks};
provided that the initial data $u_0$ has a \emph{steep} gradient at some point, one can easily show, by an ODE technique, that the gradient of velocity blows up at finite time $T_*$ and it occurs at non-integrable order in time, see \eqref{tstar} in Section \ref{Subsec8.2}. 
  From this together with the continuum equation, we see that 
 $L^\infty$ norm of density is unbounded as $t$ approaches $T_*$. This is not the case in the presence of the pressure, in general. 
 We refer to Section~\ref{RO} for a more detailed discussion and some numerical experiments exhibiting this density blow-up phenomena in terms of solitary waves for the case $K=0$. 

Another interesting point is that the appropriate  solution spaces for the pressureless system are different from those for \eqref{EP} with $K>0$; the latter is a symmetrizable system. For the well-posedness for the case $K=0$, the velocity function requires to have one better differentiability than that of density due to the structure of the equations. 
This causes different difficulties in the analysis  such as establishing the uniform resolvent bound as mentioned earlier.

Before we close the introduction, we remark on some issues related to the nonlinear stability. To tackle the asymptotic nonlinear stability, we encounter several obstacles to overcome. Among others, one of the most fundamental issues is to show global existence of smooth solutions. 
As discussed earlier, the linear system has a neutral spectrum, but the dispersion relation reveals the system is weakly dispersive. To our best knowledge, no global existence result has been proved for the 1D Euler-Poisson system \eqref{EP}. Since the system is weakly dispersive but strongly nonlinear, it is not expected that smooth solutions exist globally, in general. Regarding the global existence of the Euler-Poisson system, there are a few positive results. It is proved in \cite{GP} that \emph{curl-free} small smooth solutions of the 3D Euler-Poisson system near the constant state can exist globally in time.
When the initial-boundary value problem for the Euler-Poisson system is considered in the half space $\mathbb{R}_+^n$, $n\ge1$, global smooth solutions near the boundary layer solutions (called \emph{plasma sheaths}) is proved to exist, in which  no curl-free condition is required, see \cite{JKS19,S12}.
The authors establish the decay estimates by taking advantage of the exponentially weighted space and the assumption that the flow is \emph{super-ion-sonic},  referred to as Bohm's criterion, which is physically relevant to the formation of plasma sheath.  There are two reasons (even at the linear level) why similar estimates cannot be drawn for our present problem: 
 (i) the solitary waves we consider are \emph{not small} and (ii) the dissipative mechanism due to the boundary is absent. To our best knowledge, the study of global existence is widely open, and it seems a challenging problem. Some related issues specially for the one-dimensional pressureless Euler-Poisson system is discussed in Section~\ref{RO}.

This paper is organized as follows. In Section~\ref{S2-1}, we summarize some key properties of the solitary waves that are used for our stability analysis. Section~\ref{S2} is devoted to classifying  the spectrum of the linearized operator $\mathcal{L}$ by using the Evans function. Some crucial lemmas, such as the behaviors (in $\lambda)$ of the matrix eigenvalues of the asymptotic matrix for the associated first order ODE system, required to our analysis   will be  proved throughout Section~\ref{Sec_Split} and Section~\ref{Sec5}. Due to the presence of the pressure, the form of the \emph{generalized} dispersion relation \eqref{Charact1} of our model is different from the ones of \cite{Sche,Li,PW2}.   The splitting properties of  matrix eigenvalues (Lemma \ref{splitting mu lam} and Lemma \ref{splitting mu lam weight}) will be proved in Section~\ref{Sec_Split}. The behavior of the Evans function for the small and large eigenvalue parameter as well as the uniform resolvent estimate will be covered in Section~\ref{Sec5}. Summarizing all preliminary results, Theorem \ref{LinStab} will be proved in Section~\ref{Sec_LinConvec}. Some additional topics will be discussed in Section 7 and Section 8. In Section 7, we will present an $L^2$-instability criterion, which is applicable to the Euler-Poisson solitary waves with \emph{large amplitudes}, as well as its numerical experiments. In Section 8, we will discuss the question of global existence vs. finite time singularity formation of the pressureless Euler-Poisson system.

\section{Solitary waves for the Euler-Poisson system}\label{S2-1}
In this section, we briefly discuss some key properties of the solitary waves for \eqref{EP}. 
Plugging the traveling wave Ansatz $(n,u,\phi)(x)$, where $x=s-ct$, into  \eqref{EP}, we obtain
\begin{subequations}\label{EP2Travel2}
\begin{align}[left = \empheqlbrace\,]
& - c\,\partial_x n + \partial_{x}((1+n)u) = 0, \label{EP2Travel2mas} \\ 
& -c\,\partial_x u  + u\partial_{x} u + K \partial_{x} \log(1+n)  = -\partial_{x}\phi, \label{EP2Travel2mom} \\ 
& -\partial_{x}^2\phi = (1+n) - e^{\phi}. \label{EP2Travel2Poi}
\end{align}
\end{subequations}
Integrating the first two equations of \eqref{EP2Travel2} in $x$, one obtains from the far-field condition  \eqref{bdCon x2} that
\begin{subequations}\label{TravelEqB}
\begin{align}[left = \empheqlbrace\,]
& (1+n)(c-u)=c, \label{TravelEqB1} \\ 
& \phi = H(n,c):=\frac{c^2}{2}\left(1 - \frac{1}{(1+n)^2} \right)  - K\ln (1+n), \label{TravelEqB2}
\end{align}
\end{subequations}
where \eqref{TravelEqB1} is used to get \eqref{TravelEqB2}. Hence,   \eqref{EP2Travel2}  can be reduced as the first order ODE system for $n$ and $\partial_x \phi$:
\begin{subequations}\label{ODE_n_E}
\begin{align}[left = \empheqlbrace\,]
& \partial_n H(n,c) \partial_x n = \partial_x\phi, \label{ODE_n_E1} \\
& -\partial_x^2\phi = 1+n - e^{H(n,c)}.\label{ODE_n_E2}
\end{align}
\end{subequations} 
The first integral of \eqref{ODE_n_E} is given by 
\begin{equation}\label{gnc}
\frac{1}{2}(\partial_x\phi)^2=\frac{c^2}{1+n} + K(1+n) + e^{H(n)} - c^2-K-1=: g(n,c).
\end{equation}

The existence of the solitary wave solutions is shown in \cite{Cor} (see also \cite{BK}) through a phase plane analysis of the reduced system \eqref{ODE_n_E}. To be more precise, we define a small amplitude parameter $\veps>0$  by 
\begin{equation}\label{speed}
\veps:=c - \mathsf{V},
\end{equation}
where $\mathsf{V}:=\sqrt{1+K}$   is  the \emph{ion-sound speed}. We also define  
 \begin{equation}\label{ep0}
 \ve_K :=\sqrt{K} \zeta - \mathsf{V}>0,
 \end{equation}
where $\zeta>\sqrt{K+1}/\sqrt{K}$ is a root of the equation
$\zeta^{K}\left[ K (\zeta - 1)^2 + 1 \right]  - \exp \left(K \left(\zeta^2 - 1 \right) / 2\right)=0$. For each $\veps \in (0,\veps_K)$,\footnote{This condition is also the necessary condition for the existence of non-trivial smooth solutions. When $\veps=\veps_K$, the traveling wave solution is not differentiable at  the peak. When $\veps=0$, there is only a trivial solution. See \cite{Cor}.} 
the system of traveling wave equations  \eqref{EP2Travel2} admits a unique non-trivial smooth solution $(n,u,\phi)=(n_c,u_c,\phi_c)$ satisfying
\[
(n_c,u_c,\phi_c)(x)=(n_c,u_c,\phi_c)(-x) \text{ for } x\in\mathbb{R}, \text{ and } \quad \partial_xn_c,\partial_xu_c,\partial_x\phi_c > 0 \quad \text{ for } x\in(-\infty,0).
\]
Moreover, it satisfies the far-field condition \eqref{bdCon x2} and decays exponentially fast as $|x| \to \infty$. We also note that $n_c,u_c,\phi_c>0$ for $x\in\mathbb{R}$.

The peak value $n_c^\ast:=n_c(0)$ is a unique zero of $g$ defined in \eqref{gnc} on the interval $(0,n_s)$, where $n_s=c/\sqrt{K}-1$ is a unique positive zero of 
\begin{equation}\label{deriHnc}
\partial_n H(n,c) = \frac{1}{1+n}\left(\frac{c^2}{(1+n)^2} -K \right).
\end{equation}
Since $\partial_n H(n,c)$ is positive for $n\in(0,n_s)$, we have from \eqref{TravelEqB1} that for all $x\in\mathbb{R}$, 
\[
\frac{c^2}{(1+n_c)^2} - K =  (u_c-c)^2 - K > 0.
\]
 This positive function will be denoted by
\begin{equation}\label{Useful_Id3}
J=J(x,\veps):=(c-u_c(x))^2-K>0.
\end{equation}

One can check that $n_c(x)$ is differentiable in $c \in (\sqrt{1+K},\sqrt{K}\zeta)$, and $\partial_c n_c(x)$ decays to zero exponential as $|x| \to 0$, and thus so do $u_c$ and $\phi_c$. 
This will be discussed in Remark \ref{Rem11}.\\

For notational simplicity, we let
\begin{equation}\label{EigenFun4}
Y_c(x):=(n_c,u_c)^T(x), \quad Z_c(x):=(u_c,n_c)(x).
\end{equation}
Differentiating the traveling wave equations \eqref{EP2Travel2} in $x$ and $c$, respectively, one obtains that
\begin{subequations}\label{EigenFunc}
\begin{align}[left = \empheqlbrace\,]
& \mathcal{L}\partial_xY_c = 0,  \quad \mathcal{L}^\ast Z_c = 0, \label{EigenFunc1}\\
& \mathcal{L} \partial_cY_c = - \partial_x Y_c,
 \quad
 \mathcal{L}^\ast \left( \int_{-\infty}^x\partial_c Z_c(s)\,ds \right)
 = Z_c,\label{EigenFunc2}
\end{align}
\end{subequations}
where $\mathcal{L}^\ast$ is the adjoint operator of $\mathcal{L}$ with respect to the standard $L^2$ inner product. More precisely, this is given by
\begin{equation}\label{AdjL}
\mathcal{L}^{\ast} Z 
:=  (\partial_xZ) L  + \left( (-\partial_x^2+e^{\phi_c})^{-1}(\partial_x Z_2), 0 \right), \quad Z=(Z_1,Z_2),
\end{equation}
and $L$ is the matrix defined in \eqref{DefL}.
We note that  $Y_c$, $\partial_xY_c$, and $\partial_cY_c$ decay exponentially to zero as $|x| \to +\infty$. On the other hand, $\small\smallint^x \partial_c Z_c\,ds$ decays exponentially to zero as $x\to -\infty$, but  it is merely bounded as $x \to +\infty$.

Lastly, we briefly discuss the result of \cite{BK}
which will be  used throughout this paper. 
\begin{theorem}[\cite{BK}]\label{MainThm4}
For $K >0$, let $\mathsf{V}$ and $\veps$ be as in \eqref{speed}, and $\xi=\veps^{1/2}x$. Let $j$ be any non-negative integer. Then there exist positive constants $\veps_\ast$, $C_\ast$ and $C_{j}$ such that for all $\veps\in(0,\veps_\ast]$
\begin{equation}\label{pointesti2}
\left| \partial_\xi^j \left( \frac{n_c}{\veps},\frac{u_c}{\veps},\frac{\phi_c}{\veps} \right)(\veps^{-1/2}\xi) - \partial_\xi^j(\Psi_\textrm{KdV},\mathsf{V}\Psi_\textrm{KdV},\Psi_\textrm{KdV})(\xi) \right|  \leq \veps C_{j}e^{-C_\ast|\xi|},
\end{equation}
where $\Psi_\textrm{KdV}(\xi)$, defined in \eqref{solutionKdV2}, satisfies the KdV equation \eqref{KdV}.  Here, $C_\ast$ and $C_j$ are independent of $\veps\in(0,\veps_\ast]$ and $\xi\in\mathbb{R}$. Moreover, in the variable of $x$, there holds
\begin{equation}\label{pointestimateinX}
|\partial_x^j n_c(x)| + |\partial_x^j u_c(x)| + |\partial_x^j \phi_c(x)| \leq \veps^{1+ j/2}C_je^{-C_\ast\veps^{1/2}|x|}.
\end{equation}  
\end{theorem}

\begin{remark}\label{Rem11}
We remark that $\partial_c n_c(x)$ decays exponentially fast as $|x|\to \infty$. To see this, 
we let $E_c:=-\partial_x \phi_c$ and $h(n,c):=\partial_cH(n,c)$. Then, $(n_c,E_c)$ satisfies the following ODE system (recall \eqref{ODE_n_E})
\[
\partial_x \begin{pmatrix}
n \\
E
\end{pmatrix}
=
\begin{pmatrix}
-E/ [h(n,c)] \\
1+n - e^{H(n,c)}
\end{pmatrix}
=:F(n,E,c),
\quad
(n,E)=(n_c,E_c).
\]
Since the peak value of $n_c(x)$, denoted by $n_c^\ast$,  is differentiable in $c$, it is standard (\cite{Hartman}) that  $(n_c,E_c)(x)$ is differentiable in $c$, and it satisfies the initial value problem of the linear inhomogeneous system
\begin{equation}
\begin{split}
\partial_x
\begin{pmatrix}
\partial_cn_c \\
\partial_cE_c
\end{pmatrix}
=
& 
[DF]_{(n,E)=(n_c,E_c)} \begin{pmatrix}
\partial_c n_c \\
\partial_c E_c
\end{pmatrix}
+
\partial_cF(n,E,c)|_{(n,E)=(n_c,E_c)}, \\
&
\left(\partial_cn_c(0), \partial_cE_c(0) \right) = \left(\partial_cn_c^\ast , 0 \right),
\end{split}
\end{equation}
where $[DF]$ is the Jacobian matrix of $F$ in $(n,E)$ variable, i.e., 
\[
[DF](n,E,c) = 
\begin{pmatrix}
\frac{E\partial_nh}{h^2} & -\frac{1}{h} \\
1-e^H h& 0
\end{pmatrix}.
\]
The eigenvalues of the matrix 
\[
B(c):=DF(0,0,c) = 
\begin{pmatrix}
0 & \frac{1}{K-c^2}\\
1 + K -c^2 & 0
\end{pmatrix}
\]
are $\lambda_c := \sqrt{\frac{c^2-1-K}{c^2-K}}>0$ and $-\lambda_c$, and the associated right eigenvectors are $\boldsymbol{v}_\pm:= \left(1, \pm\lambda_c (K-c^2) \right)^T$, respectively. Therefore, there is a projection $P:=\frac{\boldsymbol{v}_-\boldsymbol{v}_-^T}{\boldsymbol{v}_-^T\boldsymbol{v}_-}$
such that 
\begin{equation}\label{eB}
|e^{B(x-s)}P| \leq Ce^{-\lambda_c(x-s)} \text{ for } x>s, \quad |e^{B(x-s)}(I-P)| \leq Ce^{-\lambda_c(s-x)} \text{ for } s>x.
\end{equation}
Since $(n_c,E_c)$ exponentially decays to zero as $|x| \to 0$,  so does $\mathbf{f}$ where $\mathbf{f}$ is defined by
\[
\mathbf{f}(x,c):= \left( [DF]_{(n,E)=(n_c,E_c)} -[DF]_{(n,E)=(0,0)} \right) \begin{pmatrix}
\partial_c n_c \\
\partial_c E_c
\end{pmatrix}
+
\partial_cF(n,E,c)_{(n,E)=(n_c,E_c)}.
\]
By the variation of constants, we have that for $x>0$,
\[
\begin{split}
\begin{pmatrix}
\partial_cn_c \\
\partial_cE_c
\end{pmatrix}
(x)
& = \underbrace{e^{Bx}\left( \begin{pmatrix}
\partial_cn_c^\ast \\
0
\end{pmatrix} + \int_0^\infty e^{-Bs}(I-P)\mathbf{f}(s)\,ds \right) }_{=:I_1}\\
&   + \underbrace{\int_0^xe^{B(x-s)}P\mathbf{f}(s)\,ds - \int_x^\infty e^{B(x-s)}(I-P)\mathbf{f}(s)\,ds}_{=:I_2}.
\end{split}
\]
Using \eqref{eB} and the fact that $\mathbf{f}$ exponentially decays to zero, we see that $I_2$ exponentially decays to zero as $x \to +\infty$. Moreover, $(\partial_cn_c,\partial_cE_c)(x)$ is bounded in $x$ since $(n_c,E_c)(x) \to (0,0)$ as $|x| \to \infty$ for all $c\in (\sqrt{1+K},\sqrt{K}\zeta)$. Therefore, $I_1$ must be bounded in $x>0$, and thus $I_1$ exponentially decays to zero as $x \to +\infty$.  Since $I_1$ and $I_2$ both exponentially decay to zero as $x \to +\infty$, so does 
$\partial_cn_c$. By the symmetry $\partial_cn_c(x)= \partial_cn_c(-x)$, $\partial_cn_c$ decays exponentially fast as $x\to -\infty$ as well. 
\end{remark}

\section{Spectral analysis of the linearized equations}\label{S2}
In this section, we study the spectrum of the linear operator $\mathcal{L}$ defined in \eqref{Ch3_Def_mcL} whose main results are summarized in subsection \ref{SubS_Spec}. To this end, throughout subsection \ref{SecDefEva}--\ref{SubS_EvEp}, we first define the Evans function for the associated eigenvalue problem of the operator $\mathcal{L}$ (this function will be  called as the Evans function for the Euler-Poisson system),
 and discuss its properties. The remaining subsections are devoted to studying the location and order of zeros of the associated Evans function (see Proposition \ref{Pro_Ev_Ch}). 

\subsection{Definition and Properties of the Evans Function}\label{SecDefEva}
Before we define the Evans function for the Euler-Poisson system, we briefly review the definition of Evans function and its properties in a general setting by mainly following \cite{PW}. 

We consider a first-order ODE system 
\begin{equation}\label{ODE_LinEP}
\frac{d\mathbf{y}}{dx} = A(x,\lambda,\veps)\mathbf{y}.
\end{equation}
In the next subsection, we will see that the eigenvalue problem \eqref{EigenEP} can be written  in the form of  \eqref{ODE_LinEP} with $\mathbf{y}=(\dt{n},\dt{u},\dt{\phi},\partial_x\dt{\phi})^T$ for which the coefficient matrix $A$ is defined in \eqref{ODE_LinEP1}.

For each $\veps>0$, the Evans function for \eqref{ODE_LinEP} is defined on a simply connected domain $\Omega^\veps\subset \mathbb{C}$,  where the following conditions hold true:
\begin{enumerate}
\item[\textbf{H1}] $A(x,\lambda,\veps)$ is continuous in $(x,\lambda)\in\mathbb{R}\times \Omega^\veps$ and is analytic in $\lambda \in \Omega^\veps$ for fixed $x \in \mathbb{R}$.
\item[\textbf{H2}] $A(x,\lambda,\veps)$ converges to $A^\infty(\lambda,\veps)$ as $|x| \to \infty$, uniformly in $\lambda$ on any compact subset of $\Omega^\veps$. 
\item[\textbf{H3}] The integral $\int_{-\infty}^\infty |A(x,\lambda,\veps) - A^\infty(\lambda,\veps) |\,dx$ converges for all $\lambda$, uniformly on any compact subset of $\Omega^\veps$.
\item[\textbf{H4}] For every $\lambda \in \Omega^\veps$,   the  matrix eigenvalues $\mu_j=\mu_j(\lambda,\veps)$ of $A^\infty=A^\infty(\lambda,\veps)$  can be labelled so that 
\begin{equation*}
\oRe \mu_1 < \mu_\ast:=\min\{\oRe \mu_j: j=2,3,4\}.
\end{equation*}
\end{enumerate} 
Under the assumption $\mathbf{H4}$, one can choose the right and left eigenvectors of $A^\infty(\lambda,\veps)$ associated with a simple eigenvalue $\mu_1$, denoted by $\mathbf{v}_1=\mathbf{v}_1(\lambda,\veps)$ and $\mathbf{w}_1=\mathbf{w}_1(\lambda,\veps)$ respectively, such that they are analytic in $\lambda \in \Omega^\veps$ and the normalization $\mathbf{w}_1\mathbf{v}_1=1$ holds (see Chapter 2, Section 4.1 of \cite{Kato}).
Under the assumptions \textbf{H1}--\textbf{H4},  \eqref{ODE_LinEP} has a unique solution $\mathbf{y}^+=\mathbf{y}^+(x,\lambda,\veps)$ satisfying 
\begin{equation}\label{ODEsol1}
\lim_{x \to +\infty}e^{-\mu_1 x}\mathbf{y}^+ = \mathbf{v}_1,
\end{equation}
and the transposed ODE system 
\begin{equation}\label{TranspoODE_DefEvans}
\frac{d\mathbf{z}}{dx} = -\mathbf{z} A(x,\lambda,\veps),
\end{equation}
where $\mathbf{z}$ is a row vector, has a unique solution $\mathbf{z}^-=\mathbf{z}^-(x,\lambda,\veps)$ satisfying 
\begin{equation}\label{TranspoODE_sol_DefEvans}
\lim_{x \to -\infty}e^{\mu_1 x}\mathbf{z}^-= \mathbf{w}_1.
\end{equation}
Here the solutions $\mathbf{y}^+$ and $\mathbf{z}^-$ can be constructed so that they are analytic in $\lambda \in \Omega^\veps$ for fixed $x\in\mathbb{R}$. 

Now the Evans function $D(\lambda,\veps)$ for \eqref{ODE_LinEP}  is defined by
\begin{equation}\label{Evans-Def}
D(\lambda,\veps):= \mathbf{z}^-\mathbf{y}^+(x,\lambda,\veps).
\end{equation}
Then $D(\lambda,\veps)$  is analytic in $\lambda$ and is independent of $x$. Moreover, it is characterized by
\begin{equation}\label{CharacEvansD}
\lim_{x \to -\infty}e^{-\mu_1 x}\mathbf{y}^+ = D(\lambda,\veps)\mathbf{v}_1.
\end{equation}

We will see that $\oRe \mu_1 < 0 < \mu_\ast$ holds on the domain $\oRe \lambda >0$. On this domain, \eqref{ODEsol1} and \eqref{CharacEvansD} imply that $D(\lambda,\veps)=0$ if and only if $\mathbf{y}^+(x,\lambda,\veps)$ is an $(L^2)^4$-solution of  \eqref{ODE_LinEP}. For the same reason, $D(\lambda,\veps)=0$ if and only if $\mathbf{y}^+(x,\lambda,\veps)$ is an $(L_\beta^2)^4$-solution to \eqref{ODE_LinEP}, provided that $\beta>0$ satisfies 
\begin{equation}\label{splitting weighted 2}
\oRe \mu_1 +{\etba} < 0 < \mu_\ast + \etba.
\end{equation}
The domain $\Omega^\veps$ will be defined (see \eqref{ND-eps}) in such a way that it contains the \emph{closed} right-half plane and \eqref{splitting weighted 2} holds for $\lambda \in \Omega^\veps$. We also remark that the zeros of $D(\lambda,\veps)$ are isolated since it is an analytic function in $\lambda$.

We summarize  the characterizations of asymptotic behaviors of solutions of \eqref{ODE_LinEP} and \eqref{TranspoODE_DefEvans}, which will be crucially used in the following analysis.
\begin{proposition}[\cite{PW}, Proposition 1.6]\label{Prop_ResultPW}
Suppose \textbf{H1}--\textbf{H4} hold on $\Omega^\veps$ for each $\veps$.  Let $\mathbf{y}^+(x,\lambda,\veps)$ and $\mathbf{z}^-(x,\lambda,\veps)$ be the solutions to \eqref{ODE_LinEP} $($resp.\,\eqref{TranspoODE_DefEvans}$)$ satisfying \eqref{ODEsol1} $($resp.\,\eqref{TranspoODE_sol_DefEvans}$)$. For $\lambda \in \Omega^\veps$,  the following statements hold true.
\begin{enumerate}[{(a)}]
\item \label{Dy+} The following are equivalent:
	\begin{enumerate}[{(i)}]
	\item $D(\lambda,\veps)=0$;
	\item $\mathbf{y}^+(x,\lambda,\veps)=O(e^{(\mu_\ast-\theta) x})$ as $x \to -\infty$ for  $0< \theta < \mu_\ast - \oRe \mu_1$.
	\end{enumerate}
\item \label{Dy+1} For any solution $\mathbf{y}(x,\lambda,\veps)$ of \eqref{ODE_LinEP}, the following are equivalent:
	\begin{enumerate}[{(i)}]
	\item $\mathbf{y}(x,\lambda,\veps) = O(e^{\mu_1 x})$ as $x \to +\infty$;
	\item $\mathbf{y}(x,\lambda,\veps) = \alpha \mathbf{y}^+(x,\lambda,\veps)$ for some constant $\alpha\in \mathbb{C}$;  
	\item $\mathbf{y}(x,\lambda,\veps) = o(e^{(\mu_\ast -\theta)x})$ as $x \to +\infty$ for  $0< \theta < \mu_\ast - \oRe \mu_1$.  
	\end{enumerate}
	
\item  \label{Dy+2} For any solution $\mathbf{z}(x,\lambda,\veps)$ of \eqref{TranspoODE_DefEvans}, the following are equivalent:
	\begin{enumerate}[{(i)}]
	\item $\mathbf{z}(x,\lambda,\veps) = \alpha \mathbf{z}^-(x,\lambda,\veps)$ for some constant $\alpha\in \mathbb{C}$
	\item $\mathbf{z}(x,\lambda,\veps) = o(e^{(-\mu_\ast +\theta)x})$ as $x \to -\infty$ for  $0< \theta < \mu_\ast - \oRe \mu_1$.
	\end{enumerate} 
\end{enumerate}
\end{proposition}

\begin{proposition}[\cite{PW}, Proposition 1.21]\label{Prop_DeYZ}
Under the same assumptions as in Proposition \ref{Prop_ResultPW}, the following statements hold. 
\begin{enumerate}[{(a)}]
\item For all nonnegative integers $j=0,1,2,\cdots$, there holds that 
\begin{equation}\label{Prop_DeYZ1}
\partial_\lambda^j \mathbf{y}^+(x,\lambda,\veps) = O(e^{(\mu_1+\theta) x})  \text{ as } x \to +\infty  \; \text{ and } \;
\partial_\lambda^j \mathbf{z}^- = O(e^{-(\mu_1 + \theta) x}) \text{ as } x \to -\infty
\end{equation}
for any small $\theta>0$. 
\item If $0=D(\lambda,\veps)=\cdots = \partial_\lambda^{k-1} D(\lambda,\veps) \neq \partial_\lambda^k D(\lambda,\veps)$, then we have that for $ j=0,1,\cdots,k-1$,
\begin{equation}\label{Prop_DeYZ2}
\partial_\lambda^j \mathbf{y}^+(x,\lambda,\veps) = O(e^{\mu_\ast x - \theta x})  \text{ as }  x \to -\infty 
\end{equation}
for any small $\theta>0$, and that
\begin{equation}\label{Prop_DeYZ3}
\lim_{x \to -\infty}e^{-\mu_1x}\partial_\lambda^k \mathbf{y}^+(x,\lambda,\veps) = \partial_\lambda^k D(\lambda,\veps) \mathbf{v}_1.
\end{equation}
\end{enumerate}
\end{proposition}

\subsection{Reformulation of the Eigenvalue Problem}\label{SS_RefEP}
We rewrite the eigenvalue problem \eqref{EigenEP} as a first-order  linear ODE system \eqref{ODE_LinEP}. The matrix $L$ defined in \eqref{DefL} is invertible 
since $\mathrm{det} L= J>0$ from \eqref{Useful_Id3}. Hence  from \eqref{Ch3_Def_mcL} and \eqref{EigenEP}, we have
\begin{equation}\label{Transpor_ODE}
\begin{split}
\begin{pmatrix}
0 \\
0
\end{pmatrix}
= L^{-1}\left(\lambda -\mathcal{L}\right) \begin{pmatrix}
\dt{n} \\
\dt{u}
\end{pmatrix} 
=
 \partial_x \begin{pmatrix}
\dt{n} \\
\dt{u}
\end{pmatrix} + L^{-1}\left[ \left(\lambda  + (\partial_x L)\right)\begin{pmatrix}
\dt{n} \\
\dt{u}
\end{pmatrix}  +  \begin{pmatrix}
0 \\
\partial_x \dt{\phi}
\end{pmatrix} \right].
\end{split}
\end{equation}
We rewrite the Poisson equation \eqref{EP3Poi}  as
\begin{equation}\label{Poisson_ODE}
\partial_x\dt{\phi} =: \dt{\psi} , \quad \partial_x\dt{\psi}  =  e^{\phi_c}\dt{\phi} - \dt{n}.
\end{equation}
By letting $\mathbf{y}:=(\dt{n},\dt{u},\dt{\phi},\dt{\psi})^T$, we obtain the system \eqref{ODE_LinEP} with the coefficient matrix
\begin{equation}\label{ODE_LinEP1}
A=A(x,\lambda,\veps):= 
\scriptsize\left(
\begin{array}{c|c}
L^{-1} & \begin{array}{cc}
						0 & 0 \\
						0 & 0 
					\end{array} \\
		\hline 
\begin{array}{cc}
0 & 0 \\
0 & 0
\end{array}
& 
\begin{array}{cc}
1 & 0 \\
0 & 1
\end{array}
\end{array}
\right)
\left(
\begin{array}{c|c}
-\lambda I_2 - \partial_xL & \begin{array}{cc}
						0 & 0 \\
						0 & -1 
					\end{array} \\
		\hline 
\begin{array}{cc}
0 & 0 \\
-1 & 0
\end{array}
& 
\begin{array}{cc}
0 & 1 \\
e^{\phi_c} & 0
\end{array}
\end{array}
\right).
\end{equation}
Here $A$ can be decomposed as $A(x,\lambda,\veps) =A_1(x,\veps) + \lambda A_2(x,\veps)$, where 
\begin{subequations}\label{A_Decompose}
\begin{align}
 A_1 & :=  \scriptsize\begin{pmatrix}
\frac{(c-u_c)\partial_xu_c}{J} - \frac{K\partial_x n_c}{J(1+n_c)} & \frac{(c-u_c)\partial_xn_c}{J} + \frac{(1+n_c)\partial_xu_c}{J} & 0 & \frac{1+n_c}{J} \\
\frac{K\partial_x u_c}{J(1+n_c)} -\frac{K(c-u_c)\partial_xn_c}{J(1+n_c)^2} & \frac{K\partial_xn_c}{J(1+n_c)} + \frac{(c-u_c)\partial_xu_c}{J} & 0 & \frac{c-u_c}{J} \\
0 & 0 & 0 & 1 \\
-1 & 0 & e^{\phi_c} & 0
\end{pmatrix}, 
\\ 
A_2& := 
\left( 
\begin{array}{c|c}
-L^{-1} & \mathbf{0}_2 \\
\hline
\mathbf{0}_2 & \mathbf{0}_2
\end{array}
\right) = \scriptsize\frac{1}{J}\left( 
\begin{array}{c|c}
\begin{array}{cc}
c-u_c & 1+n_c \\ 
\dfrac{K}{1+n_c} & c-u_c
\end{array} & \mathbf{0}_2 \\
\hline
\mathbf{0}_2 & \mathbf{0}_2
\end{array}
\right)
\end{align}
\end{subequations}  
and $\mathbf{0}_2$ is a $2\times 2$ zero matrix.

Now it is clear that for each $\beta \in [0,1)$, the mapping 
\[
(\dt{n},\dt{u})^T   \mapsto (\dt{n},\dt{u},\dt{\phi},\dt{\psi})^T, 
\]
where $(\dt{\phi},\dt{\psi})$ is uniquely determined  by the Poisson equation \eqref{Poisson_ODE} for given $\dt{n} \in L_\beta^2$, is an isomorphism between $\mathrm{ker}(\lambda -\mathcal{L}) \subset (L_\beta^2)^2$ and $\mathrm{ker} \big( d/dx - A(x,\lambda,\veps) \big) \subset (L_\beta^2)^4$.

For any solution $Z \in (L_{-\beta}^2)^2$, $\beta \in [0,1)$, to the eigenvalue problem for the adjoint operator 
\[
(\lambda-\mathcal{L}^\ast)Z = 0
\]
(see \eqref{AdjL}), it is straightforward to check that 
\[
Z \mapsto (ZL,\tilde{\phi},\tilde{\psi}),
\]
where $(\tilde{\phi},\tilde{\psi})$ is a unique solution to the equations
\begin{equation}\label{Adj1}
(-\partial_x^2 + e^{\phi_c})\tilde{\psi} = -\partial_x \left(Z  (0,1)^T\right), \quad  \tilde{\phi} := -\partial_x \tilde{\psi} 
+ Z (0,1)^T,
\end{equation}
is an isomorphism between $\mathrm{ker}(\lambda -\mathcal{L}^\ast) \subset (L_{-\beta}^2)^2$ and $\mathrm{ker} \big( d/dx + A(x,\lambda,\veps) \big) \subset (L_{-\beta}^2)^4$.

We denote the first two components of $\mathbf{y}^+$ and $\mathbf{z}^-$ by $Y^+$ and $Z^-$, respectively. Then, $Y^+$ and $Z^-$ satisfy
\begin{equation}\label{LargeAmpl1}
(\lambda - \mathcal{L})Y^+ = 0, \quad 
(\lambda - \mathcal{L}^*)\left(Z^-L^{-1} \right) = 0.
\end{equation}

\subsection{The Evans function for the Euler-Poisson system}\label{SubS_EvEp}
In this subsection, we specify the domain of the Evans function $D(\lambda,\veps)$ for the Euler-Poisson system, and show that  the zeros of $D(\lambda,\veps)$ coincide with the  eigenvalues of the linearized operator $\mathcal{L}$.

By \eqref{pointestimateinX}, the coefficient matrix $A(x,\lambda,\veps)$ in \eqref{ODE_LinEP1} converges to the asymptotic matrix
\begin{equation}\label{A_Asymptotic}
A^\infty(\lambda,\veps)
:=  
\scriptsize\begin{pmatrix} 
\dfrac{c\lambda}{c^2-K} & \dfrac{\lambda}{c^2-K} & 0 & \dfrac{1}{c^2-K} \\  
\dfrac{K \lambda}{c^2-K} & \dfrac{c\lambda}{c^2-K} & 0 & \dfrac{c}{c^2-K} \\
0 & 0 & 0 & 1 \\
-1 & 0 & 1 & 0
\end{pmatrix}
\end{equation}
exponentially fast as $|x| \to \infty$. The matrix eigenvalues $\mu$ of $A^\infty$ are the zeros of the characteristic polynomial
\begin{equation}\label{dispersEP}
\begin{split}
d(\mu) = d(\mu,\lambda,\veps)
& := \textnormal{det}\,\left(\mu I-A^\infty(\lambda,\veps) \right) \\
& =(c^2-K)^{-1}\left( (\mu^2-1)\left[(\lambda - c\mu)^2 - K \mu^2 \right] + \mu^2 \right).
\end{split}
\end{equation}
Since $d(\pm 1) \neq 0$ and $c^2-K \neq 0$ for all $\veps \geq 0$ and $\lambda\in\mathbb{C}$, $d(\mu)=0$ is equivalent to that $\mu$ satisfies one of the equations
\begin{equation}\label{Charact1}
d_{\pm}(\mu)=d_{\pm}(\mu,\veps):= \mu \left(c \pm \sqrt{\frac{1}{1-\mu^2}+ K}\, \right) = \lambda.
\end{equation}

For each $\veps\geq 0$ and $\lambda \in \mathbb{C}$, $d(\mu)$ has four zeros $\mu_j$ $(j=1,2,3,4)$ counted with multiplicities. For a non-zero simple eigenvalue $\mu_j$ of $A^\infty$, the corresponding  right and left eigenvectors, denoted by $\mathbf{v}_j$ and $\mathbf{w}_j$, satisfying the normalization $\mathbf{w}_j\mathbf{v}_j=1$ can be chosen as follows:
\begin{equation}\label{LRVec_A}
\mathbf{v}_j  := \left(1,\; \frac{c\mu_j - \lambda}{\mu_j} , \; \frac{1}{1-\mu_j^2}, \;\frac{\mu_j}{1-\mu_j^2}  \right)^T, \quad \mathbf{w}_j  := \frac{\boldsymbol{\pi}_j}{\boldsymbol{\pi}_j  \mathbf{v}_j },
\end{equation}
where 
\begin{subequations}\label{eigenvec_A}
\begin{align}[left = \empheqlbrace\,]
 \boldsymbol{\pi}_j & :=  \left(\left( \frac{c\lambda}{\mu_j} - (c^2-K)\right)(1-\mu_j^2), \; -\lambda\frac{1-\mu_j^2}{\mu_j}, \; 1, \;\mu_j \right), \label{eigenvec_A1} \\
 \boldsymbol{\pi}_j \mathbf{v}_j &  =  \frac{\lambda^2(1-\mu_j^2)}{\mu_j^2} - (c^2-K)(1-\mu_j^2) +\frac{1+\mu_j^2}{1-\mu_j^2}. \label{eigenvec_A2}
\end{align}
\end{subequations}
We note that $\boldsymbol{\pi}_j \mathbf{v}_j  \neq 0$ when $\mu_j$ is semi-simple.

The splitting properties of the matrix eigenvalues are summarized in the following two lemmas whose proofs will be given in Section \ref{Sec_Split}. 
\begin{lemma}\label{splitting mu lam} 
The zeros of $d(\mu)=d(\mu,\lambda,\veps)$  can be labelled so that the following  properties hold.
\begin{enumerate}
\item For $\veps>0$, 
\begin{subequations} \label{EigenSpliting1} 
\begin{align}
& \oRe\mu_1 < 0 = \oRe\mu_2 = \oRe\mu_3 < \oRe \mu_4, & \text{when} \;  \oRe\lambda =0, \label{EigenSpliting1_1} \\
& \oRe\mu_1 < 0 < \oRe\mu_j, \quad  (j=2,3,4), & \text{when} \; \oRe\lambda >0. \label{EigenSpliting1_2}
\end{align}
\end{subequations}
\item  For $\veps=0$,
\begin{subequations}\label{EigenSpliting2}
\begin{align} 
& \oRe\mu_1 < 0 = \oRe\mu_2 = \oRe\mu_3 < \oRe \mu_4, & \text{when} \; & \oRe\lambda =0 \; \text{and} \; \lambda \neq 0, \label{EigenSpliting2_1} \\
& \oRe\mu_1 < 0 < \oRe\mu_j, \quad  (j=2,3,4), & \text{when} \; & \oRe\lambda >0. \label{EigenSpliting2_2}
\end{align}
\end{subequations}
\end{enumerate}
\end{lemma}

For $0 < c_0 < \sqrt{2\mathsf{V}}$, we let ${\etba} := c_0\veps^{1/2}$, where $\veps>0$, and   define a domain 
\begin{equation}\label{ND-eps}
 \Omega^\veps:= \left\{\lambda : \oRe \lambda \geq -\veps^{3/2}\eta(c_0) \right\}, \text{ where } \eta(c_0):= \frac{c_0}{2}\left(1 - \frac{c_0^2}{2\sqrt{1+K}} \right).
\end{equation}
We note that $\Omega^\veps$ contains the closed right-half plane $\{\lambda:  \oRe \lambda \geq 0\}$ for any $\veps>0$ since $\eta(c_0)>0$ on $(0,\sqrt{2\mathsf{V}})$.

\begin{lemma}\label{splitting mu lam weight}
For $0 < c_0 < \sqrt{2\mathsf{V}}$, let ${\etba} = c_0\veps^{1/2}$, where $\veps>0$. Then there exists $\veps_K' < \veps_K$ such that for all $\veps \in(0, \veps_K']$, the following hold.
\begin{enumerate}
\item The two parametrized curves $\{d_\pm(ik-\etba): k\in\mathbb{R}\}$ (see \eqref{Charact1}) lie on $\mathbb{C}\setminus \Omega^\veps$. 
\item For all $\lambda \in \Omega^\veps$, 
 the zeros of $d(\mu)=d(\mu,\lambda,\veps)$ can be labeled such that they satisfy
\begin{equation}\label{split mu weight}
\oRe \mu_1 + \etba <0  < \oRe \mu_j+\etba, \quad (j=2,3,4).
\end{equation}
\end{enumerate}
\end{lemma}
 In each statement involving $\Omega^\veps$ and $\veps_K'$, the same setting as that of Lemma \ref{splitting mu lam weight} is assumed unless otherwise stated.

Now we show that the Evans function for the Euler-Poisson system is defined on the region $\Omega^\veps$.
\begin{proposition}\label{DefEvansFunctDomain}
Let $\veps_K$ be as in \eqref{ep0}. Then for each $\veps\in(0,\veps_K)$, the Evans function $D(\lambda,\veps)$ for the system \eqref{ODE_LinEP} associated with the eigenvalue problem \eqref{EigenEP} is defined on the domain $\oRe \lambda\geq 0$. Furthermore, for each $\veps\in (0,\veps_K']$,  $D(\lambda,\veps)$ is analytically extended to $\Omega^\veps$.
\end{proposition} 

\begin{proof}
We verify \textbf{H1}--\textbf{H4} in Section \ref{SecDefEva}. From the form of $A(x,\lambda,\veps)$ and that $A$ converges to $A^\infty$ exponentially fast as $|x| \to \infty$, it is easy to see that \textbf{H1}-\textbf{H3} hold on the domain $\mathbb{C}$ as long as the solitary wave $(n_c,u_c,\phi_c)$ exists, that is, for each $\veps \in (0,\veps_K)$. From \eqref{EigenSpliting1}, \textbf{H4} holds on $\oRe \geq 0$ for each $\veps >0$.  This proves the first assertion.  The second assertion  follows from Lemma \ref{splitting mu lam weight}.
\end{proof}

\begin{remark}\label{Rem_Dom}
 For $\veps=0$ and $\lambda=0$, we have $\mu_j=0$ for all $j=1,2,3,4$. Hence, when $\veps=0$, $D(\lambda)$ is not defined for $\lambda=0$ since \textbf{H4} is not satisfied. In this case, $D(\lambda)$ is defined on the domain $\Omega^0:=\{\lambda: \oRe\lambda \geq 0, \lambda \neq 0\}$, and we have that $D(\lambda)=1$ for $\lambda\in\Omega^0$ by  \eqref{CharacEvansD}  since the coefficient matrix of the system \eqref{ODE_LinEP} is independent of $x$ when $\veps=0$.
\end{remark}

We characterize the eigenvalues of $\mathcal{L}$ as the  zeros the Evans function $D(\lambda,\veps)$.
\begin{proposition}\label{Prop_EvaEquiUnW}
\begin{enumerate}[{(a)}]
\item For each $\veps \in (0,\veps_K)$, the following holds: for $\lambda \in \mathbb{C}$ with $\oRe \lambda>0$, the eigenvalue problem \eqref{EigenEP} has a non-trivial solution in $(L^2)^2$ if and only if $D(\lambda,\veps)=0$.
\item  For $0 < c_0 < \sqrt{2\mathsf{V}}$, let $\etba = c_0\veps^{1/2}$, where $\veps>0$. Then there exists $\veps_K'> 0$ such that for each $ \veps \in(0,\veps_K']$, the following holds: for $\lambda \in \Omega^\veps$, the eigenvalue problem \eqref{EigenEP} has a non-trivial solution in $(L_\etba^2)^2$ if and only if $D(\lambda,\veps)=0$.
\end{enumerate}

\end{proposition}

\begin{proof}
We first prove the first assertion. If $D(\lambda,\veps)=0$ for $\lambda$ with $\oRe \lambda >0$, then $\mathbf{y}^+(x,\lambda,\veps)$ satisfying \eqref{ODEsol1} is a non-trivial $(L^2)^4$-solution to the system \eqref{ODE_LinEP} by \eqref{EigenSpliting1_2} and Proposition \ref{Prop_ResultPW}(\ref{Dy+}). Hence, the first two components of $\mathbf{y}^+$ is a $(L^2)^2$-solution to the eigenvalue problem \eqref{EigenEP}. Conversely, if \eqref{EigenEP} has a non-trivial $(L^2)^2$-solution for $\lambda$ with $\oRe \lambda >0$, then \eqref{ODE_LinEP} has a non-trivial solution $\mathbf{y}(x,\lambda,\veps)$, which is  bounded in $x\in\mathbb{R}$  since $\mathbf{y} \in H^1$. From \eqref{EigenSpliting1_2} and Proposition \ref{Prop_ResultPW}(\ref{Dy+1}), we see that $\mathbf{y}=\alpha \mathbf{y}^+$ for some constant $\alpha \neq 0$. Since $\mathbf{y}^+$ is bounded in $x$, we have that $\mathbf{y}^+=o(e^{\mu_1 x})$ as $x \to -\infty$, equivalently, $D(\lambda,\veps)=0$ by \eqref{CharacEvansD}.

We show the second assertion. Suppose that $D(\lambda,\veps)=0$ for $\lambda \in \Omega^\veps$. By Proposition \ref{Prop_ResultPW}(\ref{Dy+}), $e^{\etba x}\mathbf{y}^+ = O(e^{\etba x}e^{(\mu_\ast - \theta) x})$ as $x\to -\infty$, and thus  $e^{\etba x}\mathbf{y}^+ $ exponentially decays to zero as $x \to \pm\infty$ by \eqref{split mu weight}. Hence, the first two components of $\mathbf{y}^+$ is a solution to the eigenvalue problem \eqref{EigenEP} in $(L_\etba^2)^2$.  Conversely, suppose that \eqref{EigenEP} has a nontrivial $(L_\etba^2)^2$ solution for $\lambda \in \Omega^\veps$. Then \eqref{ODE_LinEP} has a solution $\mathbf{y}$ such that $e^{\etba x}\mathbf{y}\in (L^2)^4$. Since $e^{\etba x}\mathbf{y} \in (H^1)^4$, $e^{\etba x}\mathbf{y}$ is bounded in $x$, and thus we have  
\begin{equation}\label{Ch3_DecayEq1}
\mathbf{y} = O(e^{-\etba x}) \quad \text{as} \quad |x| \to \infty.
\end{equation}
Multiplying \eqref{Ch3_DecayEq1} by $e^{-(\mu_\ast-\theta) x }$ for sufficiently small $\theta>0$, we see that  $\mathbf{y} = o(e^{(\mu_\ast  -\theta) x})$ as $x \to +\infty$ by \eqref{split mu weight}.
By Proposition \ref{Prop_ResultPW}(\ref{Dy+1}) and \eqref{Ch3_DecayEq1}, this implies that  $\mathbf{y}^+ = O(e^{-\etba x}) \quad \text{as} \quad x \to -\infty$, which yields that $\mathbf{y}^+ = o(e^{\mu_1 x})$ as $x \to -\infty$ by \eqref{split mu weight}, equivalently, $D(\lambda,\veps)=0$.
\end{proof}

\begin{remark}\label{Rem_Sym}
If $\lambda$ is an $(L^2)^2$-eigenvalue of $\mathcal{L}$, then so are $\overline{\lambda}$ and $-\lambda$ by the symmetry of the solitary wave, $(n_c,u_c,\phi_c)(x)=(n_c,u_c,\phi_c)(-x)$. Indeed, if $\mathbf{y}(x,\lambda)$ is a solution to  \eqref{ODE_LinEP}, then we have
\begin{equation}\label{EigenSymmetry}
\frac{d}{dx}\overline{\mathbf{y}(x,\lambda)} =( A_1 + \overline{\lambda}A_2 ) \overline{\mathbf{y}(x,\lambda)}, \quad \frac{d}{dx}\widetilde{\mathbf{y}} =( A_1 - \lambda A_2 ) \widetilde{\mathbf{y}},
\end{equation}
where $\widetilde{\mathbf{y}}:=\left(y_1(-x,\lambda),y_2(-x,\lambda),y_3(-x,\lambda),-y_4(-x,\lambda)\right)^T$.  We remark that on the domain $\oRe \lambda \leq  0$, where $\oRe\mu_1 < \mu_\ast \leq 0$ holds, the zeros of the Evans function is not related to the $(L^2)^2$- eigenvalues, in principle. For $\oRe\lambda=0$, for instance, we have $\oRe \mu_1 < 0=\mu_\ast$, and hence $\mathbf{y}^+$ (an analytic continuation of $\mathbf{y}^+$ defined on the domain $\oRe\lambda >0$) may oscillate without decaying as $x \to -\infty$. The eigenfunction $\widetilde{\mathbf{y}}$ corresponding to the eigenvalue $- \lambda$ is not an analytic continuation of $\mathbf{y}^+$. The zeros of the Evans function on $\oRe \lambda \leq  0$  correspond to the so-called resonance poles (\cite{PW},\cite{PW2},\cite{Kapi}).
\end{remark}

We have the following relation between the algebraic multiplicity of eigenvalues of $\mathcal{L}$ and  the order of zeros of $D(\lambda,\veps)$.
\begin{proposition}\label{Ord2}
For each $\veps \in (0,\veps_K')$, the following holds: for $\lambda \in \Omega^\veps$ with $D(\lambda,\veps)=0$, the order of $\lambda$ as a zero of $D(\lambda,\veps)$ coincides with the algebraic multiplicity of $\lambda$ as an $(L_\beta^2)^2$-eigenvalue of $\mathcal{L}$.
\end{proposition}
\begin{proof}
 By taking derivatives of \eqref{LargeAmpl1} in $\lambda$, we have  
\begin{equation}\label{lambdaDeri_LS}
(\lambda-\mathcal{L})Y^+=0,  \quad (\lambda-\mathcal{L})\partial_\lambda^{j+1}Y^+=-(j+1)\partial_\lambda^jY^+ \quad \text{for } j=0,1,2,\cdots.
\end{equation}
Suppose that $D(\lambda)=\cdots = \partial_\lambda^{k-1} D(\lambda)=0$ and $\partial_\lambda^k D(\lambda)\neq 0$.  From \eqref{Prop_DeYZ1}--\eqref{Prop_DeYZ2} and \eqref{split mu weight}, we have that $e^{{\etba} x}\partial_\lambda^j Y^+ \in (H^1)^2$ for $j=0,1,\cdots,k-1$, that is, the algebraic multiplicity of $\lambda$ is at least $k$. We show the following assertions:
\begin{enumerate}[{(i)}]
\item \label{geo} (the geometric multiplicity is one) every non-trivial $(L_\etba^2)^2$-solution to the eigenvalue problem \eqref{EigenEP} is a constant multiple of $Y^+(x,\lambda)$,
\item \label{alge} (the algebraic multiplicity is $k$) there is no $(L_\etba^2)^2$-function $\widetilde{Y}(x,\lambda)$ satisfying
\begin{equation}\label{lambdaDeri1_LS}
(\lambda-\mathcal{L})\widetilde{Y}=-k\partial_\lambda^{k-1} Y^+.
\end{equation}
\end{enumerate} 
We already showed the assertion \eqref{geo} in the proof of Proposition \ref{Prop_EvaEquiUnW}. To check \eqref{alge}, we suppose that there is a function $\widetilde{Y} \in (H_\etba^1)^2 $ satisfying \eqref{lambdaDeri1_LS}. Then $e^{{\etba} x}\widetilde{Y}$ is bounded in $x \in \mathbb{R}$, and thus we have
\[
\widetilde{Y} = O(e^{-\beta x}) \text{ as } x \to \pm \infty.
\]
From \eqref{lambdaDeri_LS} and \eqref{lambdaDeri1_LS}, $Y_0(x,\lambda) := \widetilde{Y} - \partial_\lambda^k Y^+$ satisfies
\[
(\lambda-\mathcal{L})Y_0=(0,0)^T, \quad  Y_0 = O(e^{-{\etba} x}) \text{ as } x \to +\infty
\] 
since  $e^{{\etba} x}\partial_\lambda^kY^+$ is bounded in $x >0$ from \eqref{Prop_DeYZ1}.   Hence, from \eqref{split mu weight}, $Y_0 = o(e^{(\mu_\ast - \theta)x})$ as $x \to +\infty$ for sufficiently small $\theta>0$, and this implies that $Y_0=\alpha_0 Y^+$ for  some constant $\alpha_0\neq 0$  by Proposition \ref{Prop_ResultPW}(\ref{Dy+1}). Now we obtain that 
\[
\partial_\lambda^k Y^+ = \widetilde{Y} - \alpha_0Y^+ = O(e^{-{\etba} x}) \quad \text{as} \; x \to -\infty.
\]
This is a contradiction since $\oRe \mu_1<-\etba$, but  $\lim_{x \to -\infty}e^{-\mu_1 x}\partial_\lambda^k Y^+(x,\lambda) \neq 0$ by \eqref{Prop_DeYZ3}.

This result also implies that  if the algebraic multiplicity of $\lambda$ is $k$, then the order of $\lambda$ as a zero of $D(\lambda)$ must be $k$.
\end{proof}

We now present the characterization of the zeros of the Evans function for the Euler-Poisson system.
\begin{proposition}\label{Pro_Ev_Ch}
There exists $\veps_0<\veps_K'$ such that for $\veps \in (0,\veps_0]$, $\lambda=0$ is the only zero of $D(\lambda,\veps)$ on the region $\Omega^\veps$. Moreover, the order of $\lambda=0$ is exactly two.
\end{proposition}
Proposition \ref{Pro_Ev_Ch} directly follows from Proposition \ref{RelationDandDtilde} and Corollary \ref{cor_orderEvans}. The proofs of these results will be given throughout subsection 3.5--3.9.

\subsection{Spectrum of the linearized operator}\label{SubS_Spec}
Using the results in the previous subsections, we classify the spectrum of the  operator $\mathcal{L}$ on $(L^2)^2$ and $(L_\beta^2)^2$. The proof of the following Lemma is given in Appendix \ref{Lem39}.
\begin{lemma}\label{Fredequiv}
For each $\beta \in [0,1)$ and $\veps \in (0,\veps_K)$, consider the operators  $\lambda -\mathcal{L}: (H_\etba^1)^2\subset (L_\etba^2)^2 \to (L_\etba^2)^2$ and $d/dx-A(x,\lambda,\veps): (H_\etba^1)^4\subset (L_\etba^2)^4 \to  (L_\etba^2)^4$, where $\mathcal{L}$ and $A$ are defined in \eqref{Ch3_Def_mcL} and \eqref{ODE_LinEP1} respectively. For any non-negative interger $k$,  $\lambda -\mathcal{L}$ is Fredholm with index $k$ if and only if $d/dx-A(x,\lambda,\veps)$ is Fredholm with index $k$.
\end{lemma}

\begin{proposition}\label{Clas_Spec}
There exists $\veps_0>0$  such that the following holds.
\begin{enumerate}[(a)]
\item  For each $\veps \in (0,\veps_0]$, consider the operator $\mathcal{L}: (H^1)^2\subset (L^2)^2 \to (L^2)^2$. Then, $\sigma_{\mathrm{ess}}(\mathcal{L})$ is the imaginary axis and the resolvent set of $\mathcal{L}$ is $\mathbb{C}\setminus \sigma_{\mathrm{ess}}(\mathcal{L})$.
\item  For each $\veps \in (0,\veps_0]$, consider the operator $\mathcal{L}: (H_\beta^1)^2\subset (L_\beta^2)^2 \to (L_\beta^2)^2$. Then, 
	\begin{enumerate}[{(i)}]
	\item $\sigma_{\mathrm{ess}}(\mathcal{L})$ is the union of two parametrized curves $\{d_\pm(ik-\beta):k\in\mathbb{R}\}$, which lie in $\mathbb{C}\setminus \Omega^\veps$, 
	\item $\sigma_{\mathrm{pt}}(\mathcal{L}) \cap \Omega^\veps = \{0\}$, in particular, $\lambda=0$ is an isolated $(L_\beta^2)^2$-eigenvalue with the algebraic multiplicity two,
	\item the set $\Omega^\veps\setminus\{0\}$ is a subset of the resolvent set of $\mathcal{L}$.
\end{enumerate}
\end{enumerate}

\end{proposition}
\begin{proof}[Proof of Proposition \ref{Clas_Spec}]
For notational simplicity, we let $\mathcal{A}(\lambda)=d/dx - A(x,\lambda,\veps)$. From the characterization of the Fredholm properties of the operator $\mathcal{A}(\lambda)$ in terms of the existence of exponential dichotomies for the ODE \eqref{ODE_LinEP} (\cite{Palmer, Palmer2}) and the roughness of exponential dichotomies (\cite{Coppel2}), the Fredholm properties of $\mc{A}(\lambda)$ are determined by the hyperbolicity of the asymptotic matrices $\lim_{x\to\pm\infty}A(x,\lambda,\veps)$ (see also \cite{Sand}). First of all, either $\mathcal{A}(\lambda)$ is not Fredholm or $\mathcal{A}(\lambda)$ is Fredholm with index zero since the asymptotic matrices are equal to $A^\infty(\lambda)$. Hence, by Lemma \ref{Fredequiv}, the essential spectrum of $\mathcal{L}$ consists of $\lambda$ for which $\lambda-\mathcal{L}$ is not Fredholm.

Now we prove the first assertion. By the results of \cite{Palmer,Palmer2,Coppel2} mentioned above,  $\lambda-\mathcal{L}$ is not Fredholm on $(L^2)^2$ if and only if $\mathrm{det}(\mu- A^\infty(\lambda)) = 0$ for some $\mu \in i\mathbb{R}$, which is precisely the case that $\lambda$ lies in the imaginary axis by \eqref{EigenSpliting1_1}. Since the Evans function $D(\lambda,\veps)$ does not vanish on $\oRe \lambda >0$ by Proposition \ref{Pro_Ev_Ch}, the kernel of $\lambda-\mathcal{L}$ is trivial on $\oRe \lambda >0$ by Proposition \ref{Prop_EvaEquiUnW}. By the symmetry of the solitary waves, the kernel of $\lambda-\mathcal{L}$ is also trivial on $\oRe \lambda<0$ (see Remark \ref{Rem_Sym}). Since $\lambda-\mathcal{L}$ is Fredholm with index zero on the regions $\oRe \lambda<0$ and $\oRe \lambda>0$, the union of these regions are precisely the resolvent set.    This finishes the proof of the first assertion.

To prove the second assertion, we observe that studying the spectrum of the operator $\mc{A}(\lambda)$ on $(L_\etba^2)^4$ is equivalent to studying the spectrum of the operator 
\[
e^{\beta x}\mc{A}(\lambda) e^{-\beta x} := \frac{d}{dx} - \beta - A(x,\lambda,\veps)
\]
on $(L^2)^4$. Hence, $\lambda-\mathcal{L}$ is not Fredholm on $(L_\etba^2)^2$ if and only if $\mathrm{det}(\mu-\beta - A^\infty(\lambda)) = 0$ for some $\mu \in i\mathbb{R}$.  The set of such $\lambda$ is precisely the union of two parametrized curves $\{d_\pm(ik-\beta): k \in \mathbb{R}\}\subset \mathbb{C}\setminus \Omega^\veps$ by Lemma \ref{splitting mu lam weight}. Combining Proposition \ref{Prop_EvaEquiUnW} and Proposition \ref{Pro_Ev_Ch}, we see that  $\lambda-\mathcal{L}$ has a non-trivial $(L_\beta^2)^2$-solution only at $\lambda=0$ on the region $\Omega^\veps$. Moreover, the algebraic multiplicity of $\lambda=0$ is two by Proposition \ref{Ord2} and Proposition \ref{Pro_Ev_Ch}. Since $\lambda-\mathcal{L}$ is Fredholm with index zero on the region $\Omega^\veps$, the proof of the second assertion is completed.
\end{proof}

\subsection{Derivatives of $D(\lambda)$ at $\lambda=0$}
 We will show that $D(\lambda,\veps)|_{\lambda=0}=\partial_\lambda D(\lambda,\veps)|_{\lambda=0}=0$ for $\veps\in(0,\veps_K)$. In other words,  $\lambda=0$ is a zero of the Evans function for the Euler-Poisson system with the order at least two. For  later purposes, we also calculate $\partial_\lambda^2 D(\lambda,\veps)|_{\lambda=0}$. We remark that these results are valid as long as the non-trivial smooth solitary wave $(n_c,u_c,\phi_c)$ exists. For notational simplicity, we suppress the $\veps$-dependence in this subsection.
 
When $D(\lambda_0)=0$, we have a simple derivative formula of the Evans function (Theorem 1.11,  \cite{PW}): 
\begin{equation}\label{EvansDerivativeFormu}
\begin{split}
\partial_\lambda D(\lambda)|_{\lambda=\lambda_0}
& = -\int_{-\infty}^\infty  (\mathbf{z}^-\partial_\lambda A \mathbf{y}^+)(x,\lambda) |_{\lambda=\lambda_0}\,dx
\end{split}
\end{equation}
in the sense of an improper integral. Furthermore, when $D(\lambda_0) = \partial_\lambda D(\lambda)|_{\lambda=\lambda_0} = 0$, we have
\begin{equation}\label{EvansDerivativeFormu2}
\partial_\lambda^2 D(\lambda)|_{\lambda=\lambda_0}
= -\int_{-\infty}^\infty \left( \partial_\lambda\mathbf{z}^-\partial_\lambda A \mathbf{y}^+ + \mathbf{z}^-\partial_\lambda^2 A \mathbf{y}^+ + \mathbf{z}^-\partial_\lambda A \partial_\lambda\mathbf{y}^+ \right)|_{\lambda=\lambda_0}\,dx.
\end{equation}

\begin{lemma}\label{EvansOrder}
For each $\veps\in(0,\veps_K)$, there hold
\begin{enumerate}[{(a)}]
\item\label{EvansOrder1} There exists a constant $\alpha_1< 0$ such that $\mathbf{y}^+(x,\lambda,\veps)|_{\lambda=0}=\alpha_1\mathbf{y}_c(x)$, where
\begin{equation*}
\mathbf{y}_c(x):=(\partial_x n_c, \partial_x u_c, \partial_x \phi_c, \partial_x^2\phi_c)^T.
\end{equation*}
\item\label{EvansOrder2} There exists a constant $\alpha_2 > 0$ such that $\mathbf{z}^-(x,\lambda,\veps)|_{\lambda=0}=\alpha_2\mathbf{z}_c(x)$, where
\begin{equation*}
\mathbf{z}_c(x) := \big((u_c,n_c)L,\partial_x^2\phi_c + n_c,-\partial_x\phi_c\big).
\end{equation*}
\end{enumerate}
\end{lemma}
\begin{proof}
Recalling the reformulation of the eigenvalue problems to the associated ODE systems (see subsection \ref{SS_RefEP}) and that $(\partial_x n_c,\partial_x u_c)^T$ and $(u_c,n_c)$ are the eigenfunctions of $\mathcal{L}$ and $\mathcal{L}^\ast$ associated with $\lambda=0$, respectively (see \eqref{EigenFunc1}), it is easy to check that $\mathbf{y}_c$ and $\mathbf{z}_c$  satisfy the system \eqref{ODE_LinEP} and the transposed system \eqref{TranspoODE_DefEvans} with $\lambda=0$, respectively.  

 Since $\mathbf{y}_c$ and $\mathbf{z}_c$ exponentially decay to zero as $|x| \to \infty$, we have that for sufficiently small $\theta>0$,
\[
\mathbf{y}_c(x) = o(e^{-\theta x}) \textrm{ as } x \to +\infty \quad \textrm{and} \quad  \mathbf{z}_c(x) = o(e^{\theta x})\textrm{ as } x \to -\infty.
\]
By Proposition \ref{Prop_ResultPW}.\eqref{Dy+1} and \eqref{Dy+2}, we have
\begin{equation}\label{ProOrd3}
\mathbf{y}^+|_{\lambda=0}= \alpha_1\mathbf{y}_c \quad \textrm{and} \quad \mathbf{z}^-|_{\lambda=0}=\alpha_2\mathbf{z}_c
\end{equation}
for some constants $\alpha_1, \alpha_2 \neq 0$ since we have $\oRe \mu_1 < 0=\mu_\ast$ when $\lambda=0$ by \eqref{EigenSpliting1_1}.

To complete the proof, it remains to check the signs of $\alpha_1$ and $\alpha_2$. From the first equality of \eqref{ProOrd3}, we have that at $\lambda=0$,
\[
\lim_{x \to +\infty}e^{-\oRe \mu_1 x}\alpha_1\partial_x n_c = 1,
\]
where we have used \eqref{ODEsol1} and \eqref{LRVec_A}, and the first component of $\mathbf{y}_c$ is considered. This implies that $\alpha_1 <0$ since $\partial_x n_c(x) < 0$ for $x>0$. In a similar fashion, from the second equality of \eqref{ProOrd3},  we obtain  that at $\lambda=0$,
\[
\begin{split}
\lim_{x \to -\infty}e^{\oRe \mu_1 x}\alpha_2(\partial_x^2\phi_c + n_c) 
& = \left.\frac{1}{\boldsymbol{\pi}_1\mathbf{v}_1}\right|_{\lambda=0}  = c^2-1-K + \mu_1^2(c^2-K) > 0,
\end{split}
\]
where \eqref{TranspoODE_sol_DefEvans}, \eqref{LRVec_A} and \eqref{eigenvec_A} are used, and  the third component of $\mathbf{z}_c$ is considered in the first equality.  It is straightforward to check the second equality. This together with the fact that $\partial_x^2\phi_c(x) + n_c(x) >0$ for all $x<0$ with  sufficiently large $|x|$ implies that $\alpha_2>0$. We are done.
\end{proof}

Lemma \ref{EvansOrder} together with the derivative formulas \eqref{EvansDerivativeFormu} and \eqref{EvansDerivativeFormu2} yields the following proposition.
\begin{proposition}\label{EvansZero}
For each $\veps\in(0,\veps_K)$, there hold
\begin{enumerate}[{(a)}]
\item\label{EvansZero1st} $ D(\lambda,\veps)|_{\lambda=0} = \partial_\lambda D(\lambda,\veps)|_{\lambda=0} = 0,$
\item \label{EvansZero2nd}$\mathrm{sign}\,\partial_\lambda^2 D(\lambda,\veps)|_{\lambda=0} = \mathrm{sign}\, \partial_c \int_{-\infty}^\infty (n_cu_c)(x)\,dx.$
\end{enumerate}
\end{proposition}

\begin{proof} 
We first show the first assertion. From Lemma \ref{EvansOrder}.\eqref{EvansOrder1}, we see that $\mathbf{y}^+|_{\lambda=0}$ decays to zero as $x \to -\infty$. Hence it follows that $D(\lambda)|_{\lambda=0}=0$ from  the characterization of the Evans function \eqref{CharacEvansD} and that $\oRe \mu_1 < 0$ at $\lambda=0$. 

We let $Z^-$ and $Y^+$ be the first two components of $\mathbf{z}^-$ and $\mathbf{y}^+$, respectively.  From the form of the coefficient matrix $A$ (see \eqref{ODE_LinEP1}), the derivative formula \eqref{EvansDerivativeFormu} is reduced to 
 \[
\partial_\lambda D(\lambda)|_{\lambda=0} = \int_{-\infty}^\infty (Z^- L^{-1} Y^+)|_{\lambda=0}\,dx.
 \]
 Using Lemma \ref{EvansOrder} and that $n_cu_c$ tends to $0$ as $|x| \to \infty$, we obtain that 
\[
\begin{split}
\partial_\lambda D(\lambda)|_{\lambda=0}
&  = \alpha_1\alpha_2\int_{-\infty}^\infty \big((u_c,n_c)L\big)L^{-1}
(\partial_xn_c, \partial_xu_c)^T\,dx \\
& = \alpha_1\alpha_2\int_{-\infty}^\infty \partial_x(n_cu_c)\,dx = 0.
\end{split}
\]

Now we prove the second assertion. We note that the derivative formula \eqref{EvansDerivativeFormu2} is reduced to 
\begin{equation}\label{EvSecDeri}
\begin{split}
\partial_\lambda^2D(\lambda)|_{\lambda=0} 
& = \int_{-\infty}^\infty \left(\partial_\lambda Z^-L^{-1}Y^+  + Z^-L^{-1}\partial_\lambda
Y^+\right)|_{\lambda=0}\,dx.
\end{split}
\end{equation}
By taking $\partial_\lambda$ of \eqref{LargeAmpl1},   we obtain that at $\lambda=0$,
\begin{subequations}\label{EigenFun3}
\begin{align}[left = \empheqlbrace\,]
& \mathcal{L}\partial_\lambda Y^+
= Y^+
= \alpha_1 \partial_xY_c
= -\alpha_1\mathcal{L}\partial_c Y_c, \\
& \mathcal{L}^*\left(\partial_\lambda Z^-L^{-1} \right) = Z^-L^{-1}= \alpha_2Z_c
=
\alpha_2
\mathcal{L}^\ast \left( \int_{-\infty}^x \partial_c Z_c\,ds \right),
\end{align}
\end{subequations}  
where we have used Lemma \ref{EvansOrder} in the second equality, and \eqref{EigenFunc2} in the last equality (see \eqref{EigenFun4} for the notations $Y_c$ and $Z_c$). Since $\mathcal{L}$ and $\mathcal{L}^\ast$ are linear, we have from \eqref{EigenFun3} that at $\lambda=0$,
\begin{subequations}\label{EigenFun31}
\begin{align}[left = \empheqlbrace\,]
& \mathcal{L}(\partial_\lambda Y^+ + \alpha_1 \partial_c Y_c) = 0, \\
& \mathcal{L}^*\left(\partial_\lambda Z^-L^{-1} - \alpha_2 \int_{-\infty}^x \partial_c Z_c\,ds  \right) = 0.
\end{align}
\end{subequations}

From \eqref{Prop_DeYZ1},  $\partial_\lambda Y^+|_{\lambda=0}$ and $\partial_\lambda Z^-|_{\lambda=0}$ exponentially decay to zero as $x\to+\infty$ and $x \to -\infty$, respectively. We also recall that $\partial_cY_c$ and $\textstyle{\int^x \partial_cZ_c\,ds}$ exponentially decay to zero as $x\to+\infty$ and $x \to -\infty$, respectively.  Hence, \eqref{EigenFun31} and Proposition \ref{Prop_ResultPW} imply that for $\lambda=0$,
\begin{subequations}
\begin{align}[left = \empheqlbrace\,]
& \partial_\lambda Y^+  + \alpha_1 \partial_cY_c = \alpha_1' Y^+
 = 
\alpha_1\alpha_1' \partial_x Y_c, \label{EvansOrder3}\\
& \partial_\lambda Z^- - \alpha_2 \left(\int_{-\infty}^x\partial_c Z_c \,ds\right) \,L  = \alpha_2' Z^- = \alpha_2\alpha_2' Z_c L\label{EvansOrder4}
\end{align}
\end{subequations}  
for some constants $\alpha_1',\alpha_2'\neq 0$, where we have used Lemma \ref{EvansOrder} in the second equality.

Now applying Lemma \ref{EvansOrder}.\eqref{EvansOrder2}, \eqref{EvansOrder3} and by integration by parts, we have that for  $\lambda=0$,
\begin{equation}\label{EvansOrder5}
\begin{split}
\int_{-\infty}^\infty Z^-L^{-1}\partial_\lambda Y^+ \,dx
& = \alpha_2\int_{-\infty}^\infty Z_c \partial_\lambda Y^+ \,dx \\
& = \alpha_1\alpha_2\int_{-\infty}^\infty Z_c  \left( - \partial_c Y_c +\alpha_1' \partial_x Y_c,\right) \,dx \\
& = -\alpha_1\alpha_2\int_{-\infty}^\infty \partial_c(n_cu_c)\,dx.
\end{split}
\end{equation}
Applying Lemma \ref{EvansOrder}.\eqref{EvansOrder1} and \eqref{EvansOrder4} in a similar fashion, we obtain that for $\lambda=0$,
\begin{equation}\label{EvansOrder6}
\begin{split}
\int_{-\infty}^\infty \partial_\lambda Z^- L^{-1} Y^+  \,dx 
& = \alpha_1\int_{-\infty}^\infty \partial_\lambda Z^- L^{-1}\partial_x Y_c \,dx \\
& = \alpha_1\alpha_2 \int_{-\infty}^\infty \left( \int_{-\infty}^x \partial_c Z_c(s)\,ds 
-\alpha_2' Z_c \right) \partial_x Y_c \,dx \\
& = -\alpha_1\alpha_2 \int_{-\infty}^\infty \partial_c (n_cu_c)\,dx.
\end{split}
\end{equation}
Combining \eqref{EvansOrder5}, \eqref{EvansOrder6}, and \eqref{EvSecDeri}, we finally arrive at
\[
\partial_\lambda^2D(\lambda)|_{\lambda=0} = -2\alpha_1\alpha_2\partial_c\int_{-\infty}^\infty n_cu_c\,dx.
\]
Recalling that $\alpha_1<0$ and $\alpha_2>0$ (see Lemma \ref{EvansOrder}), we finish the proof of the second assertion.
\end{proof}

\subsection{The Evans Function for the KdV Equation}
In the KdV scailng, we formally obtain (by letting  $\dot{n}_\ast =p_2$ in  \eqref{FormDerivEigEPtoKdV4}) the eigenvalue problem of the KdV equation
\begin{equation}\label{KdVEigen}
\Lambda p_2 - \partial_\xi p_2 + \mathsf{V}\partial_\xi(\Psi_{\textrm{KdV}}p_2) + (2\mathsf{V})^{-1}\partial_\xi^3p_2=0.
\end{equation}
By  the change of variables 
\begin{equation}\label{KdVchangevari}
\Lambda =(2\mathsf{V})^{1/2}\tilde{\Lambda}, \quad \tilde{\xi}=(2\mathsf{V})^{1/2}\xi, \quad \tilde{p}(\tilde{\xi},\tilde{\Lambda})=p_2(\xi,\Lambda),
\end{equation}
\eqref{KdVEigen} becomes (see \eqref{solutionKdV2} for the form of $\Psi_{\textrm{KdV}}$)
\begin{equation}\label{KdVEigen2}
\tilde{\Lambda} \tilde{p} - \partial_{\tilde{\xi}}\tilde{p} + \partial_{\tilde{\xi}}(\tilde{\Psi}\tilde{p}) + \partial_{\tilde{\xi}}^3 \tilde{p}=0, \quad \text{where} \; \tilde{\Psi}(\tilde{\xi})  := 3\,\text{sech}^2(\tilde{\xi}/2) = \mathsf{V}\Psi_{\textrm{KdV}}(\xi).
\end{equation}

The Evans function for \eqref{KdVEigen2} is studied in \cite{PW,PW1}. We briefly summarize the results, and apply those directly to the equation  \eqref{KdVEigen}. The characteristic polynomial associated with \eqref{KdVEigen2} is 
\begin{equation}\label{dispersKdV2}
\tilde{d}_{KdV}(\tilde{\kappa}) = \tilde{d}_{KdV}(\tilde{\kappa},\tilde{\Lambda}) := \tilde{\Lambda} - \tilde{\kappa} + \tilde{\kappa}^3.
\end{equation}
For $\tilde{\Lambda} \in \tilde{\Omega}_{KdV}:=\mathbb{C}\setminus (-\infty,-2/(3\sqrt{3})]$, the zeros $\tilde{\kappa}_j$ of $\tilde{d}_{KdV}(\tilde{\kappa})$ can be labelled so that 
\begin{equation}\label{kappa bar splitting}
\oRe \tilde{\kappa}_1 <  \oRe \tilde{\kappa}_j \quad (j=2,3),
\end{equation}
and in particular, $\oRe \tilde{\kappa}_1<0$ (see Proposition 2.3 of \cite{PW1}). The Evans function $\tilde{D}_{KdV}(\tilde{\Lambda})$ for  \eqref{KdVEigen2} is defined on $\tilde{\Omega}_{KdV}$, and $\tilde{D}_{KdV}(\tilde{\Lambda})$ is characterized by the property 
\begin{equation}\label{EvansKdV-}
\lim_{\tilde{\xi} \to -\infty}e^{-\tilde{\kappa}_1\tilde{\xi}}\tilde{p}^+(\tilde{\xi},\tilde{\Lambda}) =   \tilde{D}_{KdV}(\tilde{\Lambda}) ,
\end{equation}
where $\tilde{p}^+$ is a unique solution to \eqref{KdVEigen2} satisfying
\begin{equation}\label{EvansKdV+}
\lim_{\tilde{\xi} \to +\infty}e^{-\tilde{\kappa}_1\tilde{\xi}}\tilde{p}^+(\tilde{\xi},\tilde{\Lambda}) = 1.
\end{equation}
In particular, $\tilde{D}_{KdV}(\tilde{\Lambda})$ is explicitly given by (Theorem 3.1 of \cite{PW1}) 
\begin{equation}\label{EvansKdVExplicit}
\tilde{D}_{KdV}(\tilde{\Lambda}) = \left(\frac{\tilde{\kappa}_1+1}{\tilde{\kappa}_1-1}\right)^2.
\end{equation}
From \eqref{dispersKdV2} and \eqref{EvansKdVExplicit}, one can check that $\tilde{D}_{KdV}(\tilde{\Lambda})$ vanishes only at $\tilde{\Lambda}=0$ and that the order of $\tilde{\Lambda}=0$ is two as a zero of $\tilde{D}_{KdV}(\tilde{\Lambda})$. Moreover, $\tilde{D}_{KdV}(\tilde{\Lambda}) \to 1$ as $|\tilde{\Lambda}| \to \infty$.

Now we apply the above results to construct the Evans function for the KdV equation \eqref{KdVEigen}.
The characteristic polynomial associated with \eqref{KdVEigen} is 
\begin{equation}\label{dispersKdV}
d_{KdV}(\kappa) = d_{KdV}(\kappa,\Lambda):=2\mathsf{V}(\Lambda - \kappa + (2\mathsf{V})^{-1}\kappa^3).
\end{equation}
We note that the zeros $\kappa_j$ of $d_{KdV}(\kappa)$ are related to $\tilde{\kappa}_j$ as  
\begin{equation}\label{kappa barkap rel}
\kappa_j = (2\mathsf{V})^{1/2}\tilde{\kappa}_j.
\end{equation}
Hence, from the relations \eqref{KdVchangevari} and \eqref{kappa barkap rel}, we have for 
\begin{equation}
\Lambda \in \Omega_{KdV}:=\mathbb{C}\setminus \left( -\infty, - (2\sqrt{2\mathsf{V}})/(3\sqrt{3}) \right],
\end{equation}
  $\kappa_j$ can be labelled  so that 
\begin{equation}\label{kappa splitting}
\oRe \kappa_1 <  \oRe \kappa_j \quad (j=2,3), \quad \oRe \kappa_1 <0.
\end{equation}

\begin{proposition}\label{KdVSummary} 
\begin{enumerate}
\item The Evans function $D_{KdV}(\Lambda)$ for the KdV equation \eqref{KdVEigen} is defined on the domain $\Omega_{KdV}$ and satisfies 
\begin{equation}\label{EvansKdV2-}
\lim_{\xi \to -\infty}e^{-\kappa_1 \xi}p_2^+(\xi,\Lambda) = D_{KdV}(\Lambda),
\end{equation}
where $\kappa_1$ is a unique zero of $d_{KdV}(\kappa)$ satisfying \eqref{kappa splitting}, and $p_2^+$ is a unique solution to \eqref{KdVEigen} satisfying
\begin{equation}\label{EvansKdV2+}
\lim_{\xi \to +\infty}e^{-\kappa_1 \xi}p_2^+(\xi,\Lambda) = 1.
\end{equation}
\item $D_{KdV}(\Lambda) = \left(\frac{\kappa_1+\sqrt{2\mathsf{V}}}{\kappa_1-\sqrt{2\mathsf{V}}}\right)^2$. 
\item $\Lambda=0$ is the only zero of $D_{KdV}(\Lambda)$, and its order is two.
\item $D_{KdV}(\Lambda) \to 1$ as $|\Lambda| \to \infty$ with $\Lambda \in \Omega_{KdV}$.
\end{enumerate}
\end{proposition}

\begin{proof}
The first statement can be checked by reformulating the eigenvalue problem \eqref{KdVEigen}  to the associated first-order ODE system (let $\mathbf{p}_2:=(p_2,\partial_\xi p_2,\partial_\xi^2 p_2 )^T$). We omit the details.

 To prove the last two assertions, it is enough to check that $D_{KdV}(\Lambda)=\tilde{D}_{KdV}(\tilde{\Lambda})$ for $\Lambda =(2\mathsf{V})^{1/2}\tilde{\Lambda}$. From the relations \eqref{KdVchangevari} and \eqref{kappa barkap rel}, we see that 
\begin{equation}\label{kappa barkap rel2}
\tilde{\kappa}_1\tilde{\xi} = \kappa_1 \xi,
\end{equation}
and that $\tilde{p}^+((2\mathsf{V})^{1/2}\xi,(2\mathsf{V})^{-1/2}\Lambda)$ satisfies  \eqref{KdVEigen}. From  \eqref{EvansKdV+}, \eqref{EvansKdV2+} and \eqref{kappa barkap rel2}, we have
\[
\tilde{p}^+((2\mathsf{V})^{1/2}\xi,(2\mathsf{V})^{-1/2}\Lambda) \sim e^{\kappa_1\xi} \quad \text{as} \quad \xi \to +\infty.
\]
Hence, we must have $\tilde{p}^+(\tilde{\xi},\tilde{\Lambda}) =  p_2^+(\xi,\Lambda)$, and we conclude that $D_{KdV}(\Lambda)=\tilde{D}_{KdV}(\tilde{\Lambda})$ from \eqref{EvansKdV-}, \eqref{EvansKdV2-} and \eqref{kappa barkap rel2}.
\end{proof}

\subsection{The Evans Function for the Euler-Poisson System in the KdV Scaling}\label{Subsec3.7}
Motivated by the formal derivation of the linearized KdV equation, we consider the transformation
\begin{subequations}\label{KdVscaling}
\begin{align}
& \xi = \veps^{1/2}\,x, \quad \lambda = \veps^{3/2}\Lambda, \label{KdVscaling1} \\
 \empheqlbrace \begin{split} 
   & \dot{n}(x) = p_2(\xi), \quad  \dot{u}(x) = \veps p_1(\xi) + \mathsf{V}p_2(\xi), \\
   & \dot{\phi}(x)  = p_2(\xi) +   \veps p_4(\xi), \quad  \dot{\psi}(x) = \veps^{1/2}\,p_3(\xi).
   \end{split} \label{KdVscaling2}
\end{align}
\end{subequations}
It is easy to see that $S\mathbf{p} = \mathbf{y}$, where $\mathbf{p}:=(p_1,p_2,p_3,p_4)^T$ and $S$ is given by 
\begin{equation}\label{Def_SSinverse}
S := \begin{pmatrix}
0 & 1 & 0 & 0 \\
\veps & \mathsf{V} & 0 & 0 \\
0 & 1 & 0 & \veps \\
0 & 0 & \veps^{1/2} & 0
\end{pmatrix}
\quad
\text{ with its inverse }
S^{-1} = \begin{pmatrix}
-\mathsf{V}\veps^{-1} & \veps^{-1} & 0 & 0 \\
1 & 0 & 0 & 0 \\
0 & 0 & 0 & \veps^{-1/2} \\
-\veps^{-1} & 0 & \veps^{-1} & 0
\end{pmatrix},
\end{equation}
 is a matrix for the transformation \eqref{KdVscaling2}.
 Then \eqref{ODE_LinEP} becomes
\begin{equation}\label{ODEscaled}
\frac{d\mathbf{p}}{d\xi} = A_\ast(\xi,\Lambda,\veps) \mathbf{p}, \quad \textrm{where } A_\ast(\xi,\Lambda,\veps):=\frac{1}{\sqrt{\veps}}S^{-1}A \left(\frac{\xi}{\sqrt{\veps}},\veps^{3/2}\Lambda,\veps\right)S.
\end{equation} 
(see Appendix \ref{SFA} for the specific form of $A_\ast$.) Using \eqref{pointesti2}, a straightforward calculation shows that there is a uniform constant $C>0$ such that for all $\veps\in[0,\veps_\ast]$,
\begin{equation}\label{ODEscaled1}
|A_\ast(\xi,\Lambda,\veps) - A_\ast^\infty(\Lambda,\veps)| \leq Ce^{-C|\xi|}, \quad \xi \in \mathbb{R},
\end{equation}
where $\textstyle A_\ast^\infty(\Lambda,\veps):=\lim_{|\xi| \to \infty}A_\ast(\xi,\Lambda,\veps)$, and for $\veps=0$,
\begin{equation}\label{ODEscaled2}
A_\ast(\xi,\Lambda,0) := 
\begin{pmatrix}
0 & -2\mathsf{V}\partial_\xi\Psi_{\textrm{KdV}} -\Lambda  & 1-2\mathsf{V}\Psi_{\textrm{KdV}} & 0 \\
0 & 0 & 1 & 0 \\
0 & \Psi_{\textrm{KdV}} & 0 & 1 \\
0 & -(2\mathsf{V}^2 + 1)\partial_\xi\Psi_{\textrm{KdV}} - 2\mathsf{V}\Lambda & 2\mathsf{V} - (2\mathsf{V}^2 + 1)\Psi_{\textrm{KdV}} & 0
\end{pmatrix}.
\end{equation}
For $\veps=0$,  a set of the last three equations of \eqref{ODEscaled} reduces to the KdV equation \eqref{KdVEigen}.

The  characteristic polynomial of the asymptotic matrix $\textstyle A_\ast^\infty(\Lambda,\veps)$ is given by
\begin{equation}\label{dispersScaled}
\begin{split}
d_\ast(\nu)= d_\ast(\nu,\Lambda,\veps)
& :=\text{det}\left(A_\ast^\infty(\Lambda,\veps) - \nu I \right)  = \veps^{-2}\text{det}\big(A^\infty(\veps^{3/2}\Lambda,\veps) - \sqrt{\veps}\,\nu I \big) \\
&  =\veps^{-2} d(\sqrt{\veps}\,\nu,\veps^{3/2}\Lambda,\veps) \\
& = \frac{\nu \,d_{\text{KdV}}(\nu) - \veps(\Lambda-\nu)^2 + \veps \nu^2\left[\veps\Lambda^2 - 2c\Lambda \nu + (2\mathsf{V}+\veps)\nu^2 \right]}{c^2-K},
\end{split}
\end{equation}
where $d$ and $d_{\text{KdV}}$ are defined in \eqref{dispersEP} and \eqref{dispersKdV}, respectively.

In order to define the Evans function for the Euler-Poisson system in the KdV scaling, we first shall inspect the zeros of the characteristic polynomial $d_\ast(\nu)$ and check \textbf{H4}. 
From the scaling \eqref{KdVscaling1} and the second line of \eqref{dispersScaled}, we see that for $\veps>0$, the zeros $\nu_j$ of $d_\ast(\nu)$ are related to the zeros $\mu_j$ of $d(\mu)$ by
\begin{equation}\label{Relation nu mu}
\nu_j = \veps^{-1/2} \mu_j.
\end{equation}
When $\veps=0$, the zeros of $d_\ast(\nu)$ are comprised of $0$ and the three zeros $\kappa_j$ of $d_{KdV}(\kappa)$. Together with \eqref{split mu weight} and \eqref{kappa splitting}, these yield that the zeros $\nu_j$ of $d_\ast(\nu)$ can be labelled so that they satisfy
\begin{equation}\label{nu split}
\oRe \nu_1 < \oRe \nu_j \quad (j=2,3,4), \quad \oRe \nu_1 < 0
\end{equation}
for all $\Lambda$ such that $\veps^{3/2}\Lambda = \lambda \in \Omega^\veps$ when $\veps \in (0,\veps_K')$ and for all $\Lambda \in \Omega_{KdV}$ when $\veps=0$.

For any fixed $0 < c_0 < \sqrt{2\mathsf{V}}$, we define the domain 
\begin{equation}\label{DomOmegatilde}
\Omega_\ast:=\left\{\Lambda : \oRe \Lambda \geq  -\eta(c_0)\right\},
\end{equation}
where $\eta(c_0)$ is a positive function of $c_0$ defined in \eqref{ND-eps}. 
Recalling the definition of $\Omega^\veps$ (see  \eqref{ND-eps}), we see that $\Lambda \in \Omega_\ast$ if and only if $\lambda \in \Omega^\veps$ as long as  $\veps^{3/2}\Lambda=\lambda $ for $\veps>0$. One may check that  $\{\Lambda:\oRe \Lambda \geq 0\}\subset \Omega_\ast \subset \Omega_{KdV}$ for all $0 < c_0 < \sqrt{2\mathsf{V}}$. Hence, we conclude that \textbf{H4} holds on the domain $\Omega_\ast$.

\begin{proposition}\label{DefEvansFunctDomain2}
For fixed $0<c_0<\sqrt{2\mathsf{V}}$, there exists $\veps_K'>0$ such that for each $\veps \in [0,\veps_K']$, the Evans function $D_\ast(\Lambda,\veps)$ for \eqref{ODEscaled} is defined on the domain $\Omega_\ast$, and it is analytic in $\Lambda\in \Omega_\ast$.  Moreover, $D_\ast(\Lambda,\veps)$ is characterized by
\begin{equation}\label{CharacEvansDtild}
\lim_{\xi \to -\infty} e^{-\nu_1 \xi}\mathbf{p}^+(\xi,\Lambda) = D_\ast(\Lambda,\veps) \mathbf{v}_1^\ast,
\end{equation}
where   $\mathbf{p}^+(\xi,\Lambda)$ is a unique solution  to \eqref{ODEscaled} satisfying 
\begin{equation}\label{ODEsol2}
\lim_{\xi \to +\infty} e^{-\nu_1 \xi}\mathbf{p}^+(\xi,\Lambda) = \mathbf{v}_1^\ast
\end{equation} 
and $\mathbf{v}_1^\ast$, a right eigenvector  of 
$A_\ast^\infty$ associated with a nonzero simple eigenvalue $\nu_1$, is given by
\begin{equation}\label{righteigenvec_tildeA}
\begin{split}
\mathbf{v}_1^\ast :=
\left(1-\frac{\Lambda}{\nu_1},\; 1,\; \frac{\nu_1}{1-\veps \nu_1^2},\;\frac{\nu_1^2}{1-\veps \nu_1^2}  \right)^T.
\end{split}
\end{equation}
\end{proposition}

\begin{proof} 
\textbf{H4} is verified from the above discussion. \textbf{H1}--\textbf{H3} can be easily checked from the structure of $A_\ast$ (see \eqref{Def_SSinverse}--\eqref{ODEscaled2} and Appendix \ref{SFA}) and  \eqref{pointesti2}. It is clear that $S^{-1}\mathbf{v}_1$ is a right eigenvector of the asymptotic matrix $A_\ast^\infty$ associated with $\nu_1$. From \eqref{KdVscaling1}, \eqref{Relation nu mu}, and that  $c=\mathsf{V}+\veps$, we obtain that
\[
 S^{-1}\mathbf{v}_1 
 = \left( -\frac{\mathsf{V}}{\veps} + \frac{c}{\veps} - \frac{\lambda}{\veps\mu_j}, 1,  \frac{1}{\sqrt{\veps}}\frac{\mu_j}{1-\mu_j^2},  -\frac{1}{\veps} + \frac{1}{\veps(1-\mu_j^2)} \right)^T = \mathbf{v}_1^\ast.
\]
This completes the proof. 
\end{proof}

\subsection{Relations among $D(\lambda,\veps)$, $D_\ast(\Lambda,\veps)$ and $D_{KdV}(\Lambda)$ and their continuities}
\begin{proposition}\label{RelationDandDtilde}
For each $\veps\in(0,\veps_K']$, $D_\ast(\Lambda,\veps)=D(\lambda,\veps)$ for $\veps^{-3/2}\lambda = \Lambda \in \Omega_\ast$. For $\veps=0$, $D_\ast(\Lambda,0) = D_{KdV}(\Lambda)$ for $\Lambda \in \Omega_\ast$.
\end{proposition}
\begin{proof}
Recalling the transform \eqref{KdVscaling}, we observe that $ S\mathbf{p}^{+}(\sqrt{\veps}x,\veps^{-3/2}\lambda)$ is a solution to \eqref{ODE_LinEP}, and that $S\mathbf{p}^{+} = O(e^{\mu_1 x})$ as $x \to +\infty$ from  \eqref{Relation nu mu} and \eqref{ODEsol2}.  Hence, by Proposition \ref{Prop_ResultPW}(\ref{Dy+1}), $S\mathbf{p}^+(\xi,\Lambda)=\alpha\mathbf{y}^+(x,\lambda)$ for some constant $\alpha \neq 0$, and in particular, $p_2^+=\alpha y_1^+$. On the other hand, the first component of $\mathbf{v}_1$ and the second component of $\mathbf{v}_1^\ast$ are equal to 1. Thus, we see that $p_2^+=y_1^+$ by \eqref{ODEsol1} and \eqref{ODEsol2}, and  $D(\lambda,\veps) = D_\ast(\Lambda,\veps)$ by \eqref{CharacEvansD} and \eqref{CharacEvansDtild}. In a similar fashion, one can obtain $D_\ast(\Lambda,0) = D_{KdV}(\Lambda)$, for which we omit the details here.
\end{proof}

\begin{proposition}\label{Prop_JointConti}
$D(\lambda,\veps)$ and $D_\ast(\Lambda,\veps)$ are jointly continuous on the sets $\{(\lambda,\veps):  \lambda \in \Omega^\veps ,\; \veps \in[0,\veps_K) \}$ and $\{(\Lambda,\veps):  \Lambda \in \Omega_\ast,  \; \veps\in[0,\veps_K'] \}$, respectively.
\end{proposition}
See Remark \ref{Rem_Dom} for the discussion on the domain $\Omega^0$. 
\begin{proof}
We briefly sketch the proof omitting the details. It can be proved by a fixed point iteration argument following Section 8 of \cite{PW2} (see also  Proposition 1.2 of \cite{PW}).

By the definition of $D(\lambda,\veps)$, it is enough to show that $\mathbf{y}^+(x,\lambda,\veps)$ and $\mathbf{z}^-(x,\lambda,\veps)$, the solutions to \eqref{ODE_LinEP} and  \eqref{TranspoODE_DefEvans} respectively, are jointly continuous in $(\lambda,\veps)$ for each $x$. Indeed, $\mathbf{y}^+$  is constructed as a fixed point $\boldsymbol{\theta}^+ =e^{-\mu_1 x}\mathbf{y}^+$ of the operator
\[
(\mathcal{F}\boldsymbol{\theta})(x;\lambda,\veps) = \mathbf{v}_1(\lambda,\veps) -\int_x^\infty e^{B(\lambda,\veps)(x-s)}\big( A(s;\lambda,\veps) - A^\infty(\lambda,\veps) \big)\boldsymbol{\theta}(s;\lambda,\veps)\,ds,
\]
where $B(\lambda,\veps)=A^\infty(\lambda,\veps) - \mu_1(\lambda,\veps)I$, on $C_b([x_0,\infty))$ for sufficiently large $x_0>0$. On any fixed compact subset of  $\{(\lambda,\veps): \lambda \in \Omega^\veps ,\; \veps \in[0,\veps_K) \}$, we see that  $A$, $A^\infty$, $\mu_1$, and $\mathbf{v}_1$ are all jointly continuous in $(\lambda,\veps)$, $|e^{Bx}| <C$ for all $x  < 0$ since $\oRe \mu_1<\mu_\ast$ (see \eqref{split mu weight}) , and $\textstyle \lim_{x_0\to+\infty}\int_{x_0}^\infty|A - A^\infty|\,ds = 0$ uniformly in $(\lambda,\veps)$. Hence, by a standard iteration argument, the fixed point $\boldsymbol{\theta}^+$ (and thus $\mathbf{y}^+$) is jointly continuous in $(\lambda,\veps)$ on any compact subset. The joint continuity of $D_\ast(\Lambda,\veps)$ can be proved in a similar manner.
\end{proof}

\subsection{Absence of Nonzero Eigenvalues}

This subsection is devoted to prove that $D_\ast(\Lambda,\veps)$ only vanishes at $\Lambda=0$ in the domain $\Omega_*$. To this end, we show
 that $D_\ast(\Lambda,\veps)$ converges to $D_{KdV}(\Lambda)$ as $\veps\to 0$ uniformly in $\Lambda\in \Omega_*$.
\begin{proposition}\label{ThmEvansConverg}
There holds
\[
\sup_{\Lambda \in \Omega_\ast}|D_\ast(\Lambda,\veps) - D_{KdV}(\Lambda) | \to 0  \quad \text{as} \;\;  \veps \to 0.
\]
\end{proposition}
 
From the convergence of $D_\ast(\Lambda,\veps)$, combined with the properties of $D_{KdV}(\Lambda)$ (Proposition  \ref{KdVSummary}) and some results of previous subsections, we have the following:
\begin{corollary}\label{cor_orderEvans}
There exists $\veps_0>0$ such that for $\veps \in (0,\veps_0]$, $\Lambda=0$ is the only zero of $D_\ast(\Lambda,\veps)$ on the region $\Omega_\ast$. Moreover, the order of $\Lambda=0$ is exactly two.
\end{corollary} 

We note that the characteriztion of the zeros of $D(\lambda,\veps)$ (Proposition \ref{Pro_Ev_Ch})  directly follows from Corollary \ref{cor_orderEvans} and Proposition \ref{RelationDandDtilde}. Since the proof of Proposition~\ref{ThmEvansConverg} consists of several steps, we prove it after we present  the proof of Corollary \ref{cor_orderEvans}.

\begin{proof}[Proof of Corollary \ref{cor_orderEvans}]
 For $\delta>0$, let $\Gamma_\delta$ be the boundary of the region $\Omega_\ast \cap \{\Lambda : |\Lambda|\leq \delta^{-1}\}$.  Since $D_{KdV}(\Lambda)$ tends to 1 as $|\Lambda|\to \infty$ (see Proposition~\ref{KdVSummary}), there exists $\delta_0>0$ such that 
\[
\inf_{\Lambda\in \Omega_\ast, |\Lambda|\geq \delta^{-1}} |D_{KdV}(\Lambda)| > \frac{1}{2} \quad \text{for all } \delta \in (0,\delta_0].
\]
Moreover, since $D_{KdV}(\Lambda)$  vanishes only at $\Lambda=0$, there exists $\gamma_0=\gamma_0(\delta_0)>0$ such that 
\[
\inf_{\Lambda\in \Gamma_{\delta_0}} |D_{KdV}(\Lambda)| > \gamma_0.
\]
We recall that $\Omega_\ast$ is closed. Hence, the above two inequalities imply that for all $\delta\in(0,\delta_0]$, 
\[
\inf_{\Lambda\in \Gamma_\delta} |D_{KdV}(\Lambda)| >  \min\{1/2,\gamma_0\} =: \gamma_0'(\delta_0). 
\]
Combined with Proposition \ref{ThmEvansConverg}, this implies that  there is $\veps_0>0$ such that for all $\veps\in(0,\veps_0]$ and $\delta \in (0,\delta_0]$,
\begin{equation}\label{ConKdV1}
|D_{KdV}(\Lambda)| > \gamma_0'(\delta_0) > |D_\ast(\Lambda,\veps) - D_{KdV}(\Lambda)| \quad \text{for }\Lambda \in \Gamma_\delta.
\end{equation}
We recall that  the order of $\Lambda=0$ as a zero of $D_{KdV}(\Lambda)$ is two, and there is no other zero. Combining Proposition \ref{EvansZero} and Proposition \ref{RelationDandDtilde}, $\Lambda=0$ is a zero of $D_\ast(\Lambda,\veps)$ with the order at least two. Now the proof is finished by applying Rouch\'e's theorem since \eqref{ConKdV1} is true for all $\delta\in (0,\delta_0]$.
\end{proof}

In order to prove Proposition \ref{ThmEvansConverg}, we divide the region $\Omega_\ast$ as follows (see Figure \ref{DecomOmega} and recall that $\Lambda \in \Omega_\ast$ if and only if $\lambda \in \Omega^\veps$ as long as  $\veps^{3/2}\Lambda=\lambda $ for $\veps>0$):
\begin{subequations}
\begin{align*}
& \mathfrak{D}_1:=\Omega_\ast \cap \{\Lambda:|\Lambda| \leq \delta^{-1}\},   \\ 
& \mathfrak{D}_2:= \Omega^\veps \cap \{\lambda: \veps^{3/2}\delta^{-1} < |\lambda| < \delta\},\\
& \mathfrak{D}_{31}:= \{\lambda : \oRe \lambda \geq 0\}  \cap \{\lambda: \delta \leq |\lambda| \leq \delta^{-1}\}, \\
& \mathfrak{D}_{32}:= \{\lambda : -\veps^{3/2}\eta(c_0) \leq  \oRe \lambda \leq  0, \, \delta/2 \leq |\oIm \lambda| \leq \delta^{-1}\},  \\
& \mathfrak{D}_4:= \Omega^\veps \cap \{\lambda: \delta^{-1} < |\lambda|\},
\end{align*}
\end{subequations}
where $-\veps^{3/2}\eta(c_0)$ is the  boundary of the domain $\Omega^\veps$ (see \eqref{ND-eps}). It is clear that for any fixed $\delta>0$,  $\Omega_\ast$ is a union of these regions for all sufficiently small $\veps>0$. In fact, $\delta>0$ will be determined later.\\

\begin{figure}[h]
    \centering
        \includegraphics[scale=0.4]{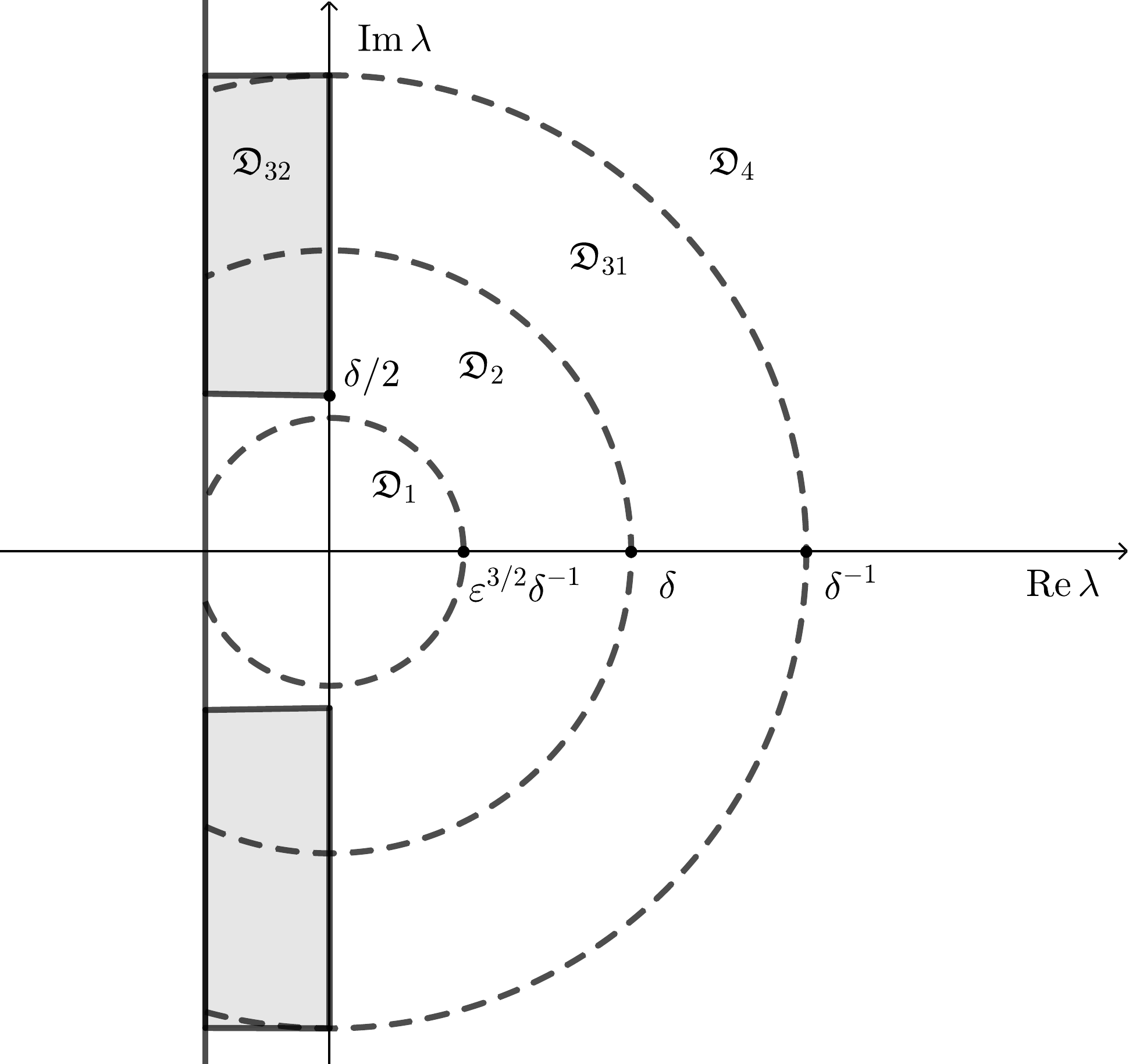}
        \caption{Decomposition of $\Omega^\veps$}
        \label{DecomOmega}
\end{figure}

\textit{On $\mathfrak{D}_1$}: We note that $D(\lambda,\veps)$ is \emph{singular} at $\lambda=0$ in the limit $\veps \to 0$. First of all, $D(\lambda,\veps)$ is not defined at $(\lambda,\veps)=(0,0)$. Moreover, we have $D(\lambda,\veps)|_{\veps=0}=1$ for all $\lambda\in\Omega^0$  (see Remark \ref{Rem_Dom}), but $D(\lambda,\veps)|_{\lambda=0}=0$ for $\veps\in(0,\veps_K]$ by Proposition \ref{EvansZero}.\eqref{EvansZero1st}. However, $D_\ast(\Lambda,\veps)$ is \emph{regular} at $\Lambda=0$ in the  sense of Lemma \ref{LemEstiD1}.
\begin{lemma}\label{LemEstiD1}
For any fixed constant $\delta>0$, 
\begin{equation}\label{EstiD_D1}
\sup_{\Lambda \in \mathfrak{D}_1}|D_\ast(\Lambda,\veps) -D_{KdV}(\Lambda)| \to 0 \quad \text{as} \quad \veps \to 0.
\end{equation}
\end{lemma} 
\begin{proof}
It is easy to see that $\sup_{\Lambda \in \mathfrak{D}_1}|D_\ast(\Lambda,\veps) -D_{KdV}(\Lambda)|$ is continuous on $[0,\veps_K']$. Indeed,
$D_\ast(\Lambda,\veps)$ is uniformly continuous on a fixed compact set $\{ (\Lambda,\veps): \veps \in[0,\veps_K'], \Lambda \in \mathfrak{D}_1\}$ since it is jointly continuous on the set (see Proposition \ref{Prop_JointConti}).  Now \eqref{EstiD_D1} follows from  that $D_\ast(\Lambda,0)=D_{KdV}(\Lambda)$ by Proposition \ref{RelationDandDtilde}.
\end{proof}

\textit{On $\mathfrak{D}_2$ and $\mathfrak{D}_4$}: On the region $\mathfrak{D}_2$, $\lambda$ is small, but not zero.  The region $\mathfrak{D}_4$ is where $\lambda$ is large. Since the proofs of Lemma \ref{Lem_Region2} and Lemma \ref{Lem_Region4} rather lengthy, they will be proved in Section \ref{Sec5}.  \\

\begin{lemma}\label{Lem_Region2}
There exist constants $C_2,\delta_2,\veps_2>0$ such that for all $\veps \in [0,\veps_2]$ and $\delta \in (0,\delta_2]$, there holds
\begin{equation}\label{EstiD_D2}
\sup_{\lambda\in \mathfrak{D}_2}|D(\lambda,\veps)-1| < C_2 \delta^{1/3}.
\end{equation}
Here $C_2$ is independent of $\veps$ and $\delta$.
\end{lemma}
\begin{lemma}\label{Lem_Region4}
There exist constants $C_4,\delta_4,\veps_4>0$ such that for all $\veps \in [0,\veps_4]$ and $\delta \in (0,\delta_4]$,
\begin{equation}\label{EstiD_D4}
\sup_{\lambda\in \mathfrak{D}_4}|D(\lambda,\veps)-1| < C_4 \veps^{1/2}.
\end{equation}
Here $C_4$ is independent of $\veps$ and $\delta$.
\end{lemma}

\textit{On $\mathfrak{D}_{31}$  and $\mathfrak{D}_{32}$}: We note that these regions do not contain $\lambda=0$, which is a singular point of $D(\lambda,\veps)$ in the limit $\veps\to 0$. We recall that  the coefficient matrix of the system \eqref{ODE_LinEP} is independent of $x$ when $\veps=0$.
\begin{lemma}
For any fixed constant $\delta>0$, there hold  
\begin{subequations}\label{EstiD_D30}
\begin{align}
& \sup_{\lambda \in \mathfrak{D}_{31}}|D(\lambda,\veps) -1| \to 0 \quad \text{as}\quad \veps \to 0, \label{EstiD_D3} \\
& \sup_{\lambda \in \mathfrak{D}_{32}}|D(\lambda,\veps) -1| \to 0 \quad \text{as}\quad \veps \to 0. \label{EstiD_D32}
\end{align}
\end{subequations}
\end{lemma}

\begin{proof}
From that $D(\lambda,\veps)$ is jointly continuous on a fixed compact set $\{(\lambda,\veps):\lambda \in \mathfrak{D}_{31}, \veps \in [0,\veps_K]\}$ and that $D(\lambda,0)=1$ for $\lambda \in \mathfrak{D}_{31}$,  \eqref{EstiD_D3} can be obtained in a similar manner as in the proof of Lemma \ref{LemEstiD1}.

Even though the region $\mathfrak{D}_{32}$ depends on $\veps$, one may extend $D(\lambda,\veps)$ to a function $D^{\veps_K}(\lambda,\veps)$, jointly continuous on a fixed compact set 
\[
\{(\lambda,\veps): -\veps_K^{3/2}\eta(c_0) \leq \oRe \lambda \leq 0, \, \delta/2\leq |\oIm \lambda| \leq \delta^{-1},\veps\in[0,\veps_K]\}
\]
 by defining
\[
D^{\veps_K}(\lambda,\veps):=
\left\{
\begin{array}{l l}
D(\lambda,\veps) & \text{ on } \{(\lambda,\veps):\lambda \in   \mathcal{D}_{32},\, \veps \in [0,\veps_K]\}, \\
D( -\veps^{3/2}\eta(c_0) + i\oIm \lambda,\veps) & \text{on} \; \left\{ (\lambda,\veps): \begin{array}{l}
   -\veps_K^{3/2}\eta(c_0) \leq \oRe \lambda < -\veps^{3/2}\eta(c_0),  \\
   \delta/2\leq |\oIm \lambda| \leq \delta^{-1},\, \veps\in[0,\veps_K]
  \end{array}\right\}.
\end{array}
\right.
\]
Hence,
\[
 \sup_{\substack{-\veps_K^{3/2}\eta(c_0) \leq \oRe \lambda \leq 0,\\  \delta/2\leq |\oIm \lambda| \leq \delta^{-1}}} |D^{\veps_K}(\lambda,\veps) - 1| = \sup_{\lambda \in \mathfrak{D}_{32}}|D(\lambda,\veps) -1|
\]
is continuous in $\veps\in[0,\veps_K]$, and we obtain \eqref{EstiD_D32}. 
\end{proof}

By using \eqref{EstiD_D1}--\eqref{EstiD_D30}, we now prove  Proposition \ref{ThmEvansConverg}.

%
%
%
%

\begin{proof}[Proof of Proposition \ref{ThmEvansConverg}]
Let $\gamma>0$ is given. From the property of $D_{KdV}$, \eqref{EstiD_D2}, and \eqref{EstiD_D4}, there exist constants $\delta_\gamma,\veps_\gamma>0$ such that for all $\veps\in(0,\veps_\gamma]$ and $\delta \in (0,\delta_\gamma]$, there hold 
\begin{equation}\label{EvansConverg1}
\sup_{|\Lambda|\geq \delta_\gamma^{-1}}|D_{KdV}(\Lambda) - 1| < \frac{\gamma}{2},
\end{equation}
\begin{equation}\label{EvansConverg2}
\sup_{\lambda\in \mathfrak{D}_2}|D(\lambda,\veps)-1| < \frac{\gamma}{2}, \quad  \sup_{\lambda\in \mathfrak{D}_4}|D(\lambda,\veps)-1| < \frac{\gamma}{2}.
\end{equation}
From \eqref{EstiD_D30}, there is a constant $\veps_3>0$ such that for all $\veps\in(0,\veps_3]$,
\begin{equation}\label{EvansConverg3}
\sup_{\lambda \in \Omega^\veps, \delta_\gamma \leq |\lambda| \leq \delta_\gamma^{-1}  }|D(\lambda,\veps) -1| < \frac{\gamma}{2}.
\end{equation}
Since $D(\lambda,\veps)=D_\ast(\Lambda,\veps)$, it follows from \eqref{EvansConverg1}, \eqref{EvansConverg2} and \eqref{EvansConverg3} that 
\begin{equation}\label{EvansConverg4}
\sup_{|\Lambda| \geq \delta_\gamma^{-1}}|D_\ast(\Lambda,\veps) -D_{KdV}(\Lambda)| < \gamma.
\end{equation}
From \eqref{EstiD_D1}, there exists $\veps_1>0$ such that for all $\veps\in (0,\veps_1]$,
\begin{equation}\label{EvansConverg5}
\sup_{|\Lambda| \leq \delta_\gamma^{-1}}|D_\ast(\Lambda,\veps) -D_{KdV}(\Lambda)| < \gamma.
\end{equation}
From \eqref{EvansConverg4} and \eqref{EvansConverg5}, we conclude that there is $\veps_0:=\min\{\veps_\gamma, \veps_1,\veps_3\}$ such that for all $\veps\in(0,\veps_0]$,
\[
\sup_{\Lambda\in\Omega_\ast}|D_\ast(\Lambda,\veps) -D_{KdV}(\Lambda)| < \gamma.
\]
This finishes the proof.
\end{proof}

\section{Splitting of matrix eigenvalues}\label{Sec_Split}
In this section, we prove Lemma \ref{splitting mu lam} and Lemma \ref{splitting mu lam weight}: the splitting properties of the roots of the characteristic polynomial $d(\mu)$.  We first present some preliminary observations.

We recall that $d(\mu)=0$ is equivalent to that $\mu$ satisfies one of the equations (see \eqref{dispersEP}--\eqref{Charact1})
\begin{equation*}
d_{\pm}(\mu)= \mu \left(c \pm \sqrt{\frac{1}{1-\mu^2}+ K}\, \right) = \lambda.
\end{equation*}
Here $d_\pm(\mu)$ are analytic in $\mu\in \mathbb{C}$ except the branch cut $(-\infty,-1]\cup [1,+\infty)$ since we may write
\[
\sqrt{\frac{1}{1-\mu^2}+ K} = \sqrt{K}\frac{\sqrt{\sqrt{1+\frac{1}{K}}+\mu}\cdot\sqrt{\sqrt{1+\frac{1}{K}}-\mu}}{\sqrt{1+\mu}\cdot\sqrt{1-\mu}}.
\]
By inspection, we have
\begin{equation}\label{Diffdpm}
\partial_\mu d_{\pm}(\mu) = c \pm \frac{1+K(1-\mu^2)^2}{(1-\mu^2)^2\sqrt{\frac{1}{1- \mu^2} + K}\,}, \quad \partial_\mu^2d_{\pm}(\mu) = \pm\frac{\mu(-K\mu^4- 2K\mu^2+3K+3)}{(1-\mu^2)^4\left(\frac{1}{1-\mu^2}+K\right)^{3/2}}.
\end{equation}

Plugging $\lambda = -i \omega$ and $\mu = ik$ for $\omega,k \in \mathbb{R}$ into $\lambda=d_{\pm}(\mu)$, we obtain 
\begin{equation}\label{map_omega2}
\omega=\omega_\pm(k):=i\cdot d_\pm(ik)= -k\left( c \pm \sqrt{\frac{1}{1+k^2} + K } \right).
\end{equation}
By inspection, we see that
\begin{equation*}
\partial_k \omega_\pm(k)  =  -\left(c \pm \frac{1+K(1+k^2)^2}{(1+k^2)^2\sqrt{\frac{1}{1+k^2} + K}} \right),  \quad \partial_k^2\omega_-(k) = \frac{k(Kk^4-2Kk^2-3K-3)}{(1+k^2)^4\left(\frac{1}{1+k^2}+K\right)^{3/2}}. 
\end{equation*}
Hence,
\begin{equation*}
\partial_k\omega_\pm |_{k=0} = -(c\pm\sqrt{1+K}), \quad \lim_{k \to \pm \infty}\partial_k\omega_\pm = -(c \pm \sqrt{K}).
\end{equation*}
Furthermore, one can check that $\partial_k\omega_-$ decreases on the interval $(0,\overline{k}_-)$ and increases on the interval $(\overline{k}_-,\infty)$, where $\overline{k}_-:= \sqrt{\tfrac{K+\sqrt{4K^2+3K}}{K}}$.
These observations yield that (see Figure \ref{graphomega})  for $\veps\geq 0$, where $\veps=c-\sqrt{1+K}$,
\begin{equation}\label{map_omega1}
\partial_k\omega_+(k) \leq  -c \quad\text{for } k\in\mathbb{R}, \quad \partial_k\omega_-(k) <0  \quad \text{for } k\in\mathbb{R}\setminus\{0\}, \quad \partial_k\omega_-(k)|_{k=0}=-\veps
\end{equation}
since $\partial_k\omega_\pm$ are symmetric about $k=0$. Lastly, we see that for any positive constant $\delta>0$, there is a constant $C(\delta)>0$ (independent of $\veps$) such that 
\begin{equation}\label{Map_om}
\sup_{|k|\in[\delta,\infty)}\partial_k\omega_-(k) 
  < \max\{\partial_k\omega_-(\delta),-(c-\sqrt{K})\} < -C(\delta).
\end{equation}
\begin{figure}[h]
    \centering
        \includegraphics[scale=0.4]{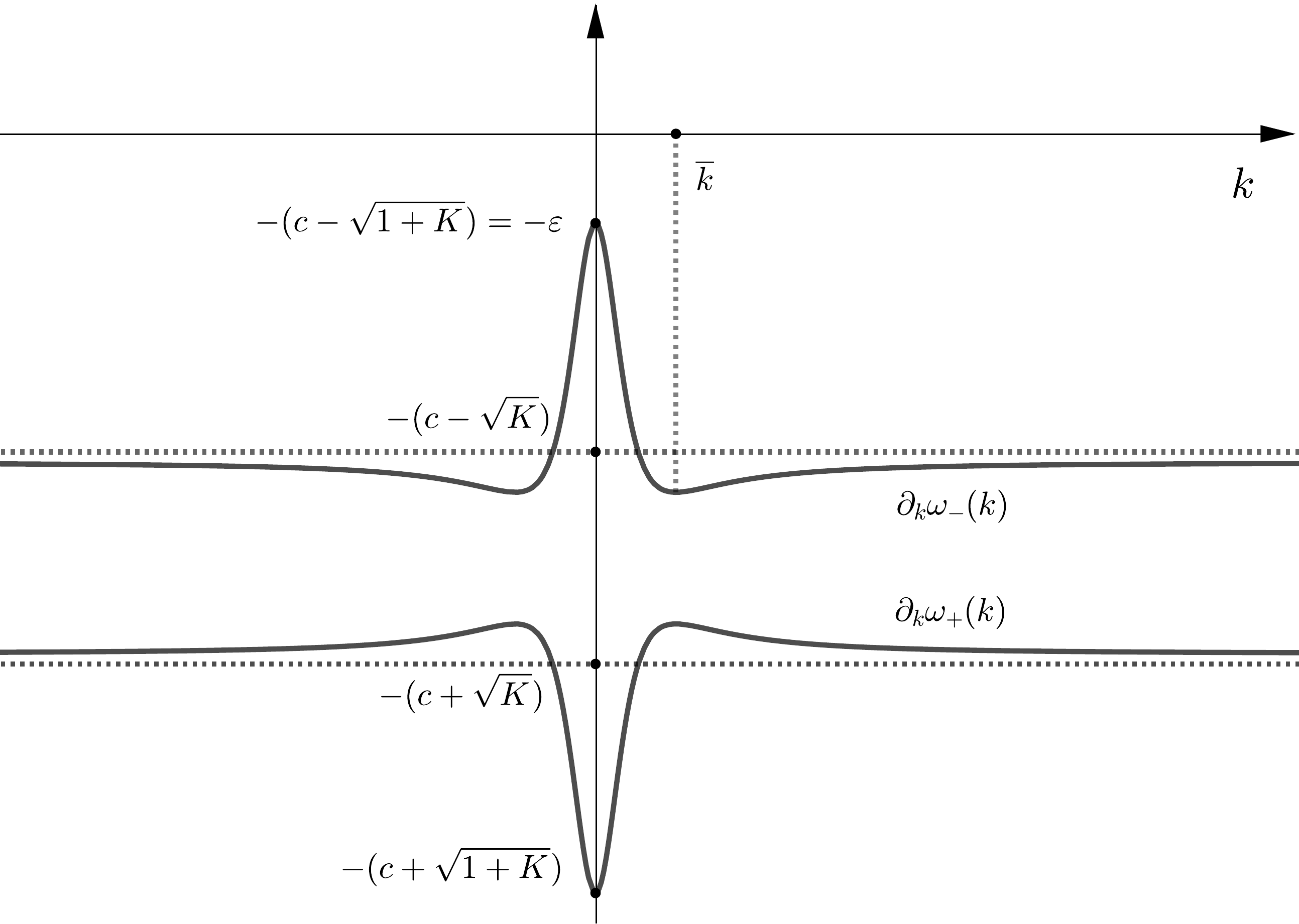}
        \caption{The graphs of $\partial_k\omega_\pm(k)$}
        \label{graphomega}
\end{figure}
\begin{proof}[Proof of Lemma \ref{splitting mu lam}]
We only consider the case $\veps>0$ since the case $\veps=0$ can be checked in a similar manner.   We first prove \eqref{EigenSpliting1_1}: for $\veps>0$, 
\[
\oRe\mu_1 < 0 = \oRe\mu_2 = \oRe\mu_3 < \oRe \mu_4 \quad  \text{when} \;  \oRe\lambda =0.
\]
 By \eqref{map_omega2}--\eqref{map_omega1}, the mappings 
\begin{equation}\label{SplitMuLam1}
k\in \mathbb{R} \mapsto d_{\pm}(ik)= -i\omega_\pm(k)\in i\mathbb{R}
\end{equation}
are one-to-one and onto. Hence for each $\lambda$ with $\oRe \lambda=0$, there exist exactly two zeros $\mu_2=\mu_2(\lambda,\veps)$ and $\mu_3=\mu_3(\lambda,\veps)$ of $d(\mu)$ satisfying $d_+(\mu_2)=\lambda$, $ d_-(\mu_3)=\lambda$, and
\begin{equation}
 \oRe\mu_2=\oRe\mu_3=0.
\end{equation}
We let $\mu_1=\mu_1(\lambda,\veps)$ and $\mu_4=\mu_4(\lambda,\veps)$ be the other two branches of solutions to $d_-(\mu)=\lambda$ satisfying
\begin{equation}\label{SplitMuLam2}
 -\sqrt{\frac{c^2-K-1}{c^2-K}}=\mu_1 < 0< \mu_4 = \sqrt{\frac{c^2-K-1}{c^2-K}} \quad \text{at } \lambda=0.
\end{equation}

Now we claim that $\mu_1$ satisfies $\oRe\mu_1<0$ without crossing the imaginary axis as long as  $\oRe \lambda=0$. Suppose that for some $\lambda_0$ with $\oRe \lambda_0=0$, there is $k_0\in\mathbb{R}$ such that $\mu_1= ik_0$ at $\lambda=\lambda_0$. Then, since the mappings \eqref{SplitMuLam1} are one-to-one and onto,  one must have $\mu_1=\mu_2$ or $\mu_1=\mu_3$ at $\lambda=\lambda_0$. In other words, $\lambda_0=d_+(\mu)=(\mu-\mu_1)^2\widetilde{d}_+(\mu)$ for some $\widetilde{d}_+$ with $\widetilde{d}_+(\mu_1) \neq 0$, or $\lambda_0=d_-(\mu)=(\mu-\mu_1)^2\widetilde{d}_-(\mu)$ for some $\widetilde{d}_-$ with $\widetilde{d}_-(\mu_1) \neq 0$ at $\lambda=\lambda_0$. This contradicts to \eqref{map_omega1} since by the chain rule,
\[
0=\partial_\mu d_\pm(\mu)|_{\mu =\mu_1(\lambda_0)=ik_0} =-i\partial_k d_\pm(ik)|_{k=k_0}= -\partial_k\omega_{\pm}(k)|_{k=k_0} \neq 0.
\]   
 For the same reason, we have $0 < \oRe\mu_4$ as long as $\oRe \lambda=0$.  Hence  \eqref{EigenSpliting1_1} is true.

Next we prove \eqref{EigenSpliting1_2}: for $\veps>0$,
\[
\oRe\mu_1 < 0 < \oRe\mu_j  \quad \text{when} \; \oRe\lambda >0 \quad (j=2,3,4).
\]
From \eqref{SplitMuLam1}, we see that   any solutions $\mu$ of $d_{\pm}(\mu)=\lambda$ with $\oRe \lambda  \neq 0$ cannot lie in the imaginary axis. Combined with the continuity, this further implies \emph{the consistent splitting property}: \textit{as long as} $\oRe\lambda>0$, 
\[
\textit{the number of zeros of }d(\mu) \textit{ lying on the left (resp. right) half-plane  does not change.}
\]
Hence it is enough to check that the inequality  \eqref{EigenSpliting1_2} holds for some sufficiently small $\lambda>0$.

First of all, it is true that $\oRe \mu_1<0<\oRe \mu_4$ for all sufficiently small $\lambda>0$ by \eqref{SplitMuLam2} and the continuity.  Since  $\mu_2=\mu_3=0$ at $\lambda=0$, expanding $d_{\pm}(\mu)$ around $\mu=0$, we have that for sufficiently small $\lambda>0$, 
 \begin{subequations}\label{SplitMuLam411}
\begin{align} 
& 0<\lambda= d_+(\mu_2) = (c+\sqrt{1+K})\mu_2 \cdot  \left( 1  + O(|\mu_2|^2) \right),  \label{SplitMuLam41} \\
& 0<\lambda= d_-(\mu_3) = (c-\sqrt{1+K})\mu_3  \cdot \left( 1  + O(|\mu_3|^2) \right). \label{SplitMuLam42}
\end{align}
\end{subequations}
Since $c-\sqrt{1+K}=\veps>0$, \eqref{SplitMuLam411} implies that $\oRe\mu_2, \oRe \mu_3 >0$ for sufficiently small $\lambda>0$. Hence  \eqref{EigenSpliting1_2} is true.
\end{proof}

\begin{proof}[Proof of Lemma \ref{splitting mu lam weight}]

Since $d_\pm(-\mu)=-d_\pm(\mu)$ for $\mu \in \mathbb{C}\setminus\{(-\infty,-1]\cup [1,\infty)\}$, we obtain 
\begin{equation*}
\small\begin{split}
\oRe\left( d_\pm(ik-\etba)  - (c\pm \sqrt{1+K})(ik-\etba) -  \frac{\pm(ik-\etba)^3}{2\sqrt{1+K}}\right)
& = \oRe  \left(  \sum_{n \geq 5,\; n \text{ is odd}} \partial_\mu^n d_\pm(\mu)|_{\mu=0}(ik-\etba)^n \right) \\
& =:\mathcal{R}_1,
\end{split}
\end{equation*}
for all $\mu=ik-\etba$ close to the origin. Since $\partial_\mu^n d_\pm(0)$ is real-valued for all nonnegative integers $n$, there exists a constant $C_1>0$, uniform in sufficiently small $k,\etba,\veps$, such that 
\begin{equation}\label{SplMuWe}
|\mathcal{R}_1| \leq C_1 \etba|ik-\etba|^4 = C_1 ( k^4\etba + 2k^2\etba^3 + \etba^5).
\end{equation}
For any fixed $0<c_0<\sqrt{2\sqrt{1+K}}$, let  $\etba=c_0\veps^{1/2}$. Then we may choose sufficiently small $\veps_1(c_0)>0$ and $\delta>0$ (independent of $\veps$) such that for all $0<\veps<\veps_1$ and $|k| \leq \delta$, \eqref{SplMuWe} holds and
\begin{equation}
\begin{split}
\oRe d_-(ik-\etba) 
& = -\etba(c-\sqrt{1+K}) +\frac{ \etba^3}{2\sqrt{1+K}} - \frac{3k^2\etba}{2\sqrt{1+K}} + \mathcal{R}_1\\
& = -\veps \etba \left(1 - \frac{c_0^2}{2\sqrt{1+K}} \right) - \frac{3k^2\etba}{2\sqrt{1+K}} +  \mathcal{R}_1 \\
& < - \frac{\veps \etba}{2}\left(1 - \frac{c_0^2}{2\sqrt{1+K}} \right) - \frac{3k^2\etba}{4\sqrt{1+K}} \\
& <- \frac{\veps \etba}{2}\left(1 - \frac{c_0^2}{2\sqrt{1+K}} \right),
\end{split}
\end{equation}
where we have used \eqref{SplMuWe} in the first inequality. Hence, for all $0<\veps<\veps_1(c_0)$, we obtain that
\begin{equation}\label{SplitMuWei2}
\sup_{|k|\leq \delta}{\oRe d_-(ik-\etba) } \leq - \frac{\veps \etba}{2}\left(1 - \frac{c_0^2}{2\sqrt{1+K}} \right).
\end{equation}

We fix $\delta>0$. By the Taylor theorem and the chain rule, for all $k\in\mathbb{R}$ and sufficiently small $\etba>0$, we have that 
\begin{equation}\label{SplitMuWei3}
\begin{split}
\oRe d_\pm(ik-\etba) 
& = \oRe d_\pm(ik) + \partial_\beta \text{Re}\,d_{\pm} (ik-\etba) |_{\etba=0} \etba + \left.\partial_\etba^2 \text{Re}\,d_{\pm} (ik-\etba)\right|_{\etba=\etba_k} \etba^2 \\
& = \partial_k\omega_\pm(k) \etba + \oRe \partial_\mu^2 d_\pm(\mu)|_{\mu=ik-\etba_k} \etba^2
\end{split}
\end{equation}
for some $\etba_k\in(0,\etba)$. From \eqref{Diffdpm}, we have that 
\begin{equation}\label{Diffdpm4}
\sup_{k\in\mathbb{R},\beta\in[0,1/2]}\left| \partial_\mu^2d_{\pm}(\mu)|_{\mu=ik-\beta} \right| \leq C_{2}
\end{equation}
for some constant $C_{2}>0$ (independent of $\veps$).
From \eqref{SplitMuWei3} and \eqref{Diffdpm4},  there exists $\veps_2>0$  such that for all $0<\veps \leq \veps_2$, 
\begin{equation}\label{SplitMuWei4}
\begin{split}
\sup_{|k|\in[\delta,\infty)}\oRe d_-(ik-\etba) 
&  < \sup_{|k|\in[\delta,\infty)}\partial_k\omega_-(k) \beta + C_2\etba^2  \\
& < -C(\delta)\etba + C_2\etba^2 \\
& < -\frac{\veps\beta}{2}\left(1- \frac{c_0^2}{2\sqrt{\mathsf{V}}} \right).
\end{split}
\end{equation}
where we have used \eqref{Map_om} in the second inequality and $\beta=c_0\veps^{1/2}$ in the last inequality.
Combining \eqref{SplitMuWei2} and \eqref{SplitMuWei4}, we obtain that for all sufficiently small $\veps>0$,
\begin{equation}\label{Mu_0}
\sup_{k\in\mathbb{R}}\oRe d_\pm(ik-\etba) < - \frac{\veps\beta}{2}\left(1- \frac{c_0^2}{2\sqrt{\mathsf{V}}} \right),
\end{equation}
where the bound for $\oRe d_+$ easily follows from \eqref{SplitMuWei3} and \eqref{Diffdpm4} using
\eqref{map_omega1}. Therefore, we have
\begin{equation}\label{Mu_00}
\{d_\pm(ik-\beta): k \in \mathbb{R}\} \subset \mathbb{C}\setminus \Omega^\veps=\{\lambda \in \mathbb{C} :  \oRe \lambda < -\veps^{3/2}\eta(c_0)\}
\end{equation}
and finish the proof of the first assertion.

To prove the second assertion, we let $\mu=\mu'-\etba$ and consider the zeros of $d(\mu'-\beta)$. By \eqref{Mu_00}, any solutions $\mu'$ of $d_\pm(\mu'-\etba)=\lambda$ with $\lambda \in \Omega^\veps$ cannot lie in the imaginary axis. This implies that as long as $\lambda \in \Omega^\veps$, the number of zeros of $d(\mu'-\beta)$ lying on the left (resp. right) half-plane does not change. Hence, it suffices to consider the point $\lambda =0\in\Omega^\veps$. 

We choose $\veps_K'>0$ such that for all $\veps\in(0,\veps_K']$, \eqref{Mu_0} holds and the solutions of $d_\pm(\mu'-\beta)=0$ can be labeled such that
\[
\mu_1' = \etba -\sqrt{\frac{c^2-1-K}{c^2-K}}< 0 < \mu'_2 = \mu'_3=\etba <  \mu_4' = \etba +\sqrt{\frac{c^2-1-K}{c^2-K}}
\]
(recall that $\beta=c_0\veps^{1/2}$, $0<c_0<\sqrt{2\mathsf{V}}$, and $c=\sqrt{1+K}+\veps$). Now we conclude that  as long as $\lambda \in \Omega^\veps$, the zeros of $d(\mu)$ can be labeled so that
\[
\oRe \mu_1 +\beta < 0 < \oRe \mu_j+\beta \quad (j=2,3,4).
\]
 This finishes the proof of the second assertion.
\end{proof}

\section{Estimates uniform in eigenvalue parameters}\label{Sec5}
The first goal of this section is to prove Lemma \ref{Lem_Region2} and Lemma \ref{Lem_Region4} (the estimates of $D(\lambda,\veps)$ on the regions $\mathfrak{D}_2$ and $\mathfrak{D}_4$), and these lemmas will be proved in subsection \ref{Sec5.2}--\ref{Sec5.3}, respectively. The second goal is to show Proposition \ref{MainIngred}.\eqref{UnibdRes}, the uniform boundedness of the resolvent operator, which is a crucial ingredient for the asymptotic linear stability result. This will be covered in subsection \ref{Sec5.4}.

As a preliminary step,  we start with studying the asymptotic behaviors of characteristic roots for non-zero small $|\lambda|$ and large $|\lambda|$ in subsection \ref{Sec5.1}.
\subsection{Asymptotic behaviors of characteristic roots}\label{Sec5.1}
By employing a perturbation argument, we investigate the behavior of the roots of the characteristic polynomial $d(\mu)$ on the region $\mathfrak{D}_2$, in which $|\lambda|$ is small, but is non-zero. 
\begin{proposition}\label{prop_EigenApprox}
\begin{enumerate}
\item There exists $\delta_2'>0$ such that as long as $(\lambda,\veps,\delta$) satisfies 
\begin{equation}\label{EigAppC}
\veps^{3/2}\delta^{-1} < |\lambda| < \delta, \quad  0 \leq \veps \leq \veps_K, \quad 0<\delta\leq \delta_2',
\end{equation}
the solutions of $d_-(\mu)=\lambda$ can be labeled so that  they satisfy
\[
\mu_j = (-2\sqrt{1+K}\,\lambda)^{1/3}e^{2\pi i j/3} (1+ O(\delta^{2/3})  ), \quad (j=1,2,3),
\]
as $\delta \to 0$ uniformly in $\veps$.
\item There exists $\delta_2''>0$ such that as long as $(\lambda,\veps,\delta)$ satisfies 
\[
0<|\lambda|<\delta, \quad 0\leq \veps \leq \veps_K, \quad 0<\delta \leq  \delta_2'',
\]
the solution of $d_+(\mu)=\lambda$ (say $\mu_4$) satisfies 
\[
\mu_4 = \frac{\lambda}{c+\sqrt{1+K}}(1+ O(\delta^{2}))
\]
as $\delta \to 0$ uniformly in $\veps$.
\end{enumerate}
\end{proposition}
\begin{proof}
By expanding $d_-(\mu)$ around $\mu=0$ and using that $c=\sqrt{1+K}+\veps$, we see that $d_-(\mu)=\lambda$ is equivalent to  
\begin{equation}\label{SmallMu2}
\frac{\mu^3}{2\sqrt{1+K}} + \lambda =  \veps\mu +\mu^5\mathcal{R}_-(\mu).
\end{equation}
 Here, $\mathcal{R}_-(\mu)$ is analytic near $\mu=0$ and independent of $\veps$, and it satisfies $\mathcal{R}_-(\mu) =O(1)$ as $\mu$ tends to $0$.  We let $\widetilde{\mu}_j= (-2\sqrt{1+K}\lambda)^{1/3}e^{2\pi i j/3}$ for $j=1,2,3$, and then plug the Ansatz $\mu=\widetilde{\mu}_j(1-\mathcal{R}_j)^{1/3}$ into \eqref{SmallMu2}. Then we obtain 
\begin{equation}\label{SmallMu}
\mathcal{R}_j = \frac{\veps\widetilde{\mu}_j}{\lambda}(1-\mathcal{R}_j)^{1/3} + \frac{\widetilde{\mu}_j^5(1-\mathcal{R}_j)^{5/3}}{\lambda}\mathcal{R}_-\left( \widetilde{\mu}_j(1-\mathcal{R}_j)^{1/3} \right).
\end{equation}
We observe that as long as $\veps^{3/2}\delta^{-1}<|\lambda|<\delta$,
\begin{subequations}\label{SmallMu1}
\begin{align}
\frac{|\veps\widetilde{\mu}_j|}{|\lambda|} & = (2\sqrt{1+K})^{1/3}\veps|\lambda|^{-2/3} \leq (2\sqrt{1+K})^{1/3} \delta^{2/3}, \\
\frac{|\widetilde{\mu}_j^5|}{|\lambda|} & = (2\sqrt{1+K})^{5/3}|\lambda|^{2/3} \leq (2\sqrt{1+K})^{5/3}\delta^{2/3}, \\
|\widetilde{\mu}_j|&  \leq (2\sqrt{1+K})^{1/3}\delta^{1/3}.
\end{align}
\end{subequations}

Hence, by a fixed point argument, there exists small $\delta_2'>0$ such that the solution $\mathcal{R}_j$ to \eqref{SmallMu} exists as long as \eqref{EigAppC} holds. Furthermore, $\mathcal{R}_j = O(\delta^{2/3})$ as $\delta \to 0$ uniformly in $\veps$. This completes the proof of the first assertion.

We prove the second assertion. By expanding $d_+(\mu)$ around $\mu=0$, we obtain that
\begin{equation}\label{SmallMu4}
\lambda= \mu \left( ( c+ \sqrt{1+K} ) + \frac{\mu^2}{2\sqrt{1+K}} + \mu^4\mathcal{R}_+(\mu) \right)=:\mu \widetilde{\mathcal{R}}_+.
\end{equation}
Since $\mu$ is small as long as $|\lambda|<\delta$ for small $\delta$, we see that $\mu=O(\delta)$ as $\delta \to 0$ uniformly in $\veps$. By dividing \eqref{SmallMu4} by $\widetilde{\mathcal{R}}_+$, we finish the proof.
\end{proof}

Next we study the asymptotic behavior of the roots of characteristic polynomial $d(\mu)$ for large $|\lambda|$.  The proof  is based on a perturbation argument using Rouché's theorem in a similar fashion as Lemma 1.20 of \cite{PW}.
\begin{proposition}\label{EigenvalueLargelambda}
The roots of the characteristic polynomial $d(\mu)$  can be labelled so that they satisfy
\begin{subequations}
\begin{align*}
 \mu_1 & = -1 + O(|\lambda|^{-2}), & \mu_4 & = 1 + O(|\lambda|^{-2}),   \\ 
 \mu_2 & = \frac{c\lambda -\sqrt{K\lambda^2 -c^2 + K }}{c^2-K} + O(|\lambda|^{-3}),  & \mu_3  & = \frac{c\lambda + \sqrt{K \lambda^2-c^2 + K }}{c^2-K} + O(|\lambda|^{-3}),
\end{align*}
\end{subequations}
as $|\lambda| \to \infty$ uniformly in $\veps \in [0,\veps_K]$. 
\end{proposition}

\begin{proof}
We consider the decomposition $(c^2-K)d(\mu) =\widetilde{d}(\mu)  + \widetilde{d}_R(\mu)$, where
\[
\widetilde{d}(\mu)=\widetilde{d}(\mu,\lambda,\veps) =(\mu^2-1)((\lambda - c\mu)^2-K \mu^2 + 1) \quad \textrm{and} \quad \widetilde{d}_R(\mu) = 1.
\]
We note that $\widetilde{d}(\mu)$ has four simple zeros
\begin{equation}\label{Lgmu}
\widetilde{\mu}_1 = -1, \quad \widetilde{\mu}_4 = 1, \quad  \widetilde{\mu}_2 = \frac{c\lambda -\sqrt{K \lambda^2-c^2 + K }}{c^2-K}, \quad \widetilde{\mu}_3 = \frac{c\lambda +\sqrt{K \lambda^2-c^2 + K }}{c^2-K}
\end{equation}
for all $\lambda$ with sufficiently large $|\lambda|$ and $\veps\in[0,\veps_K]$. Since the derivative of $\widetilde{d}(\mu)$ in $\mu$ is
\[
\partial_\mu\widetilde{d}(\mu) = 2\mu((\lambda - c\mu)^2 - K \mu^2 +1) + 2(\mu^2-1)(-c\lambda + (c^2-K)\mu),
\]
we obtain that
\begin{subequations}
\begin{align*}
\partial_\mu\widetilde{d}(\widetilde{\mu}_1) & = -2((\lambda + c)^2 - K + 1), & \partial_\mu\widetilde{d}(\widetilde{\mu}_4)&  = 2((\lambda - c)^2 - K + 1),  \\ 
\partial_\mu\widetilde{d}(\widetilde{\mu}_2) & = -2(\widetilde{\mu}_2^2-1)\sqrt{K \lambda^2 -c^2+K }, & 
\partial_\mu\widetilde{d}(\widetilde{\mu}_3) & = 2(\widetilde{\mu}_3^2-1)\sqrt{K \lambda^2 -c^2+K }.
\end{align*}
\end{subequations}
Combined with \eqref{Lgmu}, this implies that  we may take some constant $\rho_0>1$ (independent of $\veps$ and $\lambda$) and positive functions $\rho_j(\lambda)$ (uniform in $\veps\in[0,\veps_K]$) such that the following hold:
\begin{equation}\label{rho_jAsymp}
\rho_j(\lambda)=O(|\lambda|^{-2}) \;\; \text{for}\; j=1,4, \quad \rho_j(\lambda)=O(|\lambda|^{-3}) \;\; \text{for}\; j=2,3, 
\end{equation}
as $|\lambda| \to \infty$ uniformly in $\veps\in[0,\veps_K]$, and for $j=1,2,3,4$, 
\begin{equation}\label{rho_jLowbound}
\rho_j(\lambda) > \rho_0 \frac{1}{|\partial_\mu\widetilde{d}(\widetilde{\mu}_j)|} 
\end{equation}
for all sufficiently large  $|\lambda|$. 

 By the Taylor theorem, we have that on each circle $|\mu - \widetilde{\mu}_j|=\rho_j$,
\[
\begin{split}
|\widetilde{d}(\mu)| 
& = |\partial_\mu\widetilde{d}(\widetilde{\mu}_j)||\mu-\widetilde{\mu}_j| \left|1+O(|\mu-\widetilde{\mu}_j|) \right| \\
& = \rho_j |\partial_\mu\widetilde{d}(\widetilde{\mu}_j)||1+O(\rho_j)|  \\
& > \rho_0 |1+O(\rho_j)| \\
& > 1 = |\widetilde{d}_R(\mu)|
\end{split}
\]
for all $\lambda$ with sufficiently large $|\lambda|$ and $\veps\in[0,\veps_K]$, where we have used  \eqref{rho_jLowbound} in the first inequality, \eqref{rho_jAsymp} and $\rho_0>1$ in the second inequality.  Now Rouché's theorem implies that for  $j=1,2,3,4$, there is exactly one simple root $\mu_j$ of $\widetilde{d}(\mu)+\widetilde{d}_R(\mu)$ (equivalently, of $d(\mu)$) such that $|\mu_j - \widetilde{\mu}_j| < \rho_j$, which finishes the proof combined with \eqref{Lgmu} and \eqref{rho_jAsymp}.
\end{proof}

\subsection{Estimates of the Evans function for non-zero small eigenvalue parameter}\label{Sec5.2}
In this subsection, we prove Lemma \ref{Lem_Region2}.\\ 

\textit{(Statement of Lemma \ref{Lem_Region2}) There exist constants $C_2,\delta_2,\veps_2>0$ such that for all $\veps \in [0,\veps_2]$ and $\delta \in (0,\delta_2]$,
\begin{equation*}
\sup_{\lambda\in \mathfrak{D}_2}|D(\lambda,\veps)-1| < C_2 \delta^{1/3}.
\end{equation*}
Here $\mathfrak{D}_2=\Omega^\veps \cap \{\lambda: \veps^{3/2}\delta^{-1} < |\lambda| < \delta\}$, and $C_2$ is independent of $\veps$ and $\delta$.}\\

We use the following lemma:
\begin{lemma}\label{Ch3PropProofRegD_2}
Assume that the coefficient matrix of the system  \eqref{ODE_LinEP} satisfies the hypotheses \textbf{H1}--\textbf{H4}, and that $A^\infty(\lambda,\veps)$ is diagonalizable. For the left eigenvector $\mathbf{w}_j$ and the right eigenvector $\mathbf{v}_j$  associated with the eigenvalue $\mu_j$ of $A^\infty$ satisfying $\mathbf{w}_j\mathbf{v}_j=1$, $(j=1,2,3,4)$, let  $V$ and $W$ be the matrices whose $j$-th column is $\mathbf{v}_j$ and $j$-th row is $\mathbf{w}_j$, respectively.  Let $R:= A(x,\lambda,\veps) - A^\infty(\lambda,\veps)$. Then, there exist positive constants  $\delta_0 \in (0,1)$ and $C_0$  such that  if $\int_{-\infty}^\infty|WRV|\,dx \leq \delta_0$, we have
\begin{equation*}
|D(\lambda,\veps)-1| \leq C_0 \int_{-\infty}^\infty|WRV|\,dx.
\end{equation*}
\end{lemma}
We omit the proof of Lemma \ref{Ch3PropProofRegD_2} since it can be easily shown following the proofs of Lemma 10.1 and Corollary 10.2 of \cite{PW2} (or the proofs of Proposition 1.17 and Corollary 1.18 of \cite{PW}). Instead, we will present a variation of the proof of Lemma \ref{Ch3PropProofRegD_2} in the next subsection.

\begin{proof}[Proof of Lemma \ref{Lem_Region2}]
We note that Proposition \ref{prop_EigenApprox} implies in particular  that  $\mu_j(\lambda,\veps)$ are all distinct for $\lambda \in \mathfrak{D}_2$. Hence the matrix $A_\infty(\lambda,\veps)$ is diagonalizable. For $\mathbf{v}_j$ and $\mathbf{w}_j$, given in  \eqref{LRVec_A} and \eqref{eigenvec_A}, we let $V$ and $W$ be the matrices whose $j$-th column is $\mathbf{v}_j$ and $j$-th row is $\mathbf{w}_j$, respectively, where $j=1,2,3,4$. We let $R = A(x,\lambda,\veps) - A^\infty(\lambda,\veps)$ (see \eqref{A_Decompose} and \eqref{A_Asymptotic}). We may apply Lemma \ref{Ch3PropProofRegD_2} provided that the desired estimate for $\textstyle\int W RV \, dx$ holds true. 

Let $R_{jk}$ be the $(j,k)$-entry of the matrix $R$. Using \eqref{pointestimateinX}, it is straightforward to check that
\begin{equation}\label{EstiRegD2_3}
|R_{jk}| \leq C \veps e^{-C\veps^{1/2}|x|}E_{jk}, \quad (j,k=1,2,3,4),
\end{equation}
where $C$ is some positive constant uniform in $\veps$, and $E_{jk}$ is the $(j,k)$-entry of the matrix 
\begin{equation}\label{E_Matrix}
E := \begin{pmatrix}
\veps^{1/2} + |\lambda| & \veps^{1/2} + |\lambda| & 0 & 1\\
\veps^{1/2} + |\lambda| & \veps^{1/2} + |\lambda| & 0 & 1 \\
0 & 0 & 0 & 0 \\
0 & 0 & 1 & 0
\end{pmatrix}.
\end{equation}
Let $v_{jl}$ and $w_{jl}$ be the $l$-th component of $\mathbf{v}_j$ and $\mathbf{w}_j$, respectively. Then, the $(j,k)$-entry of the matrix $WRV$ is given by
\begin{equation}\label{EstiRegD2_2}
(W RV)_{jk} = \sum_{l=1,2}w_{jl}(R_{l1}v_{k1} + R_{l2}v_{k2} + R_{l4}v_{k4}) + w_{j4}R_{43}v_{k3}, \quad (j,k=1,2,3,4).
\end{equation}
Combining \eqref{EstiRegD2_3}--\eqref{EstiRegD2_2}, one can obtain from  \eqref{LRVec_A} and  \eqref{eigenvec_A} that
\begin{equation}\label{EstiRegD2_1}
|(W RV)_{jk}| \leq C\veps e^{-C\veps^{1/2}|x|} G_{jk},  \quad (j,k=1,2,3,4),
\end{equation}
where 
\begin{equation}\label{Gjk}
\begin{split}
 & G_{jk}   := \\
 &   \frac{\left[(\veps^{1/2} + |\lambda|)\left(1 + \frac{|c\mu_k-\lambda|}{|\mu_k|}  \right) + \frac{|\mu_k|}{|1-\mu_k^2|} \right] \left[(|c|+1)\frac{|\lambda|}{|\mu_j|}+ |c^2-K| \right]|1-\mu_j^2|  + \frac{|\mu_j|}{|1-\mu_k^2|}}{|\boldsymbol{\pi}_j\mathbf{v}_j|}.
\end{split}
\end{equation}
Since 
\begin{equation}\label{Gjkint}
\veps^{1/2}\int_{-\infty}^\infty e^{-C\veps^{1/2}|x|}\,dx =\frac{2}{C} \text{ for any } \veps>0,
\end{equation}
it is enough to show that there are positive constants $C_2'$, $\veps_2$ and $\delta_2$ such that for all $\veps\in(0,\veps_2]$ and $\delta \in (0,\delta_2]$, there holds that as long as  $\lambda \in \mathfrak{D}_2$,
\begin{equation}\label{GjkConc}
\veps^{1/2}G_{jk} \leq C_2' \delta^{1/3}, \quad (j,k=1,2,3,4), 
\end{equation}
where $C_2'$ is uniform in $\delta$, $\veps$, and $\lambda$. Then Lemma \ref{Lem_Region2}  will follow from \eqref{EstiRegD2_1}, \eqref{Gjkint} and \eqref{GjkConc}, together with Lemma \ref{Ch3PropProofRegD_2}.
 
It is clear that as long as $\lambda \in \mathfrak{D}_2=\Omega^\veps \cap \{\lambda: \veps^{3/2}\delta^{-1} < |\lambda| < \delta\}$, the following holds:
\begin{subequations}\label{Reldellam}
\begin{align}
\veps^{1/2} & < |\lambda|^{1/3}\delta^{1/3}, \label{Reldellam1} \\
|\lambda| & < \delta. \label{Reldellam2}
\end{align}
\end{subequations}

From Proposition \ref{prop_EigenApprox}, we have
\begin{subequations}\label{mujsmal}
\begin{align}
 \text{for } j& =1,2,3, & \text{for } j& =4, \nonumber \\
 |\mu_j| & = (2\mathsf{V}|\lambda|)^{1/3}\left( 1+ o(1) \right), &  \mu_4 & = \lambda \left( \frac{1}{c+\mathsf{V}} + o(1) \right), \label{mujsmal1}  \\ 
|1-\mu_j^2| & = 1+o(1) , &  |1-\mu_4^2| & = 1+o(1), \\
\frac{|c\mu_j-\lambda|}{|\mu_j|} & = c \left( 1+o(1)\right),  &  \frac{|c\mu_4-\lambda|}{|\mu_4|} & = \mathsf{V} + o(1),
\end{align}
\end{subequations}
as $\delta \to 0$ uniformly in $\veps$. Applying \eqref{Reldellam1} and \eqref{mujsmal},  we obtain that for all $j,k=1,2,3,4$, 
\begin{equation}\label{G_jk3}
G_{jk} \leq \frac{C|\lambda|^{1/3}\left(1+o(1) \right)}{|\boldsymbol{\pi}_j\mathbf{v}_j|} \quad \text{as } \delta \to 0 \text{ uniformly in } \veps.
\end{equation} 

To estimate $|\boldsymbol{\pi}_j \mathbf{v}_j|^{-1}$, we recall that (see \eqref{eigenvec_A2})
\begin{equation*}
\boldsymbol{\pi}_j \mathbf{v}_j  =  \frac{\lambda^2(1-\mu_j^2)}{\mu_j^2} - (c^2-K)(1-\mu_j^2) +\frac{1+\mu_j^2}{1-\mu_j^2}, \quad (j=1,2,3,4).
\end{equation*}
Using \eqref{mujsmal1} for $j=4$, it is easy to check that
\begin{equation}\label{G_jk4}
\boldsymbol{\pi}_4 \mathbf{v}_4 = 2\mathsf{V}^2 + 2c\mathsf{V}  + o(1) \quad \text{as } \delta \to 0 \text{ uniformly in } \veps.
\end{equation}
On the other hand, estimating $|\boldsymbol{\pi}_j \mathbf{v}_j|^{-1}$ for $j=1,2,3$ is not so trivial. We first observe that 
\begin{equation}\label{G_jk6}
\begin{split}
\textstyle \boldsymbol{\pi}_j \mathbf{v}_j  
& = \textstyle \frac{\mu_j^2}{1-\mu_j^2} \left(\frac{\lambda^2(1-\mu_j^2)^2}{\mu_j^4} - (c^2-K)(1-\mu_j^2)^2\frac{1}{\mu_j^2} + \frac{1}{\mu_j^2} + 1 \right) \\
& = \textstyle \frac{\mu_j^2}{1-\mu_j^2} \left(\frac{\lambda^2(1-\mu_j^2)^2}{\mu_j^4} - \frac{1}{\mu_j^2}\left[ c^2-K-1 - 2(c^2-K)\mu_j^2 + (c^2-K)\mu_j^4 \right] + 1 \right) \\
& = \textstyle \frac{\mu_j^2}{1-\mu_j^2} \left(\frac{\lambda^2(1-\mu_j^2)^2}{\mu_j^4} - \frac{1}{\mu_j^2}\left[2\veps\sqrt{1+K} + \veps^2 - 2(c^2-K)\mu_j^2 + (c^2-K)\mu_j^4 \right] + 1 \right),
\end{split}
\end{equation}
where we have used $c=\sqrt{1+K} + \veps$ in the last equality. From \eqref{Reldellam1} and \eqref{mujsmal1}, we see that for $j=1,2,3$, 
\begin{equation}\label{G_jk7}
\frac{\veps}{|\mu_j^2|} < \frac{|\lambda|^{2/3}\delta^{2/3}}{|\mu_j^2|} = \frac{\delta^{2/3}}{(2\mathsf{V})^{2/3}(1+o(1))} = o(1) \quad \text{as } \delta \to 0 \text{ uniformly in } \veps.
\end{equation}
Applying \eqref{mujsmal} and \eqref{G_jk7}, we obtain from \eqref{G_jk6} that for $j=1,2,3$,
\begin{equation}\label{G_jk8}
|\boldsymbol{\pi}_j \mathbf{v}_j| = (2(c^2-K)+1)(2\mathsf{V}|\lambda|)^{2/3}(1+o(1))  \quad \text{as } \delta \to 0 \text{ uniformly in } \veps.
\end{equation}

Now combining \eqref{G_jk3} and \eqref{G_jk4}, and using \eqref{Reldellam2}, we have that  for $k=1,2,3,4$,
\begin{equation}\label{G_jk5}
\veps^{1/2}G_{4k} \leq \veps^{1/2}\frac{C \delta^{1/3}(1+o(1))}{2\mathsf{V}^2 + 2cV + o(1)} \quad \text{as } \delta \to 0 \text{ uniformly in } \veps.
\end{equation}
Combining \eqref{G_jk3}, \eqref{G_jk8}, and using \eqref{Reldellam1}, we have that  for $j=1,2,3,$ and $k=1,2,3,4$,
\begin{equation}\label{G_jk9}
\begin{split}
\veps^{1/2}G_{jk} 
& \leq  \frac{C\veps^{1/2}|\lambda|^{1/3}(1+o(1))}{(2(c^2-K)+1)(2\mathsf{V}|\lambda|)^{2/3}(1+o(1))} \\
&  \leq \frac{C\delta^{1/3}(1+o(1))}{(2(c^2-K)+1)(2\mathsf{V})^{2/3}(1+o(1))} \quad \text{as } \delta \to 0 \text{ uniformly in } \veps.
\end{split}
\end{equation}
 
From \eqref{G_jk5} and \eqref{G_jk9}, we may choose sufficiently small $\delta_2$ and $\veps_2$ such that \eqref{GjkConc} holds for some positive constant $C_2'$, which is uniform in $\veps$, $\delta$, and $\lambda$. This finishes the proof.
\end{proof}

\subsection{Estimates of the Evans function for large eigenvalue parameter}\label{Sec5.3}
The goal of this subsection is to prove Lemma \ref{Lem_Region4}.\\

\textit{(Statement of Lemma \ref{Lem_Region4}) There exist constants $C_4,\delta_4,\veps_4>0$ such that for all $\veps \in [0,\veps_4]$ and $\delta \in (0,\delta_4]$,
\begin{equation*}
\sup_{\lambda\in \mathfrak{D}_4}|D(\lambda,\veps)-1| < C_4 \veps^{1/2}.
\end{equation*}
Here $\mathfrak{D}_4=\Omega^\veps \cap \{\lambda: \delta^{-1} < |\lambda|\}$, and $C_4$ is independent of $\veps$ and $\delta$. }\\

From Proposition \ref{EigenvalueLargelambda}, we have 
\begin{subequations}\label{mujlarge}
\begin{alignts}
 \text{for } j& =1,4, & \text{for } j& =2,3, \nonumber \\
\mu_j & = (-1)^{j} + O(|\lambda|^{-2}), &  \mu_j & = \frac{\lambda}{c+ (-1)^j\sqrt{K}}\left(1+ O(|\lambda|^{-2})\right),  \label{mujlarge1} \\
 \frac{c\mu_j-\lambda}{\mu_j} & = (-1)^{j+1}\lambda\left(1+O(|\lambda|^{-1})\right),  &   \frac{c\mu_j-\lambda}{\mu_j} & = (-1)^{j+1}\sqrt{K}\left(1+O(|\lambda|^{-2})\right), \label{mujlarge2}\\
1-\mu_j^2 & = \lambda^{-2}\left( 1+O(|\lambda|^{-1}) \right) , &  1-\mu_j^2 & = \frac{-\lambda^2}{(c+(-1)^j\sqrt{K})^2}\left(1+O(|\lambda|^{-2})\right),  \label{mujlarge3} \\
 \boldsymbol{\pi}_j  \mathbf{v}_j  & =2\lambda^2\left(1+O(|\lambda|^{-1}) \right), & \boldsymbol{\pi}_j  \mathbf{v}_j & =  -\lambda^2 \left( \frac{2(-1)^{j}\sqrt{K}}{(c+(-1)^{j}\sqrt{K})}  + O(|\lambda|^{-1}) \right) \label{muj large4}
\end{alignts}
\end{subequations}
as $|\lambda| \to \infty$ uniformly in $\veps$. Here to obtain \eqref{mujlarge3} for $j=1,4$, we have used that $d(\mu_j)=0$, that is, 
\[
\frac{1}{1-\mu_j^2} = \frac{(\lambda-c\mu_j)^2-K\mu_j^2}{\mu_j^2}
\] 
rather than use \eqref{mujlarge1} directly. 

Although it is true from \eqref{mujlarge1} that $\mu_j$ are all distinct as long as $\lambda\in\mathfrak{D}_4$ for all sufficiently small $\delta>0$ and $\veps>0$, Lemma \ref{Ch3PropProofRegD_2} cannot be directly applied as the previous analysis on the domain $\mathfrak{D}_2$. Indeed, by applying \eqref{mujlarge} to \eqref{Gjk}, we have for large $|\lambda|$,
\begin{subequations}
\begin{align*} 
|G_{jk}| & \leq C,  & (j=1,4,\, k=1,4), \quad\quad\quad\quad  & |G_{jk}|  \leq C|\lambda|^{-2}, &  (j=1,4,\, k=2,3), \\
|G_{jk}| &\leq C|\lambda|, &  (j=2,3, \, k=2,3),   \quad\quad\quad\quad &  |G_{jk}|  \leq C|\lambda|^2, &  (j=2,3,\, k=1,4).
\end{align*}
\end{subequations}
Hence, we need a more delicate approach to prove Lemma \ref{Lem_Region4}. We remark that the linear growth bound of $G_{jk}$ for $(j=2,3,\, k=2,3)$ is due to $\lambda$, the coefficient of $A_2$ (see \eqref{A_Decompose}). The main contribution of the bounds for $(j=1,4, \, k=2,3)$ and $(j=2,3,\, k=1,4)$ comes from the term $\textstyle\frac{|1-\mu_j^2|}{|1-\mu_k^2|}$ (see \eqref{Gjk} and \eqref{mujlarge3}). These terms arise due to the transport term and the Poisson equation of the linearized Euler-Poisson system. Taking these into account, we will consider a suitable compensating factor so that the uniform bounds corresponding to the cases of $(j=1,4,\, k=2,3)$ and $(j=2,3,\, k=1,4)$ are obtained, and at the same time, the term which leads the linear growth in $\lambda$ can be controlled.

We first present some preliminary steps, and then finish the proof of Lemma \ref{Lem_Region4} by employing a variation of the proof of Lemma \ref{Ch3PropProofRegD_2} as promised in the last subsection. \\

\textit{Step 1: Decomposition and compensating factor.} We write
\[
A(x,\lambda,\veps) - A^\infty(\lambda,\veps) = \lambda R^{(1)}(x,\veps) + R^{(2)}(x,\veps),
\]
where (see \eqref{A_Decompose})
\begin{equation}\label{R12}
\begin{split}
 R^{(1)} =   R^{(1)}(x,\veps)
& :=  A_2(x,\veps) - \lim_{|x|\to \infty} A_2(x,\veps), \\
R^{(2)}=R^{(2)}(x,\veps)& :=A_1(x,\veps)-\lim_{|x|\to \infty} A_1(x,\veps). 
\end{split}
\end{equation}
 Then using \eqref{pointestimateinX}, we see that the $(j,k)$-entry of $R^{(2)}$ satisfies
\begin{equation}\label{R2Esti}
|R^{(2)}_{jk}| \leq C \veps e^{-C\veps^{1/2}|x|}E^{(2)}_{jk},
\end{equation}
where $E^{(2)}_{jk}$ is the $(j,k)$-entry of $E^{(2)}$ defined by  (compare \eqref{E2_Matrix} with \eqref{E_Matrix})
\begin{equation}\label{E2_Matrix}
E^{(2)} := \begin{pmatrix}
\veps^{1/2}  & \veps^{1/2}   & 0 & 1\\
\veps^{1/2} & \veps^{1/2}  & 0 & 1 \\
0 & 0 & 0 & 0 \\
0 & 0 & 1 & 0
\end{pmatrix}.
\end{equation}

 We define the factor
\begin{equation}\label{mj}
m_j=m_j(\lambda,\veps) := (c+(-1)^j\sqrt{K}) |1-\mu_j^2|^{\frac{1}{2}}, \quad (j=1,2,3,4).
\end{equation}
Using \eqref{mujlarge3}, it is straightforward to check that 
\begin{subequations}\label{Factor}
\begin{align} 
\frac{m_k}{m_j}  & =O\left(1\right) ,  \; (j,k=1,4), &    \frac{m_k}{m_j} & = O\left(|\lambda|^2\right), \; (j=1,4,\, k=2,3), \\
\frac{m_k}{m_j} & = 1+O\left(|\lambda|^{-1}\right),  \;  (j,k=2,3), &      \frac{m_k}{m_j} & =O\left(|\lambda|^{-2}\right), \; (j=2,3,\, k=1,4),
\end{align}
\end{subequations}
as $|\lambda| \to \infty$ uniformly in $\veps$. We note that the leading order term of $m_k/m_j$ for $(j=2,3, \, k=2,3)$ is 1.\footnote{This is why $(c+(-1)^j\sqrt{K})$ is considered in \eqref{mj}. Its purpose is to obtain the decomposition \eqref{DecompW0R1V0}, where the matrix $S_1$ is symmetric.} For $V$ and $W$, the matrices whose $j$-th column is $\mathbf{v}_j$ and $j$-th row is $\mathbf{w}_j$, respectively, we define the matrices $\widetilde{V}$ and $\widetilde{W}$ as follows: 
\begin{equation}\label{V0W0}
\widetilde{V}  := V \mathrm{diag}(m_1,m_2,m_3,m_4), \quad
 \widetilde{W}  := \mathrm{diag}(m_1,m_2,m_3,m_4)^{-1}W.
\end{equation}
It is clear that
\begin{equation}\label{W0V0}
\widetilde{W}\widetilde{V}=I, \quad \widetilde{W} A^\infty  \widetilde{V} = \text{diag}(\mu_1,\mu_2,\mu_3,\mu_4),
\end{equation}
and 
\begin{equation}\label{DecompW0(A-Ainf)V0}
\widetilde{W}\left(A - A^\infty \right)\widetilde{V} = \lambda \widetilde{W} R^{(1)}\widetilde{V} + \widetilde{W}R^{(2)}\widetilde{V}.
\end{equation}

\textit{Step2: Estimate of $\widetilde{W} R^{(2)} \widetilde{V}$.} Let $v_{jl}$ and $w_{jl}$ be the $l$-th component of $\mathbf{v}_j$ and $\mathbf{w}_j$, respectively. From  \eqref{V0W0}, the $(j,k)$-entry of the matrix $\widetilde{W} R^{(2)} \widetilde{V}$ is given by
\[
\begin{split}
(\widetilde{W} R^{(2)} \widetilde{V})_{jk} 
& = \frac{m_k}{m_j}  \left(  \sum_{l=1,2}w_{jl}\left(R^{(2)}_{l1}v_{k1} + R^{(2)}_{l2}v_{k2} + R^{(2)}_{l4}v_{k4} \right) + w_{j4}R^{(2)}_{43}v_{k3} \right).
\end{split}
\]
From \eqref{LRVec_A}, \eqref{eigenvec_A}, \eqref{R2Esti} and \eqref{E2_Matrix}, we obtain 
\begin{equation}\label{BdW0R2V0_1}
|(\widetilde{W} R^{(2)} \widetilde{V})_{jk}| \leq C\veps e^{-C\veps^{1/2}|x|} G^{(2)}_{jk},
\end{equation}
where (compare $G^{(2)}_{jk}$ with $G_{jk}$ in \eqref{Gjk})
\[
\begin{split}
G^{(2)}_{jk}
& :=  \frac{m_k}{m_j}\frac{1}{|\boldsymbol{\pi}_j \mathbf{v}_j|} \times  \\
& \left\lbrace\left[\veps^{1/2}\left(1 + \frac{|c\mu_k-\lambda|}{|\mu_k|}  \right) + \frac{|\mu_k|}{|1-\mu_k^2|} \right] \left[(|c|+1)\frac{|\lambda|}{|\mu_j|}+ |c^2-K| \right]|1-\mu_j^2| 
 + \frac{|\mu_j|}{|1-\mu_k^2|} \right\rbrace.
\end{split}
\]
Using \eqref{mujlarge} and \eqref{Factor}, it is straightforward to check that 
\begin{equation}\label{bdG2jk}
|G^{(2)}_{jk}| \leq C \quad \text{for} \; j,k=1,2,3,4,
\end{equation}
uniformly in $\veps$ and $|\lambda| \geq \delta^{-1}$ for sufficiently small $\delta$. From \eqref{BdW0R2V0_1} and \eqref{bdG2jk}, we obtain the bound 
\begin{equation}\label{bdW0R2V0}
|\widetilde{W} R^{(2)}(x) \widetilde{V}| \leq C \veps e^{-C\veps^{1/2}|x|},
\end{equation}
where the constant $C$ is uniform in $\veps$ and $\lambda$ with $|\lambda| \geq \delta^{-1}$ for sufficiently small $\delta$.\\

\textit{Step 3: Estimate of $\lambda \widetilde{W} R^{(1)} \widetilde{V}$.} The $(j,k)$-entry of the matrix $\lambda \widetilde{W} R^{(1)} \widetilde{V}$ is given by 
\[
\begin{split}
(\lambda \widetilde{W} R^{(1)} \widetilde{V})_{jk} 
& = \frac{m_k}{m_j} \left(  \sum_{l=1,2}\lambda w_{jl} \left[ R^{(1)}_{l1}v_{k1} + R^{(1)}_{l2}v_{k2}\right] \right) \\
& = \frac{m_k}{m_j} \frac{\lambda}{\boldsymbol{\pi}_j\mathbf{v}_j}\left(\frac{1-\mu_j^2}{\mu_j} \right)\\
& \times \left[(c\lambda-\mu_j(c^2-K))\left(R^{(1)}_{11} + \frac{c\mu_k-\lambda}{\mu_k}R^{(1)}_{12}\right)  -\lambda \left( R^{(1)}_{21} + \frac{c\mu_k-\lambda}{\mu_k}R^{(1)}_{22} \right) \right],
\end{split}
\]
where $R^{(1)}_{jk}$ is the $(j,k)$-entry of $R^{(1)}$. On the other hand,  we have from \eqref{mujlarge} that
\begin{subequations}\label{mujlarge9}
\begin{alignts}  \text{for } j& =1,4, & \text{for } j& =2,3, \nonumber \\
\frac{\lambda(1-\mu_j^2)}{\boldsymbol{\pi}_j \mathbf{v}_j\mu_j}  & = O\left( |\lambda|^{-3} \right) , &   \frac{\lambda(1-\mu_j^2)}{\boldsymbol{\pi}_j \mathbf{v}_j\mu_j} & = \frac{1}{2(-1)^{j}\sqrt{K}}\left(1+O\left(|\lambda|^{-1}\right)\right), \\
c\lambda-\mu_j(c^2-K) & = O(|\lambda|),   &   c\lambda-\mu_j(c^2-K) & = (-1)^{j} \sqrt{K}\lambda\left( 1+O\left(|\lambda|^{-1}\right)\right).
\end{alignts}
\end{subequations}
Using \eqref{mujlarge3}, \eqref{Factor} and \eqref{mujlarge9}, a direct calculation yields a decomposition
\begin{equation}\label{DecompW0R1V0}
\lambda \widetilde{W}R^{(1)} \widetilde{V} = \frac{\lambda}{2\sqrt{K}} S_1 + \widetilde{R}^{(1)},
\end{equation}
where $S_1=S_1(x,\veps)$ is a symmetric matrix defined by 
\begin{equation}\label{S1sym}
S_1:=
\begin{pmatrix}
0 & 0 & 0 & 0 \\
0 & 2\sqrt{K} R^{(1)}_{11}-K R^{(1)}_{12}-R^{(1)}_{21}  & K R^{(1)}_{12} - R^{(1)}_{21} & 0  \\
0 & K R^{(1)}_{12} - R^{(1)}_{21}  & 2\sqrt{K} R^{(1)}_{11}+K R^{(1)}_{12}+ R^{(1)}_{21} & 0 \\
0 & 0 & 0 & 0 
\end{pmatrix}
\end{equation}
and $\widetilde{R}^{(1)}$ is a matrix whose the $(j,k)$-entry satisfies 
\begin{equation}\label{bdR1jk}
|(\widetilde{R}^{(1)})_{jk}|\leq C \veps e^{-C\veps^{1/2}|x|},
\end{equation}
where the constant $C>0$ is uniform in small $\veps>0$ and $|\lambda| \geq \delta^{-1}$ for sufficiently small $\delta$.  The symmetric matrix $S_1$ is positive semi-definite (or non-negative) for all sufficiently small $\veps$ (see Appendix \ref{NonnegaS}). Now we are ready to prove Lemma \ref{Lem_Region4}.

\begin{proof}[Proof of Lemma \ref{Lem_Region4}] From   \eqref{W0V0}, \eqref{DecompW0(A-Ainf)V0}, and \eqref{DecompW0R1V0}, we have 
\[
\begin{split}
\widetilde{W}(A - \mu_1I) \widetilde{V}
& = \widetilde{W}(A^\infty - \mu_1I) \widetilde{V} + \widetilde{W}(A -A^\infty) \widetilde{V}  \\
& = \widetilde{B} + \widetilde{F},
\end{split}
\]
where $\widetilde{F}(x,\lambda,\veps):=\widetilde{R}^{(1)} + \widetilde{W}R^{(2)}\widetilde{V}$ and 
\begin{equation}\label{Btilde}
\widetilde{B}(x,\lambda,\veps):= \text{diag}(0,\mu_2-\mu_1,\mu_3-\mu_1,\mu_4-\mu_1) + \frac{\lambda}{2\sqrt{K}}S_1.
\end{equation}
Using \eqref{bdW0R2V0} and \eqref{bdR1jk}, we see that 
\begin{equation}\label{IntFtilde}
\int_{-\infty}^\infty|\widetilde{F}(x,\lambda,\veps)|\,dx \leq C\veps^{1/2},
\end{equation}
where the constant $C$ is uniform in $\veps$ and $|\lambda| \geq \delta^{-1}$.

We let 
\[
\widetilde{\mathbf{e}}_1 :=\left( m_1^{-1}, 0, 0 , 0\right)^T, \quad \widetilde{\mathbf{e}}_1^\ast :=\left(m_1, 0, 0 , 0\right).
\]
Changing variables $\widetilde{\mathbf{y}}(x)=e^{-\mu_1x}\widetilde{W}\mathbf{y}(x) - \widetilde{\mathbf{e}}_1$, we have 
\begin{equation}\label{ytildeEqdiag}
\frac{d\widetilde{\mathbf{y}}}{dx} = \widetilde{B}(x;)\widetilde{\mathbf{y}} + \widetilde{F}(x;)(\widetilde{\mathbf{e}}_1 + \widetilde{\mathbf{y}}).
\end{equation}
With a particular choice of $\mathbf{y}^+$, we know that $\widetilde{\mathbf{y}}^+(x):= e^{-\mu_1x}\widetilde{W}\mathbf{y}^+(x) - \widetilde{\mathbf{e}}_1$ is a solution of \eqref{ytildeEqdiag} satisfying $\lim_{x \to +\infty} \widetilde{\mathbf{y}}^+(x) = 0$ from the definition of $\widetilde{W}$ and \eqref{ODEsol1}. 

Let $\Xi(x;s) $ be the fundamental matrix of the ODE system with the coefficient matrix  \eqref{Btilde}. In Lemma \ref{Lem_BtildeEq}, we will show that  $|\Xi(x;s)|\leq 1$ for $x\leq s$. Using this fact and \eqref{IntFtilde}, one may apply an iteration argument to show that there is a solution $\widetilde{\mathbf{y}}_\sharp^+$ of \eqref{ytildeEqdiag}    satisfying $\lim_{x \to +\infty}\widetilde{\mathbf{y}}_\sharp^+(x)=0$ as a fixed point of the bounded linear operator $\widetilde{\mathcal{T}}$ on $C_b([0,\infty))$  defined by
\[
\big(\widetilde{\mathcal{T}}\,\widetilde{\mathbf{y}}\big)(x) := -\int_x^{\infty} \Xi(x;s) \big[ \widetilde{F}(s)(\widetilde{\mathbf{e}}_1 + \widetilde{\mathbf{y}}(s)) \big] \,ds.
\]
Since $\widetilde{\mathbf{y}}_\sharp^+$ and $\widetilde{\mathbf{y}}^+$ tend to $0$ as $x\to +\infty$, we have $\widetilde{\mathbf{y}}_\sharp^+=\widetilde{\mathbf{y}}^+$. To see this, one may directly use Proposition \ref{Prop_ResultPW} by considering the asymptotic behavior of $\mathbf{y}_\sharp^+$ and $\mathbf{y}^+$ as $x\to+\infty$, where $\mathbf{y}_\sharp^+$, defined by $\widetilde{\mathbf{y}}_\sharp^+=e^{-\mu_1x}\widetilde{W}\mathbf{y}_\sharp^+ - \widetilde{\mathbf{e}}_1$, is a solution of the ODE \eqref{ODE_LinEP}. Indeed, we have $e^{-\mu_1 x}\mathbf{y}_\sharp^+(x) \to \widetilde{V}\widetilde{\mathbf{e}}_1 = \mathbf{v}_1$ and $ e^{-\mu_1 x}\mathbf{y}^+(x) \to \mathbf{v}_1$ as $x\to+\infty$. Hence, from the definition of $\widetilde{\mathcal{T}}$ and \eqref{IntFtilde},  we obtain
\[
\sup_{x\in[0,\infty)}|\widetilde{\mathbf{y}}^+(x)| \leq C \veps^{1/2}.
\]

In a similar fashion, one can obtain that
\[
\sup_{x\in(-\infty,0]}|\widetilde{\mathbf{z}}^-(x)| \leq C \veps^{1/2},
\]
where $\widetilde{\mathbf{z}}^-(x):=\mathbf{z}^-(x)e^{\mu_1x}\widetilde{V} -\widetilde{\mathbf{e}}_1^\ast$.
Since $D(\lambda,\veps) = \mathbf{z}^-\mathbf{y}^+ = (\widetilde{\mathbf{z}}^- + \widetilde{\mathbf{e}}_1^\ast)(\widetilde{\mathbf{y}}^+ + \widetilde{\mathbf{e}}_1) $, these two estimates yield that 
\[
|D(\lambda,\veps) -1| \leq C\veps^{1/2}.
\]
This finishes the proof.
\end{proof}

\subsection{Uniform Resolvent estimates}\label{Sec5.4}
We aim to prove the uniform boundedness of the resolvent operator, Proposition \ref{MainIngred}.\eqref{UnibdRes}. The following proposition is its restatement.
\begin{proposition}\label{Uniresol}
Consider the operator $\mathcal{L}: (L_\etba^2)^2 \to (L_\etba^2)^2$ with dense domain $(H_\etba^1)^2$. For any fixed $c_0\in(0, \sqrt{ 2{\sqrt{1+K}}/3} )$, let $\etba=c_0\veps^{1/2}$. Then there exists $\veps_0>0$ such that for each $\veps \in (0,\veps_0]$, the resolvent operator $(\lambda-\mc{L})^{-1}$ is uniformly bounded on $\oRe \lambda \geq 0$, outside any small neighborhood of the origin.
\end{proposition}
\begin{proof}
It is enough to show that 
\[
\sup_{\lambda \in \mathfrak{D}_4,\oRe \lambda  \geq 0}\|(\lambda - \mc{L})^{-1}\|_{(L_\beta^2)^2} \leq M
\]
for each small $\veps$ since  the resolvent operator is analytic in $\lambda$ on the resolvent set. (Recall that $\Omega^\veps\setminus\{0\}$  is a subset of the resolvent set. See Proposition \ref{Clas_Spec}.)

For given $e^{{\etba} x}(f_1,f_2)^T \in (L^2)^2$, we consider the solution $\mathbf{y}^\etba=e^{{\etba} x}(\dt{n},\dt{u},\dt{\phi},\dt{\psi})^T$ to the inhomogeneous ODE system 
\begin{equation}\label{inhomoWei}
(\partial_x-\etba) \mathbf{y}^{\etba} = A(x,\lambda,\veps)\mathbf{y}^{\etba} + \mathrm{diag}(L^{-1},1,1)\mathbf{f}^\etba, \quad \mathbf{f}^{\etba} := e^{{\etba} x}(f_1, f_2,0,0)^T
\end{equation}
on the region $\mathfrak{D}_4\cap \{\lambda : \oRe \lambda \geq 0\}$. Indeed, the solution $\mathbf{y}^\beta \in (L^2)^2$ exists. First of all, we see that 
\[
P(\lambda,\veps):=\frac{1}{D(\lambda,\veps)}\mathbf{y}^+(x,\lambda,\veps)\mathbf{z}^-(x,\lambda,\veps)|_{x=0}
\]
is the projection onto the space of initial conditions at $x=0$ of solutions to the system \eqref{ODE_LinEP} satisfying $O(e^{\mu_1 x})$ as $x \to +\infty$. The range of the complementary projection $I-P$ is the space of initial conditions at $x=0$ of solutions to \eqref{ODE_LinEP} satisfying $O(e^{(\mu_\ast-\theta) x})$ as $x \to -\infty$. We let $\Phi(x)=\Phi(x,\lambda,\veps)$ be the fundamental matrix solution of \eqref{ODE_LinEP} satisfying $\Phi(0)=I$. Then the Green function $G^\etba(x,x')=G^\etba(x,x',\lambda,\veps)$ of the system \eqref{inhomoWei} with $\mathbf{f}^\beta=0$
is given by 
\[
G^\etba(x,x') = \left\{ \begin{array}{l l}
e^{{\etba} (x-x')}\Phi(x)P\Phi(x')^{-1} & x>x', \\
-e^{{\etba} (x-x')}\Phi(x)(I-P)\Phi(x')^{-1} & x'>x,
\end{array}
\right.
\]
It is easy to check that $G^{\etba}$ satisfies
\[
(\partial_x-\etba) G^{\etba} = A(x,\lambda,\veps)G^{\etba} \quad \text{for } x \neq x', \quad G^\etba(x'+0,x') - G^\etba(x'-0,x')=I.
\]

Our goal is to  show that there is a constant $M'>0$, \emph{independent of} $\lambda \in \mathfrak{D}_4\cap \{\lambda : \oRe \lambda  \geq 0\}$, such that for $j,k=1,2$,
\begin{equation}\label{Expod2}
\sup_{x\in\mathbb{R}}\int_{-\infty}^\infty|(G^\etba)_{jk}(x,x')|\,dx' + \sup_{x'\in\mathbb{R}}\int_{-\infty}^\infty|(G^\etba)_{jk}(x,x')|\,dx < M'.
\end{equation}
Then, by the generalized Young's inequality, we  obtain the desired uniform (in $\lambda$) bound
\[
\|(\dt{n},\dt{u})^T\|_{(L_\etba^2)^2}=\|e^{{\etba} x}(\dt{n},\dt{u})^T\|_{(L^2)^2} \leq 4M'\|e^{{\etba} x}L^{-1}(f_1,f_2)^T\|_{(L^2)^2} \leq M\|(f_1,f_2)^T\|_{(L_\etba^2)^2}.
\]


\textit{Step 1:} To make the situation simpler, we first consider some diagonalization. We let  $\widetilde{\mathbf{y}}_\etba:=\widetilde{V}^{-1}\mathbf{y}_\beta$, where $\widetilde{V}$ is defined in \eqref{V0W0}. Then the system  \eqref{inhomoWei} with $\mathbf{f}_\beta=0$ becomes  
\begin{equation}\label{Ch3ODEDiagonal}
\begin{split}
\partial_x \widetilde{\mathbf{y}}_\etba =  \text{diag}({\etba} + \mu_i)\widetilde{\mathbf{y}}_\etba + \widetilde{W}(A-A^\infty)\widetilde{V} \widetilde{\mathbf{y}}_\etba.
\end{split}
\end{equation}
The Green function $ \widetilde{G^\etba}(x,x')$ of \eqref{Ch3ODEDiagonal} is then given by 
\begin{equation}\label{Expod3}
\begin{split}
\widetilde{G^\etba}(x,x') 
& =\widetilde{W}G^\etba(x,x')\widetilde{V}  \\
& = \left\{ \begin{array}{l l}
\widetilde{\Phi}^\etba(x) P^\etba(\widetilde{\Phi}^\etba)^{-1}(x') & x>x', \\
-\widetilde{\Phi}^\etba(x) (I-P^\etba)(\widetilde{\Phi}^\etba)^{-1}(x') & x'>x,
\end{array}
\right.
\end{split}
\end{equation}
where $\widetilde{\Phi}^\etba(x):=\widetilde{W}e^{{\etba} x}\Phi(x)\widetilde{V}$ is the fundamental solution of \eqref{Ch3ODEDiagonal} with $\Phi^\etba(0)=I$ and $P^\etba:=\widetilde{W}P\widetilde{V}=(P^\etba)^2$ is a projection. 

On the other hand, one may check that by term by term computation using \eqref{mujlarge} and \eqref{Factor},
\[
|(\widetilde{V})_{jl}| |(\widetilde{W})_{mk}| \leq C \quad \text{for }j,k=1,2 \text{ and } l,m=1,2,3,4
\]
uniformly in $\lambda\in\mathfrak{D}_4$.  Hence, for $j,k=1,2$, we obtain that
\begin{equation}\label{Expod1}
|(G^\etba)_{jk}| = | (\widetilde{V}\widetilde{G^\etba}\widetilde{W})_{jk} | \leq 64C \sum_{l,m=1,2,3,4} | (\widetilde{G^\etba})_{lm}|.
\end{equation}
In the next step, we show that there exists constants $\veps_0, C_0, \alpha_0>0$, independent of $\lambda$, such that for all $(\veps,\lambda)\in (0,\veps_0]\times\mathfrak{D}_4$ with $\oRe \lambda \geq 0$, there holds that
\begin{subequations}\label{Ch3_Expodichoeta}
\begin{align}
 |\widetilde{\Phi}^\etba(x) P^\etba(\widetilde{\Phi}^\etba)^{-1}(x')| & \leq C_0 e^{-\alpha_0(x-x')}, \quad x > x', \\
 |\widetilde{\Phi}^\etba(x) (I-P^\etba)(\widetilde{\Phi}^\etba)^{-1}(x')| & \leq C_0 e^{-\alpha_0(x'-x)},  \quad x' > x.
\end{align}
\end{subequations}
Then \eqref{Expod2} will follow from \eqref{Expod3}--\eqref{Ch3_Expodichoeta}.

\textit{Step 2:} We recall that from \eqref{DecompW0(A-Ainf)V0} and \eqref{DecompW0R1V0},
\[
\widetilde{W}(A-A^\infty)\widetilde{V} = \frac{\lambda}{2\sqrt{K}}S_1 + \widetilde{R}^{(1)} + \widetilde{W} R^{(2)}\widetilde{V}.
\]
Here, from \eqref{bdW0R2V0} and \eqref{bdR1jk}, there holds that
\begin{equation}\label{Expodi5}
|\widetilde{R}^{(1)} + \widetilde{W} R^{(2)}\widetilde{V}| \leq C \veps e^{-C\veps^{1/2}|x|}
\end{equation}
uniformly in $\lambda \in \mathfrak{D}_4$. 

Now we consider the simpler equation 
\begin{equation}\label{BtweiEq}
\frac{d\bm{a}}{dx} = \mathrm{diag}(\mu_1+\beta,\mu_2+\beta,\mu_3+\beta,\mu_4+\beta)\bm{a} + \frac{\lambda}{2\sqrt{K}}S_1\bm{a}
\end{equation}
and show that \eqref{BtweiEq} has an exponential dichotomy on $\mb{R}$ with some uniform constants.  Then, together with the estimate \eqref{Expodi5}, the roughness of exponential dichotomies (Proposition 1, Chapter 4 of \cite{Coppel2}) implies that  the system \eqref{Ch3ODEDiagonal} has an exponential dichotomy \eqref{Ch3_Expodichoeta} on $\mb{R}$ with some uniform constants $C_0$ and $\alpha_0$ for all sufficiently small $\veps>0$. We remark that the projection of the exponential dichotomy \eqref{Ch3_Expodichoeta} on $\mathbb{R}$ is unique (see the last paragraph of p.19, \cite{Coppel2}).

Let $\Theta(x)$ be the fundamental solution  of \eqref{BtweiEq} satisfying $\Theta(0) = I$.  Let $P_1:=\mathbf{e}_1\mathbf{e}_1^T$, where $\mathbf{e}_1:=(1,0,0,0)^T$. Then, the Green function of \eqref{BtweiEq} is given by
\begin{equation}\label{Expodicho} 
\begin{split}
\widetilde{G}(x,x') 
& := \left\{ 
\begin{array}{l l}
\Theta(x)P_1\Theta^{-1}(x') &  \text{for} \quad x>x',  \\
-\Theta(x)(I-P_1)\Theta^{-1}(x') & \text{for} \quad x'>x,
\end{array}
\right. \\
& = \left\{
\begin{array}{ll}
\text{diag}(e^{(\oRe \mu_1 + \etba)(x-x')},0,0,0) & \text{for} \;\; x>x', \\
\text{diag}(0,\widetilde{\Theta}(x)\widetilde{\Theta}^{-1}(x'),e^{(\oRe \mu_4 + \etba)(x-x')}) & \text{for} \;\; x'>x,
\end{array}
\right.
\end{split}
\end{equation}
where we have used \eqref{BtweiEq24} of Lemma \ref{Lem_BtildeEq} in the second line. On the other hand, we have from \eqref{mujlarge1} that for all sufficiently small $\delta>0$ and $\veps>0$,
\begin{equation}\label{Expodicho1}
\oRe \mu_1 + {\etba} < -1/2  \text { and } \oRe\mu_4 + \beta >1/2 \quad \text{for } \lambda \in \mathfrak{D}_4.
\end{equation}

Combining \eqref{Expodicho} and \eqref{Expodicho1} together with \eqref{BtweiEq2} of Lemma \ref{Lem_BtildeEq}, we obtain that there exist positive constants $\delta_0$ and  $\veps_0$ such that for all $(\veps,\lambda) \in (0,\veps_0]\times \mathfrak{D}_4$ with $\oRe \lambda \geq 0$, there holds that 
\begin{equation}
\left\{ 
\begin{array}{l l}
|\Theta(x)P_1\Theta^{-1}(x')| \leq e^{-\frac{1}{2}(x-x')} &  \text{for} \quad x>x',  \\
|\Theta(x)(I-P_1)\Theta^{-1}(x')| \leq e^{-\frac{1}{2}(x'-x)} & \text{for} \quad x'>x.
\end{array}
\right.
\end{equation}
Hence, the system \eqref{BtweiEq} possesses an exponential dichotomy on $\mathbb{R}$ with the projection $P_1$ and uniform constants. This completes the proof.
\end{proof}

\subsection{Fundamental solutions for large eigenvalue parameter}
\begin{lemma}\label{Lem_BtildeEq}
For the symmetric matrix $S_1$ defined in \eqref{S1sym}, which is non-negative for all small $\veps>0$, the following holds true.
\begin{enumerate}
\item Let $\Xi(x;x_0)\in\mathbb{C}^{4\times 4}$ be the fundamental matrix of 
\begin{equation}\label{BtildeEq}
\frac{d\bm{a}}{dx} = \mathrm{diag}\,(0,\mu_2-\mu_1,\mu_3-\mu_1,\mu_4-\mu_1)\bm{a} + \frac{\lambda}{2\sqrt{K}}S_1\bm{a}
\end{equation}
satisfying $\Xi(x_0;x_0)=I$. Then, if  $\veps>0$ and $\delta>0$ are sufficiently small, we have that for all $\lambda \in \mathfrak{D}_4$,
\begin{equation}\label{BtildeEq2}
|\Xi(x;x_0)\bm{a}_0| \leq |\bm{a}_0| \quad \text{for all} \; x\leq x_0 \; \text{and}\; \bm{a}_0\in \mathbb{C}^4.
\end{equation}

\item Let $\widetilde{\Theta}(x;x_0)\in\mathbb{C}^{2\times 2}$ be the fundamental matrix of 
\begin{equation}\label{BtweiEq23}
\frac{d\widetilde{\bm{a}}}{dx} = \mathrm{diag}(\mu_2+\beta,\mu_3+\beta)\widetilde{\bm{a}} + \frac{\lambda}{2\sqrt{K}}\widetilde{S}_1\widetilde{\bm{a}}
\end{equation}
satisfying $\widetilde{\Theta}(x;x_0)=I$, where $\widetilde{S}_1$ is the largest submatrix of $S_1$ with non-zero entries given by
\begin{equation}\label{Ch3S1tilde}
\widetilde{S}_1:=
\begin{pmatrix}
2\sqrt{K}R^{(1)}_{11}-K R^{(1)}_{12}-R^{(1)}_{21}  & K R^{(1)}_{12} - R^{(1)}_{21} \\
K R^{(1)}_{12} - R^{(1)}_{21}  & 2\sqrt{K} R^{(1)}_{11}+K R^{(1)}_{12}+R^{(1)}_{21} 
\end{pmatrix}.
\end{equation}
Then for all sufficiently small $\veps>0$, we have that on the region $\oRe \lambda \geq 0$,
\begin{equation}\label{BtweiEq2}
|\widetilde{\Theta}(x;x_0)\widetilde{\bm{a}}_0| \leq e^{2\beta(x-x_0)} |\widetilde{\bm{a}}_0| \quad \text{for all} \; x\leq x_0 \; \text{and}\; \widetilde{\bm{a}}_0\in \mathbb{C}^2.
\end{equation}
Furthermore, $\Theta(x;x_0)\in\mathbb{C}^{4\times 4}$, the fundamental matrix of \eqref{BtweiEq} satisfying $\Theta(x;x_0)=I$, is given by 
\begin{equation}\label{BtweiEq24}
\Theta(x;x_0) = \mathrm{diag}(e^{(\oRe \mu_1 + \etba)(x-x_0)}, \widetilde{\Theta}(x;x_0), e^{(\oRe \mu_4 + \etba)(x-x_0)}).
\end{equation}
\end{enumerate}
\end{lemma}

\begin{proof}
We suppress the $(\lambda,\veps)$-dependence of the fundamental solutions. We first prove \eqref{BtildeEq2}.  We let $(a_1,a_2,a_3,a_4)^T(x):=\Xi(x;x_0)\bm{a}_0$. From the structure of the equation \eqref{BtildeEq}, we observe that 
\[
\frac{da_1}{dx} = 0, \quad \frac{da_4}{dx} = (\mu_4-\mu_1)a_4.
\]
From \eqref{mujlarge1}, we have that $\mu_4-\mu_1 = 2 + O(|\lambda|^{-2})$ as $|\lambda| \to \infty$ uniformly in $\veps$. Hence, for all sufficiently small $\veps>0$ and $\delta>0$, we have that for all $\lambda \in \mathfrak{D}_4$,
\begin{equation}\label{FunSol6}
a_1(x) = a_1(x_0) \quad \text{and} \quad |a_4(x)| = |e^{(\mu_4-\mu_1)(x-x_0)}a_4(x)| \leq  |a_4(x_0)| \quad \text{for}\; x \leq x_0.
\end{equation}
On the other hand, $(a_2,a_3)^T(x)$ satisfies 
\begin{equation}\label{FunSol1}
\frac{d}{dx}\begin{pmatrix}
a_2 \\
a_3
\end{pmatrix} = \begin{pmatrix}
\mu_2-\mu_1 & 0 \\
0 & \mu_3 - \mu_1
\end{pmatrix}
\begin{pmatrix}
a_2 \\
a_3
\end{pmatrix}
+ \frac{\lambda}{2\sqrt{K}}\widetilde{S}_1
\begin{pmatrix}
a_2 \\
a_3
\end{pmatrix},
\end{equation}
where $\widetilde{S}_1$ is given in \eqref{Ch3S1tilde}. Multiplying \eqref{FunSol1} by the complex conjugate $\overline{(a_2,a_3)}$, we get that
\begin{equation}\label{FunSol2}
\frac{1}{2}\frac{d}{dx}|(a_2,a_3)|^2(x) \geq \min_{j=2,3} \oRe(\mu_j-\mu_1)|(a_2,a_3)|^2 + \frac{\oRe \lambda}{2\sqrt{K}} \overline{(a_2,a_3)}\widetilde{S}_1\begin{pmatrix}
a_2 \\
a_3
\end{pmatrix}.
\end{equation}
From \eqref{mujlarge1}, we have that for all sufficiently large $|\lambda|$ and small $\veps$,
\begin{equation}\label{FunSol3}
\min_{j=2,3}\oRe\, (\mu_j-\mu_1)  
\left\{
\begin{array}{l l}
>0 & \text{if} \; \oRe \lambda \geq 0,\\
=\frac{\oRe \lambda}{c-\sqrt{K}} + 1 +O(|\lambda|^{-1}) & \text{if} \; \oRe\lambda<0.
\end{array}
\right.
\end{equation}
Since $\widetilde{S}_1$ is non-negative, we obtain from \eqref{FunSol2} and \eqref{FunSol3} that for all sufficiently large $|\lambda|$ and small $\veps$,
\begin{equation}\label{FunSol4}
\frac{1}{2}\frac{d}{dx}|(a_2,a_3)|^2 >
\left\{
\begin{array}{l l}
0 & \text{if} \; \oRe \lambda \geq 0,\\
\frac{1}{2}|(a_2,a_3)|^2 & \text{if} \;  -\veps^{3/2}\eta(c_0) \leq \oRe \lambda <0,
\end{array}
\right.
\end{equation}
where $-\veps^{3/2}\eta(c_0)$ is the real part of the (left)  boundary of the domain $\Omega^\veps$ (see \eqref{ND-eps}). Integrating \eqref{FunSol4} over $[x,x_0]$, we have that for all sufficiently small $\veps>0$ and $\delta>0$, 
\begin{equation}\label{FunSol5}
|(a_2,a_3)|(x) \leq |(a_2,a_3)|(x_0) \quad \text{for } x \leq x_0
\end{equation}
holds for $\lambda \in \mathfrak{D}_4$. Now \eqref{BtildeEq2} follows from \eqref{FunSol6} and \eqref{FunSol4}.

Now we prove \eqref{BtweiEq2}. Multiplying \eqref{BtweiEq23} by the complex conjugate of $\widetilde{\bm{a}}^T$, 
\[
\frac{1}{2}\frac{d}{dx}|\widetilde{\bm{a}}|^2(x) \geq \min_{j=2,3} \oRe(\mu_j+\beta)|\widetilde{\bm{a}}|^2 + \frac{\oRe \lambda}{2\sqrt{K}} \overline{\widetilde{\bm{a}}^T}\widetilde{S}_1\widetilde{\bm{a}}
\]
Since $\oRe \mu_2,\oRe \mu_3 \geq 0$ when $\oRe \lambda \geq 0$ (see \eqref{EigenSpliting1}) and $\widetilde{S}_1$ is non-negative, we obtain 
\[
\frac{d}{dx}|\widetilde{\bm{a}}|^2(x) \geq 2\beta|\widetilde{\bm{a}}|^2(x),
\]
which leads \eqref{BtweiEq2}.

We omit the proof of \eqref{BtweiEq24} since it easily follows from the structure of the equation \eqref{BtweiEq}. 
\end{proof}

\section{Linear asymptotic stability in weighted spaces}\label{Sec_LinConvec}

\subsection{Semigroup generation}
In the following lemma, we show that $\mathcal{L}$ generates a $C_0$-semigroup on $(L_\beta^2)^2$, which is the assertion of \eqref{Semig} in Proposition \ref{MainIngred}.
\begin{lemma}
Consider the operator $\mathcal{L}: (L_\etba^2)^2 \to (L_\etba^2)^2$ with dense domain $(H_\etba^1)^2$. For each $\veps \in (0,\veps_K)$, $\mc{L}$ generates a $C_0$-semigroup on $(L_\etba^2)^2$.
\end{lemma}

\begin{proof}
For notational simplicity, we let $\mathbf{u}=(\mathbf{u}_1,\mathbf{u}_2):=(\dot{n},\dot{u})^T$.
Considering the change of variable $\mathbf{u}^\etba:=e^{{\etba} x}\mathbf{u}$, we show that the operator $\mc{L}_\etba:= e^{{\etba} x}\mc{L}e^{-{\etba} x}$, given by
\begin{equation}\label{Ch3_Def_mcL_eta}
\mathcal{L}_{\etba} \mathbf{u}^\etba
= - (L(\partial_x-\etba) + (\partial_xL)) \mathbf{u}^\etba
 - \begin{pmatrix}
 0 \\
 (\partial_x-\etba)(-(\partial_x-\etba)^2 +e^{\phi_c})^{-1}(\mathbf{u}_1^\etba)
 \end{pmatrix},
\end{equation}
generates a $C_0$-semigroup $e^{\mc{L}_{\etba} t}$ on $(L^2)^2$. This implies that $e^{\mc{L}t}:= e^{-{\etba} x}e^{\mc{L}_{\etba} t}e^{{\etba} x}$ is a $C_0$-semigroup generated by $\mc{L}$ on $(L_\etba^2)^2$.  Since $\mathcal{L}_{\etba}$ can be seen as a bounded perturbation of $-L\partial_x$, it is enough to check that $-L\partial_x$ generates a $C_0$-semigroup on $(L^2)^2$ (see \cite{Nagel}). 

We observe that the matrix $L$ is symmetrizable. In the variable $\widetilde{\mathbf{u}}:=L_0^{1/2}\mathbf{u}^\beta$, where  $
L_0^{1/2}:=\mathrm{diag}(\frac{\sqrt{K}}{\sqrt{1+n_c}}, \sqrt{1+n_c})$ is the positive square root of the symmetrizer, the coefficient of the operator $\partial_x$ in the transformed operator 
$\widetilde{\mathcal{L}}:=L_0^{1/2}(-L\partial_x)L_0^{-1/2}$, given by 
\[
\widetilde{\mathcal{L}}\widetilde{\mathbf{u}}
 = L_0^{1/2}(-L)L_0^{-1/2}\partial_x \widetilde{\mathbf{u}} + L_0^{1/2}(-L)\partial_x(L_0^{-1/2}) \widetilde{\mathbf{u}},
\]
is a real-valued symmetric matrix.  Note that if $\widetilde{\mathcal{L}}$ generates a $C_0$-semigroup $e^{\widetilde{\mathcal{L}}t}$, then $e^{-L\partial_xt}:=L_0^{-1/2}e^{\widetilde{\mathcal{L}}t}L_0^{1/2}$ is a $C_0$-semigroup with generator $-L\partial_x$. 
Therefore, it suffices to show that $\widetilde{\mathcal{L}}$ generates a $C_0$-semigroup $e^{\widetilde{\mathcal{L}}t}$. To this end, 
we let  $S_0:=L_0^{1/2}(-L)L_0^{-1/2}$.
For sufficiently large $a>0$, integration by parts yields that
\[
\begin{split}
\oRe \langle (S_0\partial_x -aI)\widetilde{\mathbf{u}},\widetilde{\mathbf{u}} \rangle
& = -\frac{1}{2}\oRe \langle (\partial_xS_0)\widetilde{\mathbf{u}}, \widetilde{\mathbf{u}} \rangle - a\|\widetilde{\mathbf{u}}\|_{L^2}^2 <0,
\end{split}
\] 
that is, $S_0\partial_x - aI$ is dissipative.  By the Lumer-Phillips Generation Theorem (see \cite{Nagel}), $S_0\partial_x - aI$ generates a $C_0$-(contraction)semigroup, and hence  $\widetilde{\mathcal{L}}$ also  generates a $C_0$-semigroup as a boundedly perturbed operator of  $S_0\partial_x - aI$. This completes the proof. 
\end{proof}
 
\subsection{Linear asymptotic stability}

From Proposition \ref{MainIngred}, we deduce our main result on  the asymptotic linear stability (Theorem \ref{LinStab}).
\begin{proof}[Proof of Theorem \ref{LinStab}]
Since $\lambda=0$ is an isolated eigenvalue, we can define the spectral projection 
\[
\mathcal{P}_0=\frac{1}{2\pi i }\int_{\Gamma_0}(\lambda-\mathcal{L})^{-1}\,d\lambda,
\]
where $\Gamma_0$ is a positively oriented circle with sufficiently small radius enclosing no other spectrum except $\lambda=0$. The range of $\mathcal{P}_0$, denoted by $\textrm{Ran}\mathcal{P}_0$, is a two-dimensional subspace of $(L_\beta^2)^2$. 
Note that $\mathrm{Ran }\mathcal{P}_0$ and $\mathrm{Ran }(I-\mathcal{P}_0)=\mathrm{Ker }\mathcal{P}_0$ are invariant subspaces under $\mathcal{L}$, thus  they are also invariant under the $C_0$-semigroups generated by the restricted operators $\mathcal{L}|_{\mathrm{Ran }\mathcal{P}_0}$  and $\mathcal{L}|_{\mathrm{Ran }(I-\mathcal{P}_0)}$, respectively. We have the spectral decomposition of the operator $\mathcal{L}$ as follows (see \cite{Kato}):
\[
\sigma\left( \mathcal{L}|_{\mathrm{Ran}\mathcal{P}_0} \right) = \{0\} = \sigma_{\mathrm{pt}}(\mathcal{L}), \quad \sigma\left( \mathcal{L}|_{\mathrm{Ran}(I-\mathcal{P}_0)} \right) = \sigma_{\mathrm{ess}}\left( \mathcal{L} \right) \subset \mathbb{C}\setminus \Omega^\veps.
\]

The resolvent of $\mathcal{L}|_{\mathrm{Ran}(I-\mathcal{P}_0)}$ is analytic on $ \Omega^\veps$, and thus it is uniformly (in $\lambda$) bounded on the region $\oRe \lambda \geq 0$. By the Pr\"uss theorem (\cite{Pruss}), we conclude that the $C_0$-semigroup generated by $\mathcal{L}|_{\mathrm{Ran}(I-\mathcal{P}_0)}$, which coincides with $e^{\mathcal{L}t}$ restricted on  $\textrm{Ran}(I-\mathcal{P}_0)$, satisfies the estimate \eqref{LinStabEq}.
\end{proof}

\section{Instability criterion for large amplitude solitary waves}\label{unstable}
In this section, we derive an instability criterion for \emph{large amplitude} solitary wave solutions to the isothermal Euler-Poisson system \eqref{EP}.

First of all, the Evans function $D(\lambda,\veps)$ defined in \eqref{Evans-Def} is also real whenever $\lambda$ is real-valued (see \cite{PW}, Proposition 1.10). From Proposition \ref{EvansZero}, we have that $D(\lambda,\veps)|_{\lambda=0} = \partial_\lambda D(\lambda,\veps)|_{\lambda=0} = 0$. Now we claim that if 
\[
\partial_\lambda^2 D(0,\veps_u) <0
\]
for some $\veps_u \in(0, \veps_K)$, then $\mathcal
{L}$ has a positive $L^2$-eigenvalue for some $\veps_u'\in(0,\veps_u]$.

Since $D(\lambda,\veps_u)$ is concave down near $\lambda=0$, we have the following scenarios: for $\delta_4>0$ in Lemma \ref{Lem_Region4}, either 
\begin{enumerate}[{(i)}]
\item\label{1}   $D(\lambda_u,\veps_u) = 0$ for some $\lambda_u \in (0, \delta_4^{-1}]$, or
\item\label{2} $D(\lambda,\veps_u) <0$ for all $\lambda \in (0,\delta_4^{-1}]$. In this case, either
\begin{enumerate}
\item   $D(\lambda_u,\veps_u) = 0$ for some $\lambda_u \in (\delta_4^{-1},+\infty)$, or
\item $D(\lambda,\veps_u) <0$ for all $\lambda \in (\delta_4^{-1},+\infty)$.
\end{enumerate}
\end{enumerate}

In the cases of \eqref{1} and (\ref{2})(a), $\mathcal{L}$ has  a positive $L^2$-eigenvalue $\lambda_u$. On the other hand, the case (\ref{2})(b) implies by Lemma \ref{Lem_Region4} that there must exist $\veps_u'\in (0,\veps_u)$ and $\lambda_u' \in (\delta_4^{-1},+\infty)$ such that $D(\lambda_u',\veps_u')=0$ since $D(\lambda,\veps)>1/2$ for all  $\lambda \in (\delta_4^{-1},+\infty)$ and sufficiently small $\veps$.

From this observation, together with Proposition \ref{EvansZero}.\eqref{EvansZero2nd}, we obtain the following instability criterion. 
\begin{proposition}\label{Instacrite}
Let 
\begin{equation}\label{Def_Q}
Q(c):=\int_{-\infty}^\infty (n_cu_c)(x)\,dx.
\end{equation}
If $\partial_c Q(c)<0$ for some $\veps_u \in (0,\veps_K)$, the operator $\mathcal{L}$ on $(L^2)^2$ has a positive eigenvalue  for some $\veps_u' \in (0,\veps_u]$. Furthermore, for all $\veps\in(0,\veps_K)$ there holds that
\begin{equation}\label{Eq_InstaForm}
Q(c) = \int_0^{n_c^\ast}\frac{\sqrt{2}cn^2\partial_nH(n,c)}{(1+n)\sqrt{g(n,c)}}\,dn,
\end{equation}
where $H$ and $g$ are the functions defined in \eqref{gnc} and \eqref{TravelEqB2}, respectively, and $n_c^\ast$ is the maximum value of $n_c(x)$.
\end{proposition}

Notice that only the maximum value of $n_c(x)$, $n_c^\ast$, is involved in the different form of $Q(c)$ given in  \eqref{Eq_InstaForm}. Some numerical
evaluations for the integral \eqref{Eq_InstaForm}  are presented in the next section. 
\begin{proof}
It suffices to check that  \eqref{Eq_InstaForm} holds true. Since $n_c$ and $u_c$ are symmetric about $x=0$ and are strictly increasing on $(-\infty,0)$, we have
\[
\begin{split}
\int_{-\infty}^\infty n_cu_c\,dx 
& = 2 \int_{-\infty}^0 \frac{cn_c^2}{1+n_c}\,dx  \\
& = 2 \int_{-\infty}^0 \frac{cn_c^2}{1+n_c}\frac{\partial_nH(n_c,c)\partial_xn_c}{\partial_x\phi_c}\,dx   \\
& = \sqrt{2} \int_{-\infty}^0 \frac{cn_c^2}{1+n_c}\frac{\partial_nH(n_c,c)\partial_xn_c}{\sqrt{g(n_c,c)}}\,dx  \\
& = \sqrt{2} \int_{0}^{n_c^\ast} \frac{cn^2}{1+n}\frac{\partial_nH(n,c)}{\sqrt{g(n,c)}}\,dn,
\end{split}
\]
where we have used \eqref{TravelEqB1}, \eqref{ODE_n_E1} and \eqref{gnc} for the first, second and third lines, respectively, and we have used $dn = \partial_xn_c(x)\,dx$ for the last equality. We are done.
\end{proof}

\section{Some remarks and open questions}\label{RO}
 For the sake of simplicity, in some  literatures of plasma physics, the Euler-Poisson system \eqref{EP} is often assumed  to have no pressure term (i.e., \eqref{EP} with $K=0$) in accordance with the physical assumption of \emph{cold ion}. 
In a mathematical point of view, the absence of the pressure term makes the system weakly coupled, so it enables  to analyze certain properties of the system significantly easier. These include for instance,  the finite time singularity formation and the existence of solitary waves. 
However, we remark that the pressureless Euler-Poisson system exhibits qualitatively different behaviors in the solutions, from the ones with the pressure. 
To be more specific, we shall discuss some  properties (shape of profiles and stability) of the solitary wave solutions to the isothermal and the pressureless models.

For further discussions, we remark that a smooth non-trivial solitary wave solution to \eqref{EP2Travel2} with $K=0$ satisfying \eqref{bdCon x2} exists if and only if  $1<c<\zeta_0\approx 1.5852$, where $\zeta_0 $ is a unique positive solution to 
\[
z^2 + 1 = \text{exp}(z^2/2)
\]
(see \cite{Sag} and \cite{Satt}).  Indeed, the equations \eqref{EP2Travel2} with $K=0$ can be reduced to a second-order ODE
\[
\partial_x^2\phi = e^\phi - \frac{c}{\sqrt{c^2-2\phi}},
\]
for which the associated first integral of the ODE is given by
\begin{equation}\label{Uphi}
\frac{1}{2}(\partial_x\phi)^2 = e^\phi + c\sqrt{c^2-2\phi} - 1-c^2 =: U(\phi).
\end{equation}
Here for later uses, we let $\veps_0^\ast:=\zeta_0-1$, and  denote the peak value of $n_c$, $u_c$ and $\phi_c$ by $n_c^\ast$, $u_c^\ast$ and $\phi_c^\ast$, respectively. Let $c=\sqrt{1+K}+\veps$ for $K\geq 0$.

\subsection{Numerical computations of instability criterion for  large amplitude solitary waves
}
A criterion for the instability of the \emph{large} amplitude solitary waves has been derived in Proposition~\ref{Instacrite}.  
Seeking \emph{unstable} solitary waves of large amplitude in accordance with this criterion, 
we numerically evaluate the integral $Q(c)$ defined in \eqref{Def_Q}, and the results are presented in Figure \ref{GraphIns} for the case $K>0$. We have used the form in \eqref{Eq_InstaForm} and the numerical integration has been taken over the interval $ [10^{-4}, n_c^\ast- 10^{-10}]$. In fact, it turns out that the instability criterion (Proposition \ref{Instacrite}) is inconclusive. More precisely, our numerical data show that $Q(c)$ is strictly increasing,  so one cannot conclude that there is a positive eigenvalue.
This is contrary to the numerical result found in \cite{Sche} for the pressureless Euler-Poisson system. 

\begin{figure}[h]
\begin{tabular}{cc}
\resizebox{60mm}{!}{\includegraphics{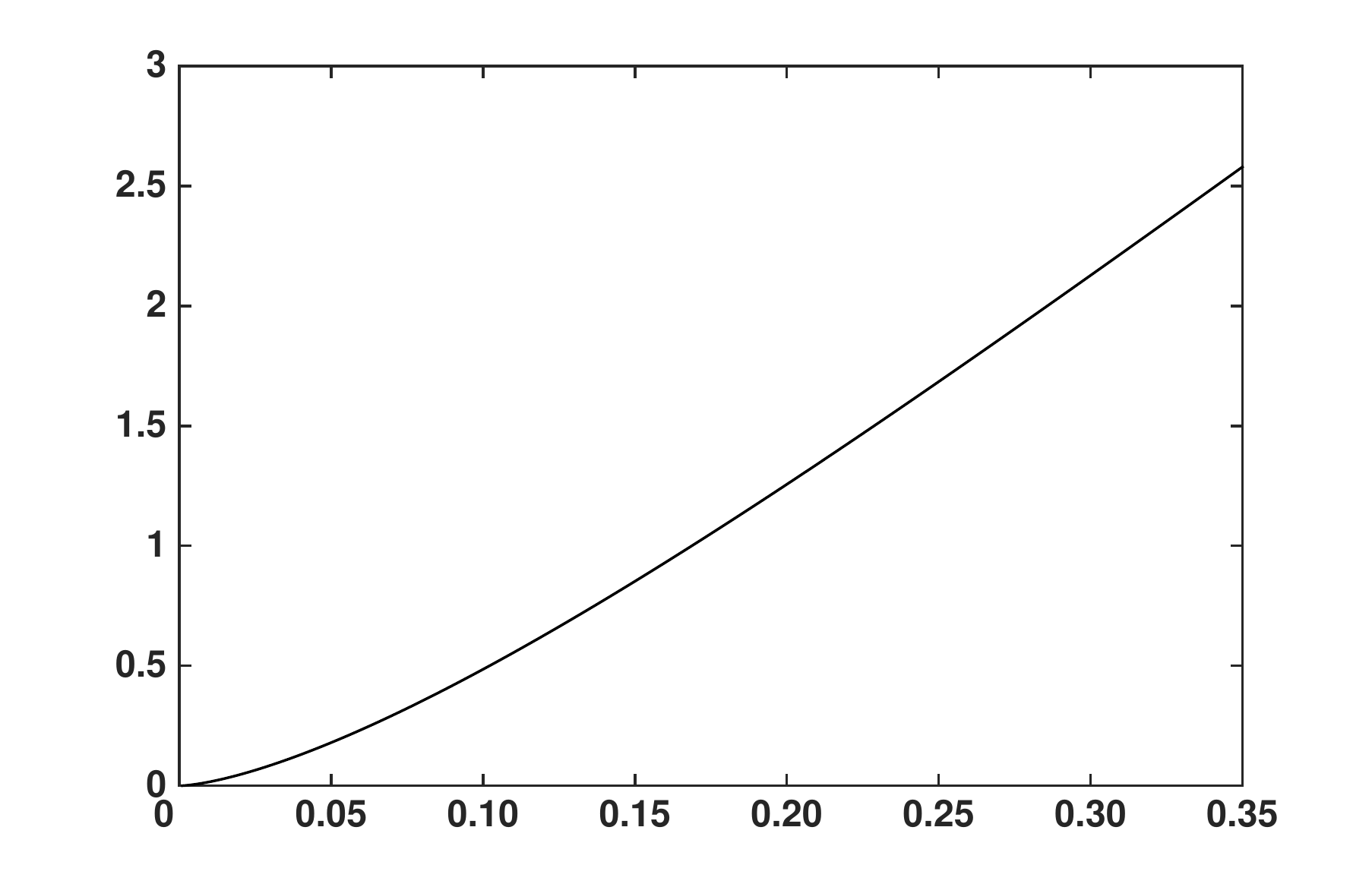}}  & \resizebox{60mm}{!}{\includegraphics{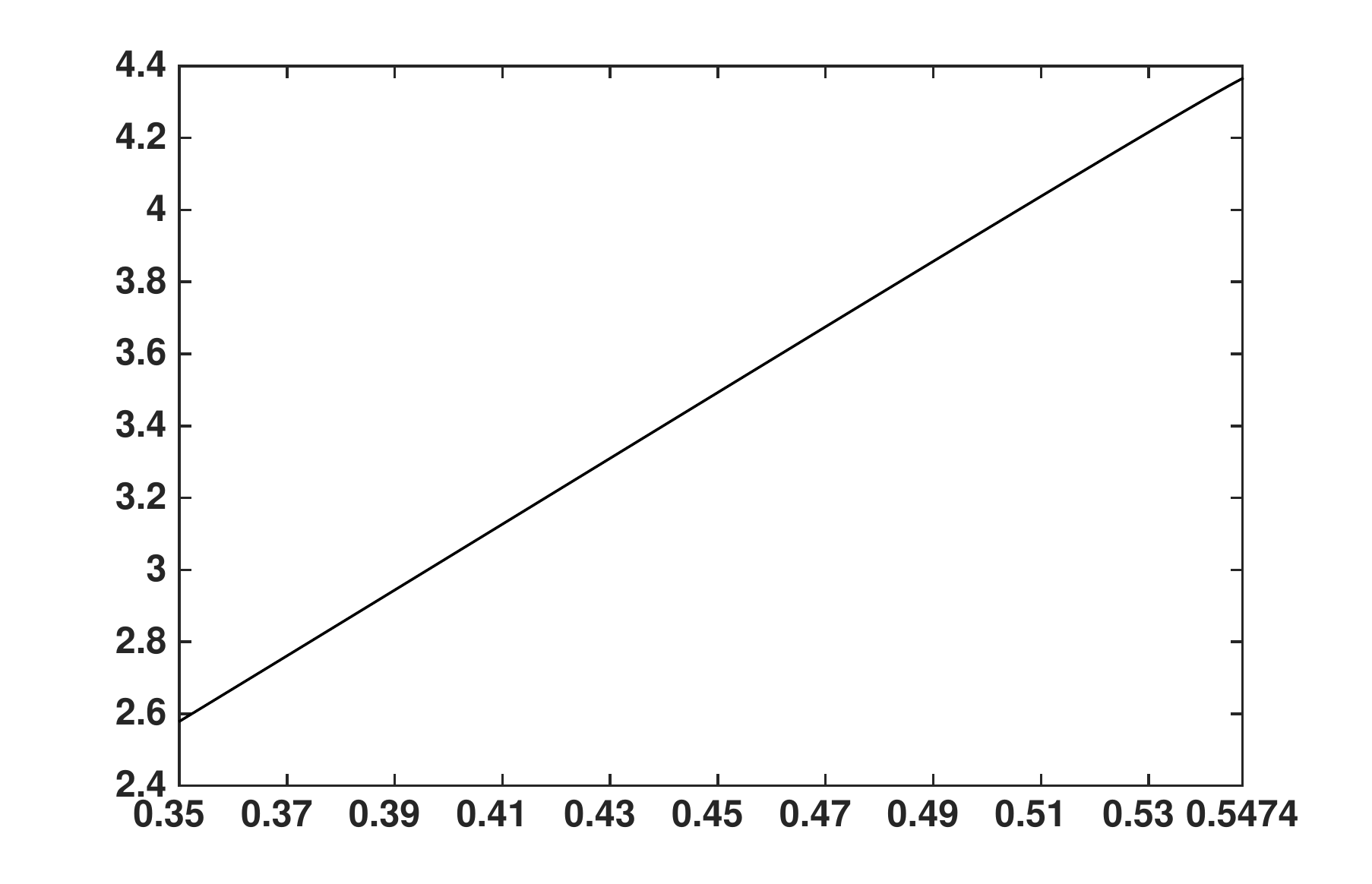}} \\
(a) $K=0.001$, $\veps_K\approx 0.5474$   & (b) $K=0.001$, $\veps_K\approx 0.5474$  \\
\resizebox{60mm}{!}{\includegraphics{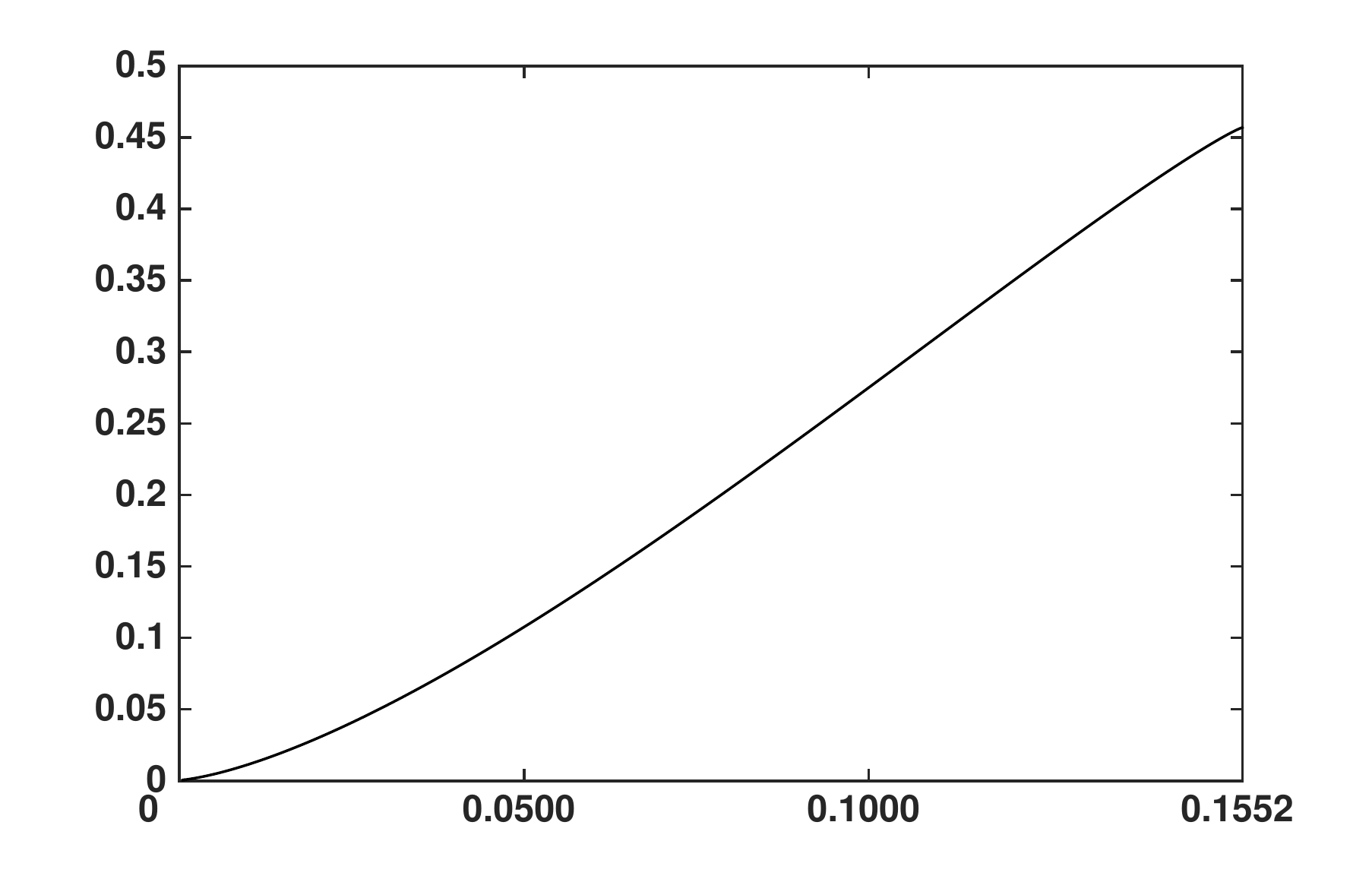}}  & \resizebox{60mm}{!}{\includegraphics{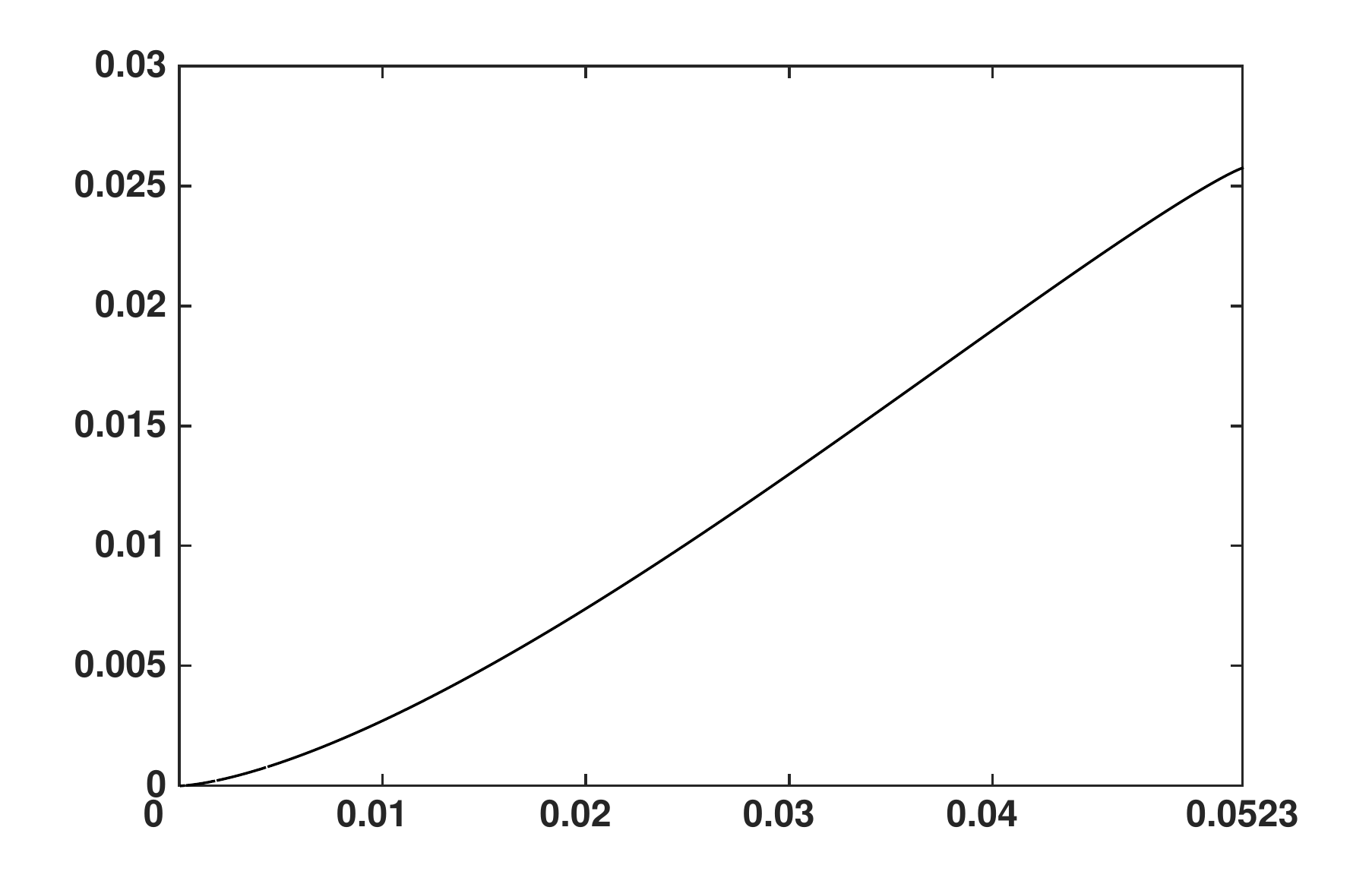}} \\
(c) $K=1$, $\veps_K \approx 0.1553$  & (d) $K=10$, $\veps_K \approx 0.0524$
\end{tabular}
\caption{The numerical plots of $Q(c)$. The horizontal axis represents $\veps=c-\sqrt{1+K}$.}
\label{GraphIns}
\end{figure}

More specifically,  it was numerically found in  \cite{Sche} that there is a critical value $c_{\text{crit}}\approx 1.52603$ such that the solitary waves are unstable for $c_{\text{crit}} <c<\zeta_0\approx 1.5852$. Making use of an instability criterion similar to Proposition \ref{Instacrite}, the authors numerically compute the integral
\begin{equation}\label{q-int}
Q_0(c):= \int_0^{\phi_c^\ast}\frac{\sqrt{2}(c-\sqrt{c^2-2\phi})^2}{\sqrt{c^2-2\phi}\sqrt{U(\phi)}}\,d\phi,
\end{equation}
where $U(\phi)$ is defined in \eqref{Uphi}. According to our further investigation, however, we suspect that their result is due to the lack of numerical accuracy and that the inaccuracy is caused by an immature round-off of the integration interval near the singularity of the integrand at $\phi_c^\ast$.  In fact,  $\phi_c^\ast$ is a singular point of the integrand of \eqref{q-int} since $\phi_c^\ast$ satisfies $U(\phi_c^\ast)=0$.  


%

Our numerical evaluations of $Q_0(c)$ are carried out with two different  integration intervals   
$I_a:= [10^{-4},\phi_c^\ast-10^{-4}]$ (Figure \ref{GraphIns2}.(a)) and $I_b:= [10^{-4},\phi_c^\ast-10^{-10}]$ (Figure \ref{GraphIns2}.(b)). The numerical evaluation of \eqref{q-int} with $I_a$ gives a consistent result with the one in \cite{Sche}, i.e.,  $\partial_cQ_0<0$ for some large $\veps < 0.585$, whereas the one with $I_b$ gives a completely different result, i.e., $\partial_cQ_0>0$ for all $\veps< 0.585$. We believe that the latter is correct since  $I_a\subset I_b \subset [0,\phi_c^\ast]$, and hence one may  conclude that the instability criterion does not give a definite conclusion even for the pressureless case. The questions regarding instability will be further investigated in the future study.


\begin{figure}[h]
\begin{tabular}{cc}
\resizebox{60mm}{!}{\includegraphics{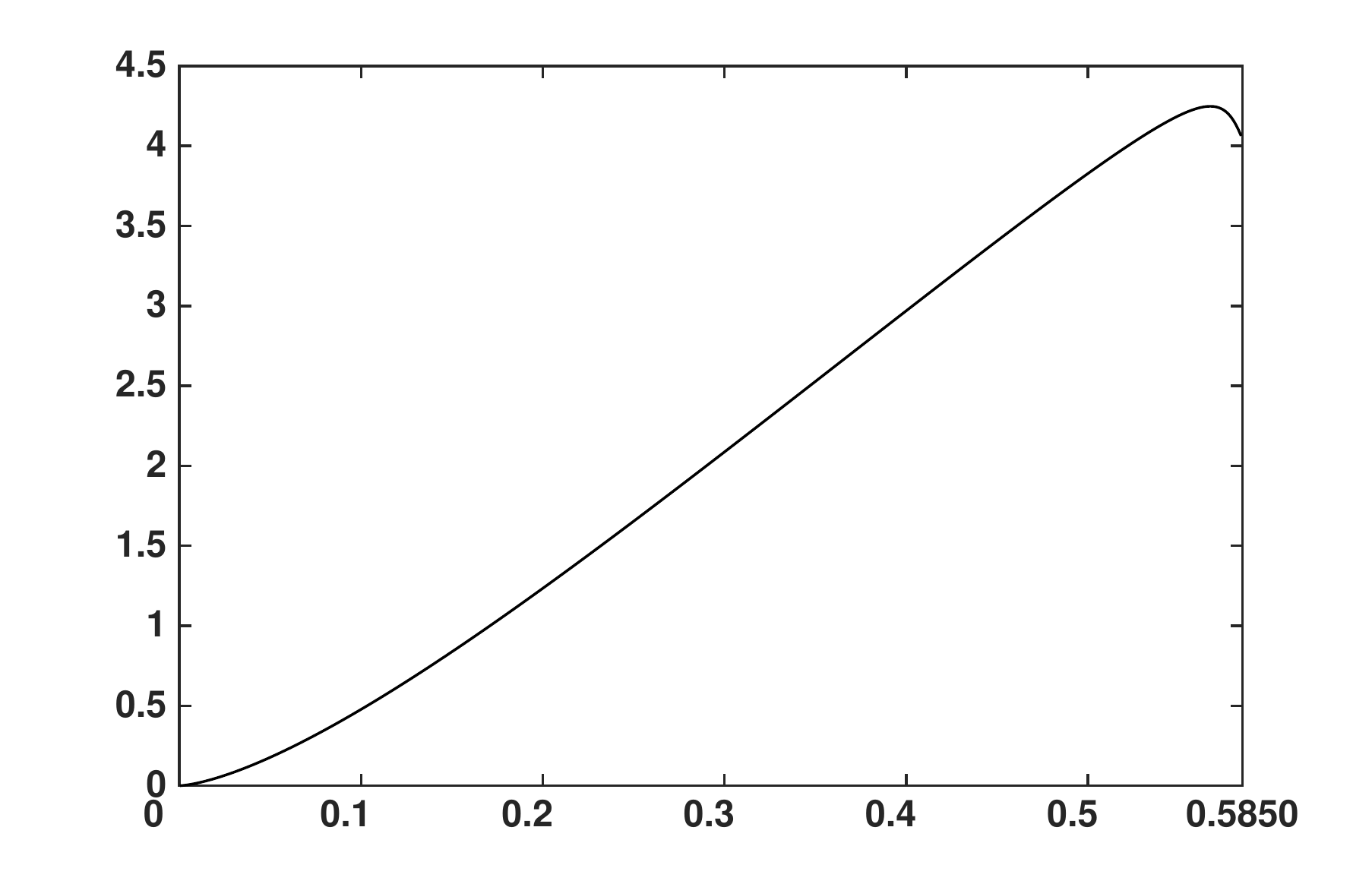}}  & \resizebox{60mm}{!}{\includegraphics{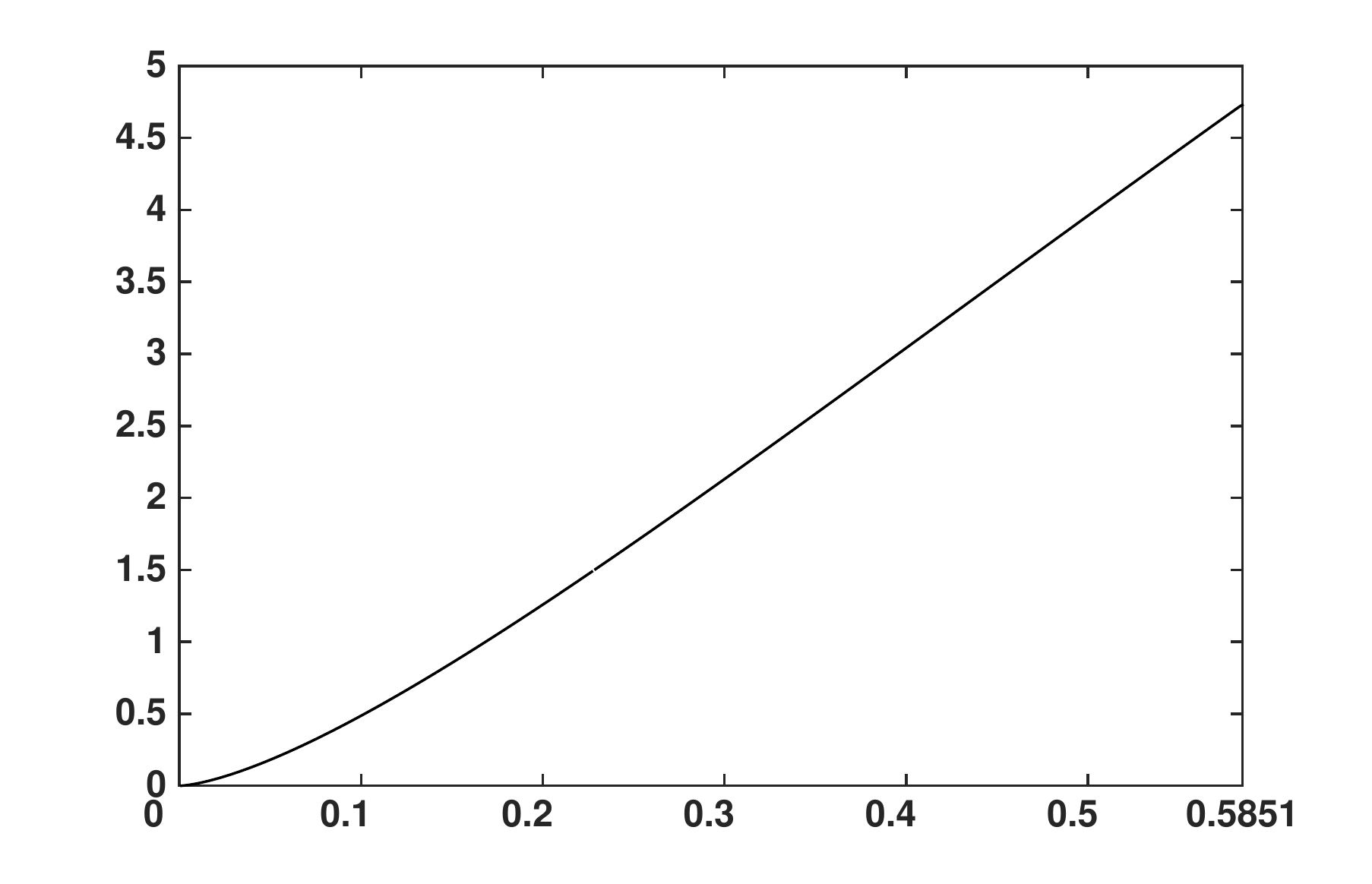}} \\
(a)  & (b)
\end{tabular}
\caption{The numerical plots of $Q_0(c)$ in \eqref{q-int} with the integration intervals $I_a$ and $I_b$, respectively. The horizontal axis represents $\veps=c-1$.}
\label{GraphIns2}
\end{figure}

\subsection{Large amplitude solitary wave profiles of the Euler-Poisson system}\label{Subsec8.2}
One of the remarkable differences between the pressureless and the isothermal Euler-Poisson systems lies in their density profiles of large amplitude solitary wave solutions. 

For the case $K=0$, $n_c^\ast$ approaches to infinity as $\veps \nearrow  \veps_0^\ast$.  In contrast, for the case $K>0$, $n_c^\ast$ remains bounded above. More precisely,  from  the proof of existence of the solitary waves in \cite{BK}, we obtain  
\[
n_c^\ast < n_s:=\frac{c}{\sqrt{K}}-1 < \frac{\sqrt{1+K}+\veps_K-\sqrt{K}}{\sqrt{K}},
\] 
where $n_s$ is a unique positive zero of \eqref{deriHnc}. 
Several numerical experiments for the large amplitude solitary waves 
 are presented in Figure \ref{GraphLargeAmpl2}, Figure \ref{PlottingN}, and Table \ref{TableMax}. 
 %
 %
For $K=0$,   the feature of $L^\infty$-\emph{blow-up} of density profile can be closely related to the fact that the pressureless Euler-Poisson system can develop the \emph{delta shock}
 in finite time when the initial data $u_0$ has a \emph{steep} gradient at some point.
%
%
%
Specifically, if the initial data $(n_0,u_0)(x)$  satisfies
\begin{equation*}\label{conj30}
\partial_x u_0(x) \le  -\sqrt{2(1+n_0(x))}
\end{equation*}
at some point $x \in \mathbb{R}$, then the maximal existence time of the smooth solution is finite, see \cite{Liu}. 
For the singularity formation at finite time $T_*$,  
one can further show, by a simple comparison technique for  ODE, that the gradient of velocity blows up 
 at a non-integrable order in time, i.e.,   
\begin{equation}\label{tstar}
-\partial_x u \gtrsim (T_*-t)^{-1} \text{ as } t\nearrow T_*.
\end{equation}
 From this together with the continuum equation,  
$$ n(x(t),t) = n_0(x_0) \exp\left\{-\int_0^t \partial_x u(x(\tau), \tau) d\tau \right\},$$
we see that $L^\infty$ norm of density becomes unbounded as $t \nearrow  T_*$. 
Here $x(t)$ is the characteristic curve issuing from $x_0$ at which
the infimum of $\partial_x u_0(x)/\sqrt{1+n_0(x) }$ is attained,
i.e., $\partial_x u_0(x_0)/\sqrt{1+n_0(x_0) }= \inf_x \partial_x u_0(x)/\sqrt{1+n_0(x) }<  -\sqrt{2}$. 
This non-physical singular behavior emerges since the pressure term is artificially ignored. 
As we discussed earlier,  this is not the case in the presence of the pressure, in general. 
We suspect that it would be due to this singular behavior if \emph{large amplitude} solitary waves are unstable for the pressureless case.



\begin{figure}[h]
\begin{tabular}{cc}
\resizebox{60mm}{!}{\includegraphics{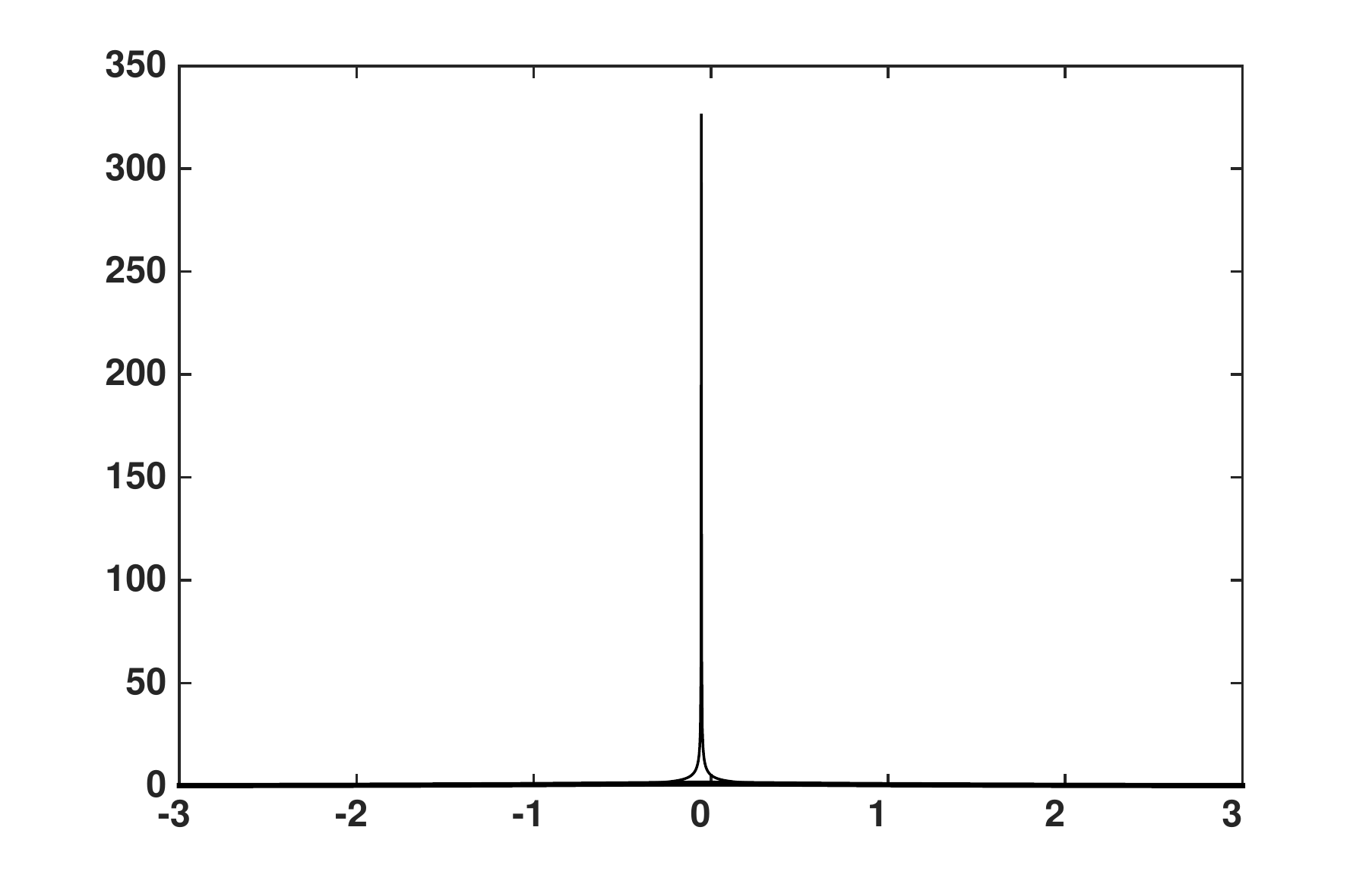}}  & \resizebox{60mm}{!}{\includegraphics{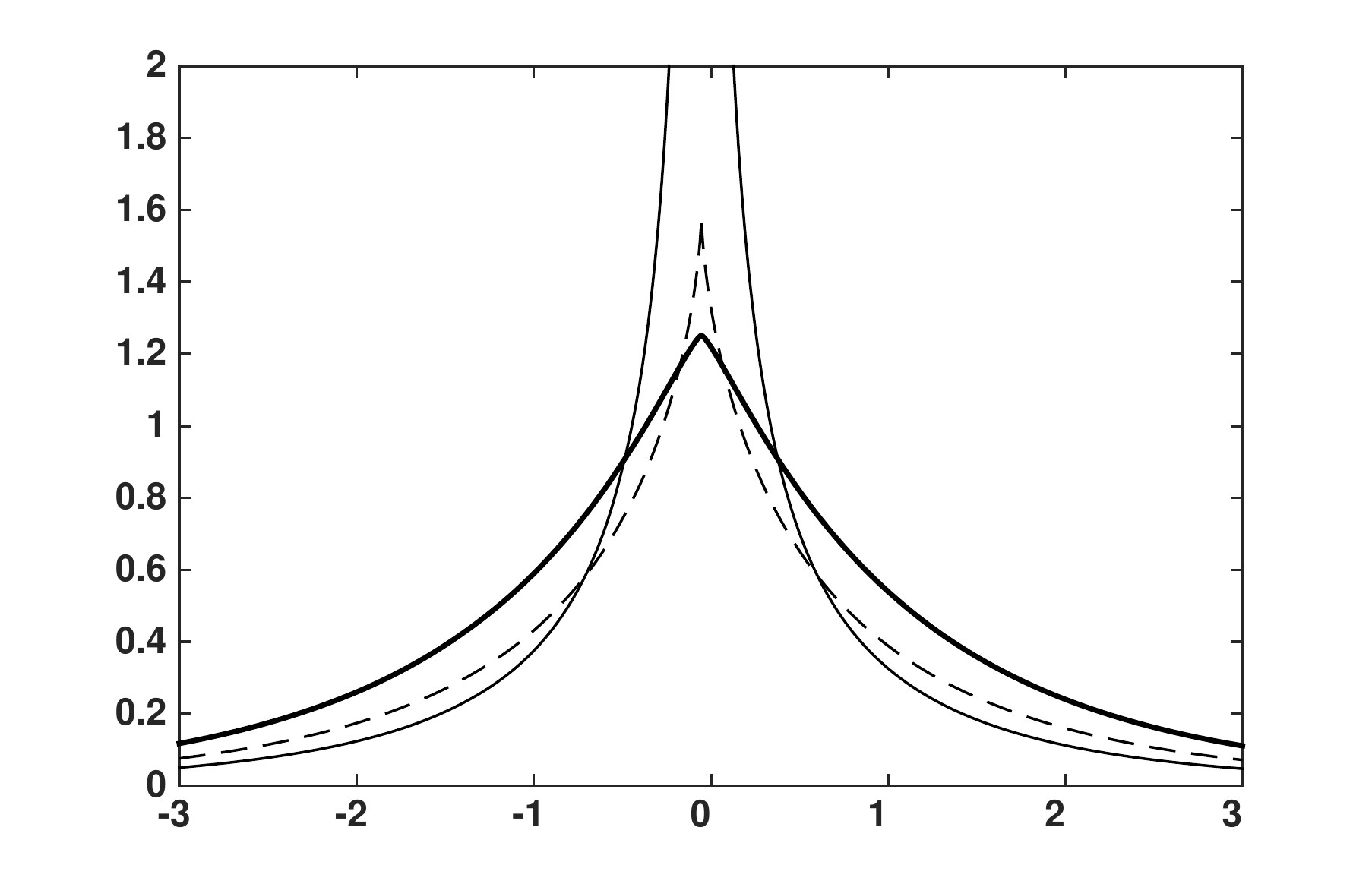}} \\
(a) $K=0$, $\veps=0.5852$ ($\veps_0^\ast \approx 0.5820$)  & (b) $K=0$,  $\veps=0.5852$ ($\veps_0^\ast \approx 0.5820$) \\
\resizebox{60mm}{!}{\includegraphics{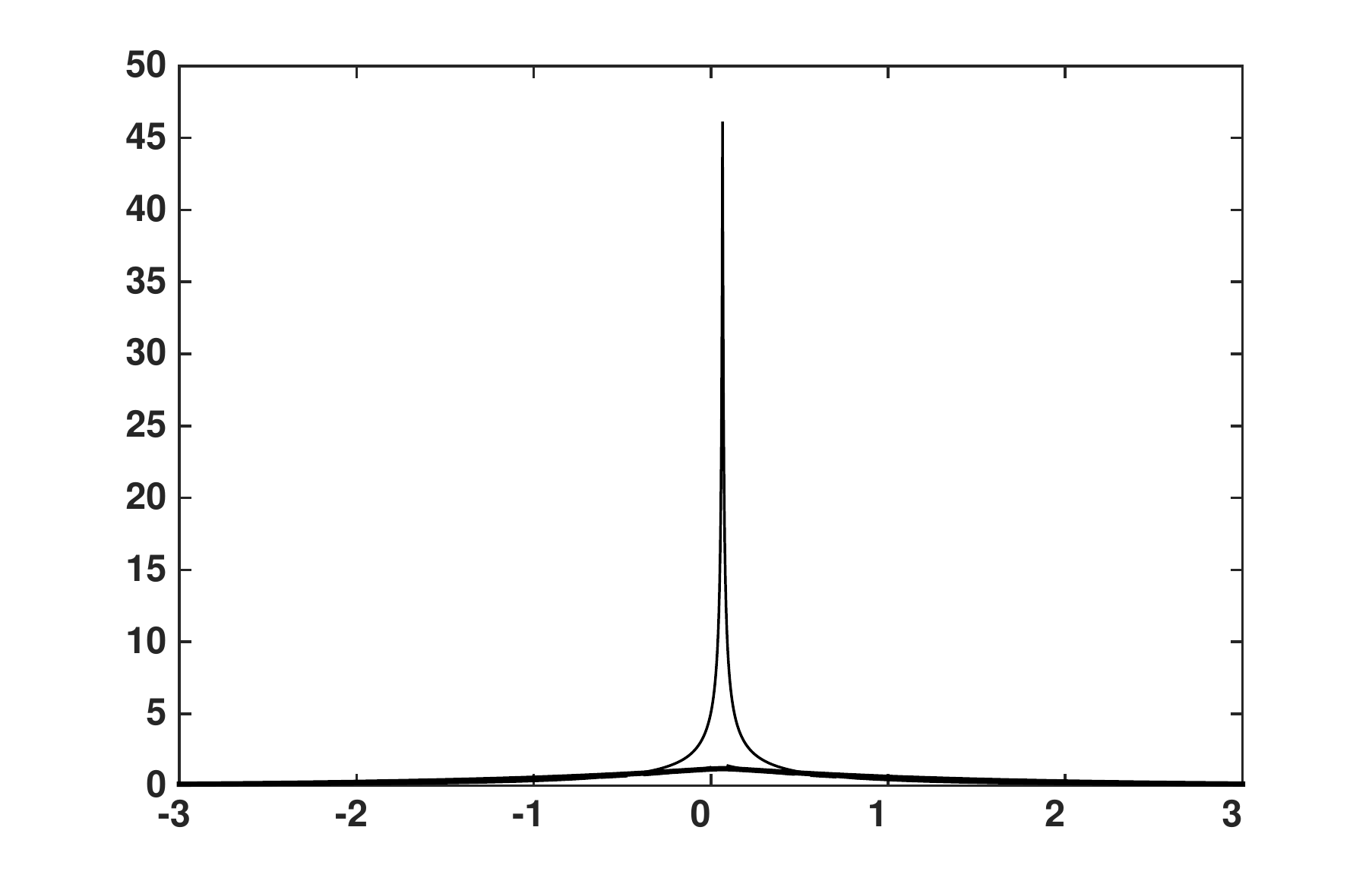}}  & \resizebox{60mm}{!}{\includegraphics{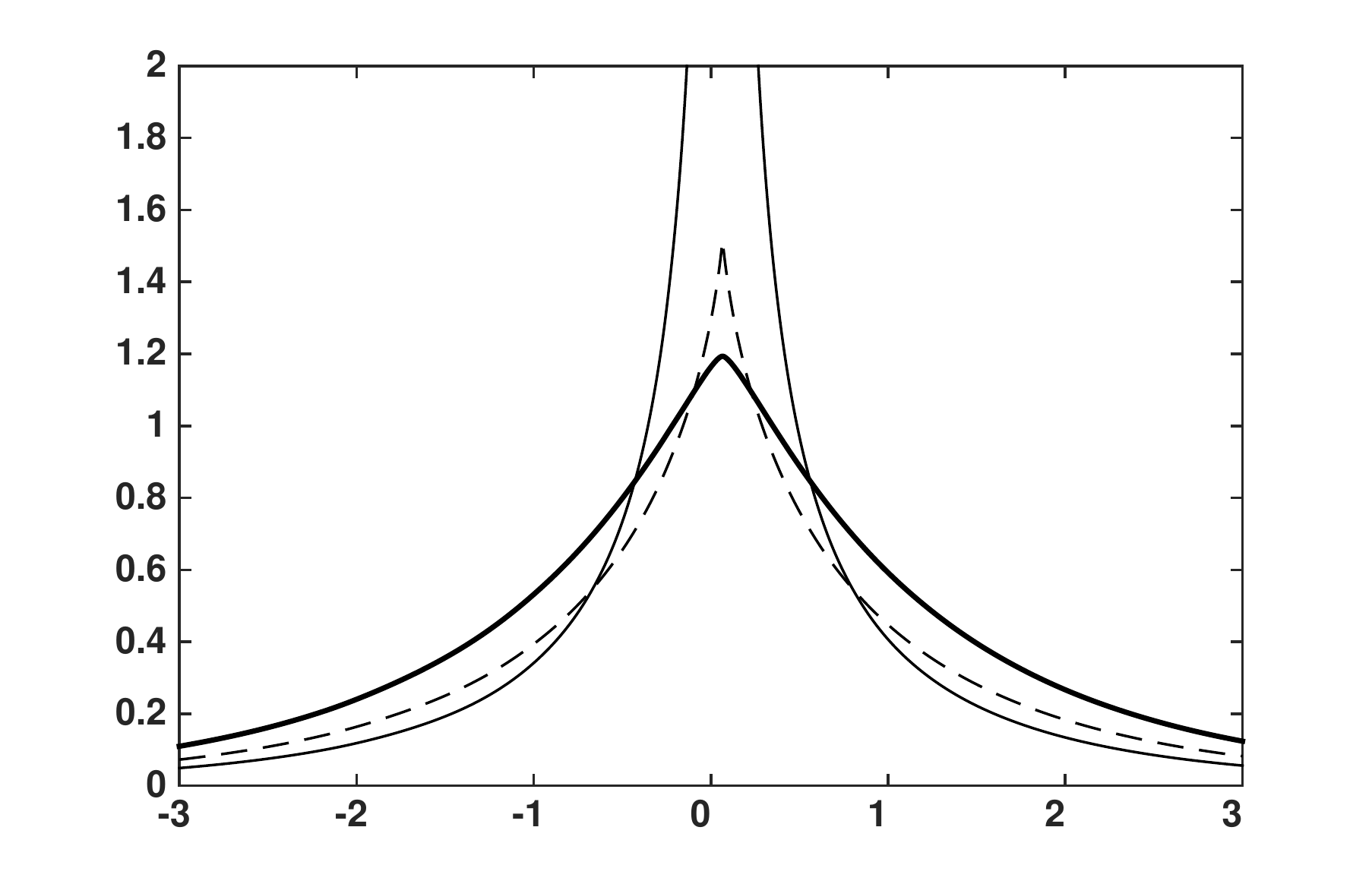}} \\
(c) $K=0.001$,  $\veps=0.5474$  ($\veps_K \approx 0.5475$) & (d) $K=0.001$,  $\veps=0.5474$ ($\veps_K \approx 0.5475$) \\
\resizebox{60mm}{!}{\includegraphics{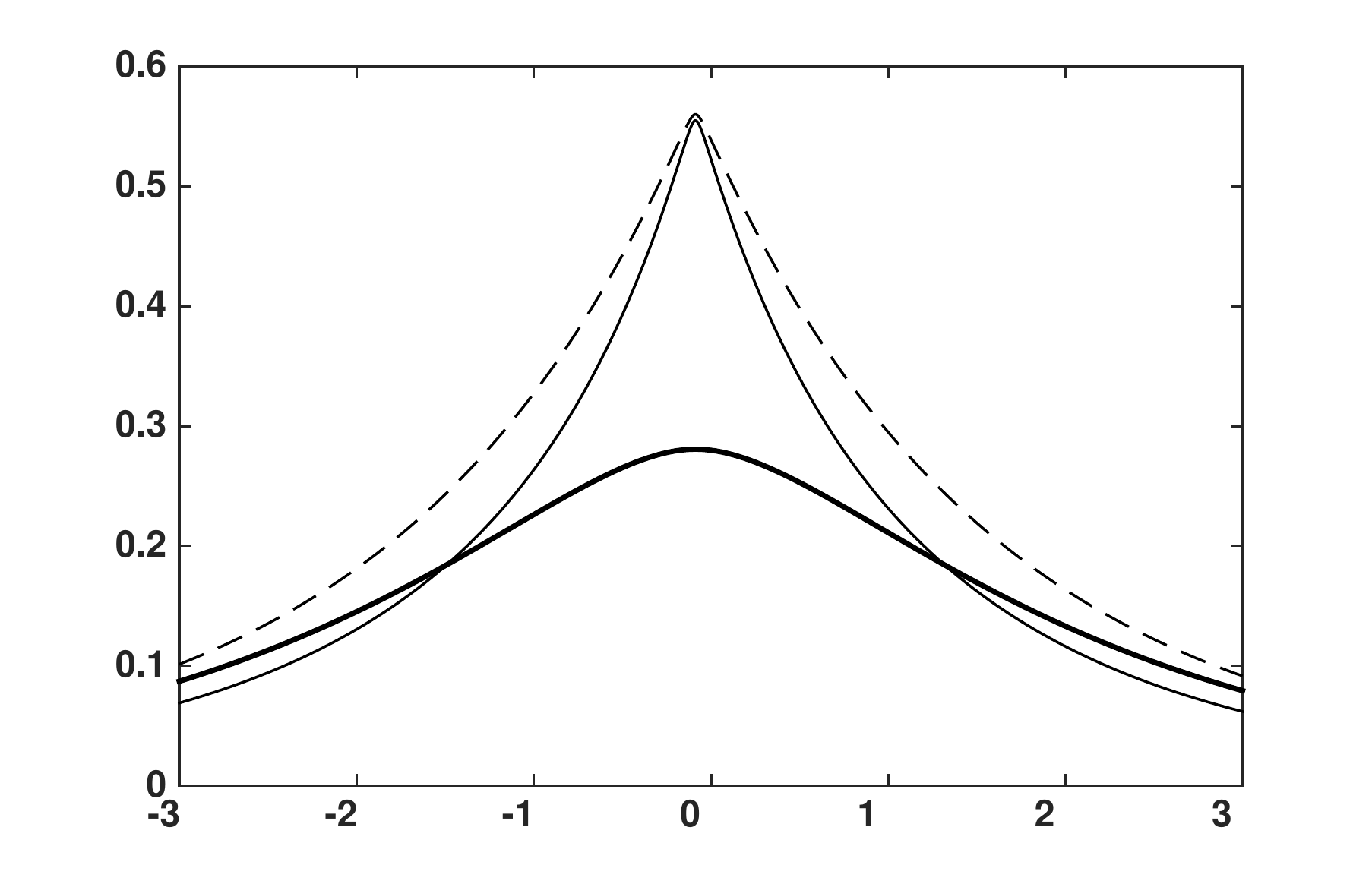}} & \resizebox{60mm}{!}{\includegraphics{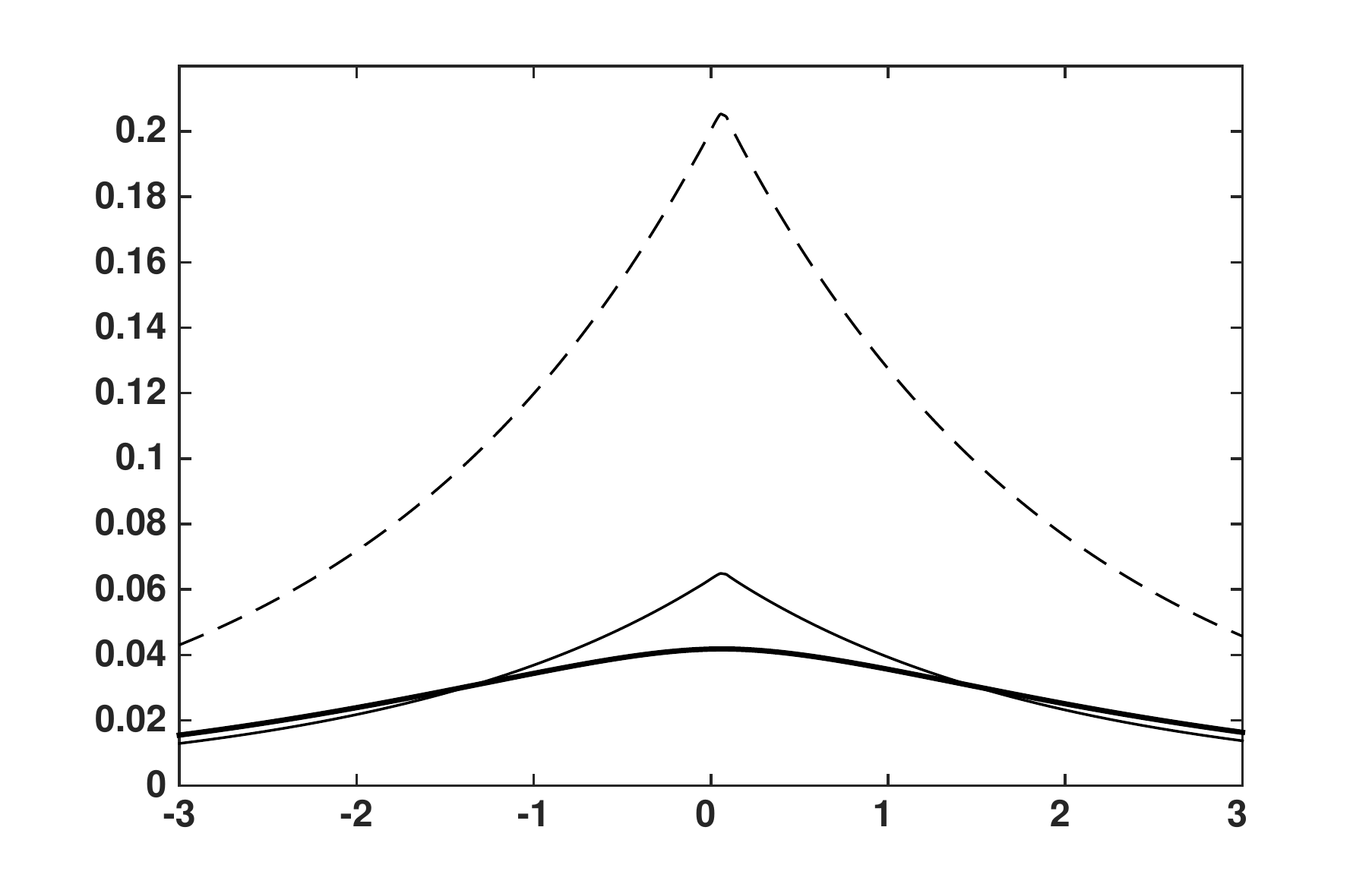}} \\
(e) $K=1$,  $\veps=0.1550$ ($\veps_K \approx 0.1553$) & (f) $K=10$,  $\veps=0.05234$ ($\veps_K \approx 0.0524$)
\end{tabular}
\caption{The graphs of $n_c$ (solid), $u_c$ (dashed), and $\phi_c$ (thick solid).}
\label{GraphLargeAmpl2}
\end{figure}

\begin{figure}[h]
\centering
\resizebox{120mm}{!}{\includegraphics{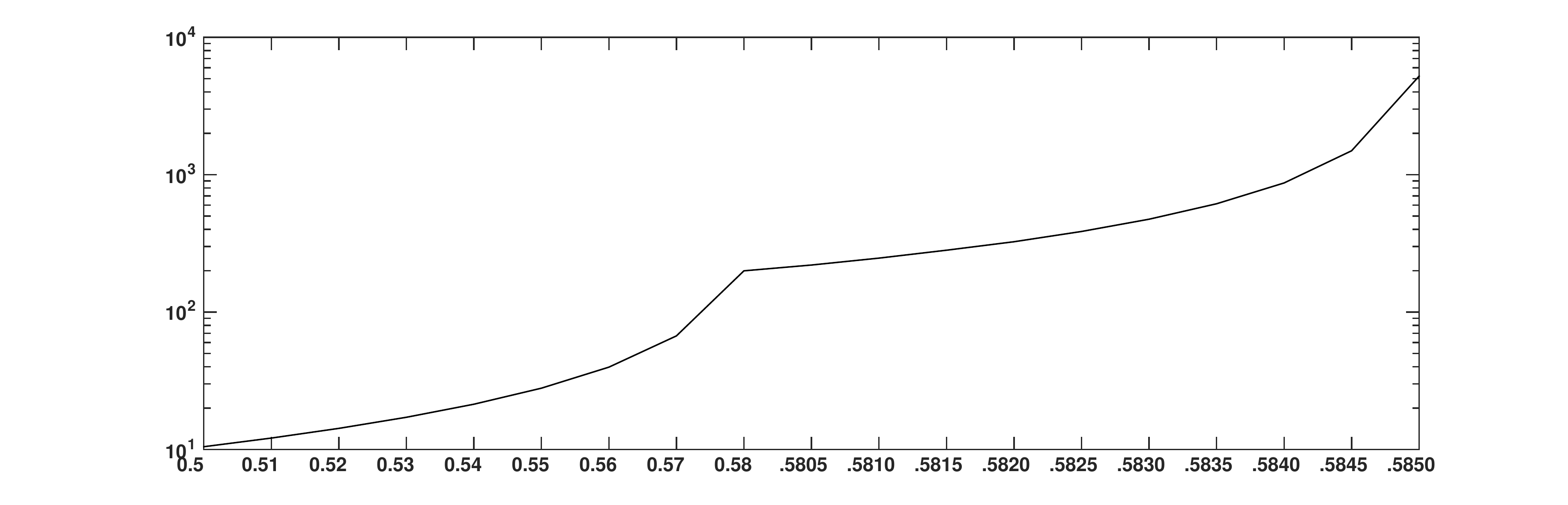}}
\caption{Plotting for $n_c^\ast$ when $K=0$.}
\label{PlottingN}
\end{figure}

\clearpage

\begin{table}\scriptsize
\centering
\begin{subtable}[t]{.45\linewidth}
      \centering
        \begin{tabular}[t]{|c||c|c|c|}
\hline
$\veps$ & $n_c^\ast$  & $u_c^\ast$ & $\phi_c^\ast$ \\
\hline
\hline
0.01 & 0.0305  & 0.0299 & 0.0298 \\
0.03 &  0.0950 & 0.0893 & 0.0880 \\
$\vdots$ & $\vdots$ & $\vdots$ & $\vdots$  \\
0.40 & 3.8463  & 1.1111 &  0.9383 \\
$\vdots$ & $\vdots$ & $\vdots$ & $\vdots$  \\
0.584 & 870.5859   & 1.5822 & 1.2545 \\
0.5845 & 1.492e+03 & 1.5834 & 1.2553 \\
0.585 & 5.209e+03  & 1.5847 & 1.2561 \\
\hline
\end{tabular}
\medskip
\caption{$K=0$, $\veps_0^\ast\approx 0.5852$}
    \end{subtable}
 \begin{subtable}[t]{.45\linewidth}
      \centering
        \begin{tabular}[t]{|c||c||c|c|c|}
\hline
$\veps$ & $n_s$ & $n_c^\ast$  & $u_c^\ast$ & $\phi_c^\ast$ \\
\hline
\hline
.01 & 30.95 & 0.0305 & 0.0299 &  0.0298 \\
.03 & 31.58 & 0.0949 &  0.0893 &  0.0880 \\
$\vdots$ & $\vdots$ & $\vdots$ & $\vdots$ & $\vdots$  \\
.40 & 43.28 & 3.8569 & 1.1121 & 0.9375 \\
$\vdots$ & $\vdots$ & $\vdots$ & $\vdots$ & $\vdots$  \\
.5465 & 47.92 & 38.6476 & 1.5080 &  1.1922 \\
.5470 & 47.93  & 41.2877  &   1.5109
 &  1.1930 \\
.5474 & 47.94  & 45.3119 & 1.5145 & 1.1936 \\
\hline
\end{tabular}
\medskip
\caption{$K=0.001$, $\veps_K\approx 0.5475$}
    \end{subtable}
   \medskip
\begin{subtable}[t]{.45\linewidth}
      \centering
        \begin{tabular}[t]{|c||c||c|c|c|}
\hline
$\veps$ & $n_s$ & $n_c^\ast$  & $u_c^\ast$ & $\phi_c^\ast$ \\
\hline
\hline
.002 & .4162 & .0043 & .0060 & .0042 \\
$\vdots$ & $\vdots$ & $\vdots$ & $\vdots$ & $\vdots$  \\
.070 & .4842 & .1690 & .2146 & .1393 \\
$\vdots$ & $\vdots$ & $\vdots$ & $\vdots$ & $\vdots$  \\
.1550 & .5692 &  .5529 &  .5587 & .2805 \\
.1552 & .5694 &  .5611 &  .5641 & .2808 \\
\hline
\end{tabular}
\medskip
\caption{$K=1$, $\veps_K\approx 0.1553$}
    \end{subtable}
 \begin{subtable}[t]{.45\linewidth}
      \centering
        \begin{tabular}[t]{|c||c||c|c|c|}
\hline
$\veps$ & $n_s$ & $n_c^\ast$  & $u_c^\ast$ & $\phi_c^\ast$ \\
\hline
\hline
.002 & .0494 & .0018 & .0060 &  .0018 \\
$\vdots$ & $\vdots$ & $\vdots$ & $\vdots$ & $\vdots$  \\
.030 & .0583 & .0295 & .0958 & .0256 \\
$\vdots$ & $\vdots$ & $\vdots$ & $\vdots$ & $\vdots$  \\
.0522 & .0653  & .0634  &   .2009
 &  .0417 \\
.0523 & .0653  & .0642 & .2033 & .0418 \\
\hline
\end{tabular}
\medskip
\caption{$K=10$, $\veps_K\approx 0.0524$}
    \end{subtable}
\caption{Peak values of $n_c$, $u_c$, and $\phi_c$.}\label{TableMax}
\end{table}

\subsection{Global existence of the pressureless Euler-Poisson system near solitary waves}
We discuss some issues regarding the global existence of the pressureless Euler-Poisson system near the solitary wave solutions. 

A sufficient condition for the initial data which leads to the finite-time singularity formation for the pressureless Euler-Poisson system (\eqref{EP}--\eqref{bdCon x2} with $K=0$) is studied in \cite{Liu}: if the initial data $(n_0,u_0)(x)$   satisfies
\begin{equation}\label{conj3}
\partial_x u_0(x) \le  -\sqrt{2(1+n_0(x))}
\end{equation}
at some point $x \in \mathbb{R}$, then the maximal existence time of the smooth solution is finite.   Interestingly, our numerical experiments demonstrate that
\begin{equation}\label{Conj2}
\inf_{x \in \mathbb{R}}(\partial_xu_c/\sqrt{1+n_c}) \searrow -\sqrt{2} \quad \text{ as }\veps \nearrow \veps_0^\ast\approx 0.5852.
\end{equation}
See Figure \ref{FigFinSing} for the numerical plot of $\partial_xu_c/\sqrt{1+n_c}$ with $\veps=0.585<\veps_0^\ast$. From this numerical experiment together with the above mentioned study of \cite{Liu}, we expect that there may be a certain \emph{critical threshold phenomena} in the pressureless Euler-Poisson system. More precicely, our conjecture is as follows:\\

\textit{The pressureless Euler-Poisson system \eqref{EP}--\eqref{bdCon x2} with  $K=0$ admits a global smooth solution if and only if the initial data satisfies
\begin{equation}\label{conj3-1}
\partial_x u_0(x) >  -\sqrt{2(1+n_0(x))}
\end{equation}
for all $x\in\mathbb{R}.$}\\

We note that as pointed out in  \cite{Liu}, the condition \eqref{conj3} is the same as the critical threshold (see \cite{ELT}) for a different type of the pressureless Euler-Poisson system with repulsive force.

We also remark that if our conjecture is proved, the global existence of smooth solutions  \emph{near the  solitary waves} is guaranteed, which will be a first step to study the nonlinear stability. 
 For the isothermal case $K>0$, the related questions such as global existence of smooth solutions and finite-time singularity formation are widely open.

\begin{figure}[h]
\centering
\resizebox{80mm}{!}{\includegraphics{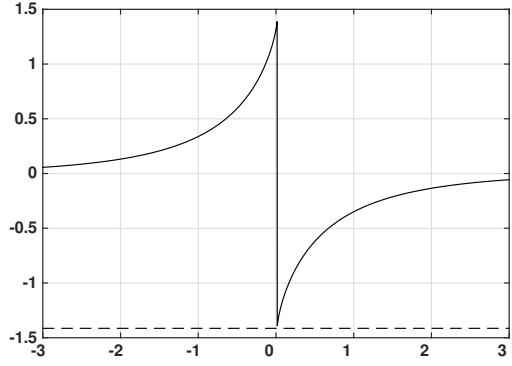}}
\caption{The gragh (solid) of $\partial_xu_c(x)/\sqrt{1+n_c(x)}$ for $\veps=0.585<\veps_0^\ast$. The horizontal (dashed) line represents $-\sqrt{2}$.}\label{FigFinSing}
\end{figure}

\section{Appendix}\label{Sec_LS_Apx}

\subsection{Eigenvalue problems of the Euler-Poisson system and the KdV equation}\label{Appendix1}
We present a formal reduction of the eigenvalue problem of the Euler-Poisson system \eqref{EigenEP} to the eigenvalue problem of the KdV equation. By intoducing the scaling
\begin{equation*}
\xi := \veps^{1/2}x, \quad \lambda := \veps^{3/2}\Lambda, \quad (n_\ast,u_\ast,\phi_\ast)(\xi):=\veps^{-1}(n_c,u_c,\phi_c)(x), \quad (\dot{n}_\ast, \dot{u}_\ast, \dot{\phi}_\ast)(\xi):=(\dot{n},\dot{u},\dot{\phi})(x),
\end{equation*}
\eqref{EigenEP} becomes (letting $(\dot{n}_\ast,\dot{u}_\ast,\dot{\phi}_\ast)=(n,u,\phi)$ for notational simplicity)
\begin{subequations}\label{FormDerivEigEPtoKdV1}
\begin{align}[left = \empheqlbrace\,]
& \veps\Lambda n  -c\partial_\xi n + \partial_\xi u + \partial_\xi(\veps n_\ast u + \veps u_\ast n) = 0, \label{FormDerivEigEPtoKdV11} \\ 
& \veps\Lambda u -c \partial_\xi u   + K\partial_\xi\left(\dfrac{n}{1+\veps n_\ast}\right) +\partial_\xi (\veps u_\ast u)  = -\partial_\xi \phi, \label{FormDerivEigEPtoKdV12} \\ 
& -\veps\partial_\xi^2 \phi = n - e^{\veps \phi_\ast}\phi. \label{FormDerivEigEPtoKdV13}
\end{align}
\end{subequations}
By integrating \eqref{FormDerivEigEPtoKdV11}--\eqref{FormDerivEigEPtoKdV12} in $\xi$, we formally obtain that (recall that $c=\sqrt{1+K}+\veps$) 
\begin{subequations}\label{FormDerivEigEPtoKdV2}
\begin{align}[left = \empheqlbrace\,]
& -\sqrt{1+K} \, n +  u = O(\veps), \label{FormDerivEigEPtoKdV21} \\ 
& K n - \sqrt{1+K} \, u + \phi =O(\veps), \label{FormDerivEigEPtoKdV22} \\ 
& n-\phi=O(\veps). \label{FormDerivEigEPtoKdV23}
\end{align}
\end{subequations}
Taking derivative of \eqref{FormDerivEigEPtoKdV13} in $\xi$, and then subtracting the resulting equation from \eqref{FormDerivEigEPtoKdV12}, $-\partial_\xi\phi$ term in the RHS of \eqref{FormDerivEigEPtoKdV12} is canceled. Then, by applying the Taylor expansion, we obtain
\begin{multline}\label{FormDerivEigEPtoKdV3}
\veps\Lambda u  -(\sqrt{1+K}+\veps) \partial_\xi u +  (K + 1 )\partial_\xi n   - K \partial_\xi(\veps n_\ast n) + \partial_\xi(\veps u_\ast u)
 + \veps\partial_\xi^3\phi - \partial_\xi(\veps \phi_\ast  \phi)\\
  =O(\veps^2).
\end{multline}
Multiplying \eqref{FormDerivEigEPtoKdV11} by $\mathsf{V}=\sqrt{1+K}$ and then adding the resulting equation to \eqref{FormDerivEigEPtoKdV3}, we see that the terms $-\sqrt{1+K}\,\partial_\xi u$ and $(1+K)\partial_\xi n$  in \eqref{FormDerivEigEPtoKdV3} are canceled, and we have
\begin{equation}\label{FormDerivEigEPtoKdV9}
\begin{split}
& \mathsf{V}\Lambda n  -\mathsf{V}\partial_\xi n  + \mathsf{V}\partial_\xi(n_\ast u+u_\ast n) \\
& + \Lambda u  - \partial_\xi u  - K \partial_\xi(n_\ast n) + \partial_\xi(u_\ast u) + \partial_\xi^3\phi - \partial_\xi(\phi_\ast\phi) =O(\veps).
\end{split}
\end{equation}
Using the relation \eqref{FormDerivEigEPtoKdV2} and Theorem \ref{MainThm4}, we obtain from \eqref{FormDerivEigEPtoKdV9} that
\begin{equation}
\Lambda n - \partial_\xi n  + \mathsf{V}\partial_\xi(\Psi_\textrm{K} n) + \frac{1}{2\mathsf{V}}\partial_\xi^3 n = O(\veps).
\end{equation}

\subsection{Specific form of $A_\ast(\xi,\Lambda,\veps)$}\label{SFA}
We denote $\partial_\xi$ by $'$ for simplicity. We let ${A_1}_\ast(\xi,\veps):=A_1(x,\veps)$, ${A_2}_\ast(\xi,\veps):=A_2(x,\veps)$, and $(n_\ast,u_\ast,\phi_\ast)(\xi):=\veps^{-1}(n_c,u_c,\phi_c)(x)$ for $\xi=\veps^{1/2}x$.  Then from \eqref{A_Decompose}  and \eqref{pointesti2}, we have
\[
{A_1}_\ast  =  \small\begin{pmatrix}
\frac{(c-\veps u_\ast)\veps^{3/2} u_\ast'}{J} - \frac{K\veps^{3/2}n_\ast'}{J(1+\veps n_\ast)} & \frac{(c-\veps u_\ast)\veps^{3/2}n_\ast'}{J} + \frac{(1+\veps n_\ast)\veps^{3/2} u_\ast'}{J} & 0 & \frac{1+\veps n_\ast}{J} \\
\frac{K\veps^{3/2} u_\ast'}{J(1+\veps n_\ast)} -\frac{K(c-\veps  u_\ast)\veps^{3/2}n_\ast'}{J(1+\veps n_\ast)^2} & \frac{K\veps^{3/2} n_\ast'}{J(1+\veps n_\ast)} + \frac{(c-\veps u_\ast)\veps^{3/2} u_\ast'}{J} & 0 & \frac{(c-\veps u_\ast)}{J} \\
0 & 0 & 0 & 1 \\
-1 & 0 & e^{\veps \phi_\ast} & 0
\end{pmatrix},
\]
\[
\frac{1}{\sqrt{\veps}}S^{-1}{A_1}_\ast S 
=
\small\begin{pmatrix}
\dfrac{a_{22}-\mathsf{V}a_{12}}{\veps^{1/2}} & \frac{a_{21} + \mathsf{V}a_{22} -\mathsf{V}a_{11} - \mathsf{V}^2a_{12} }{\veps^{3/2}} & \dfrac{a_{24}-\mathsf{V}a_{14} }{\veps} & 0 \\
\veps^{1/2} \, a_{12} & \dfrac{a_{11} + \mathsf{V}a_{12}}{\veps^{1/2}} & a_{14} & 0 \\
0 & \dfrac{e^{\veps \phi_\ast}-1}{\veps} & 0 & e^{\veps \phi_\ast} \\
\dfrac{-a_{12}}{\veps^{1/2}} & -\dfrac{a_{11} + \mathsf{V}a_{12}}{\veps^{3/2}} & \dfrac{1-a_{14}}{\veps} & 0
\end{pmatrix},
\]
\[
\frac{\veps^{3/2}\Lambda}{\sqrt{\veps}}S^{-1}{A_2}_\ast S= \frac{\Lambda}{J}
\small\begin{pmatrix}
\veps ((c-\veps  u_\ast) - \mathsf{V}(1+\veps n_\ast))  & \frac{K}{1+\veps n_\ast} - \mathsf{V}^2(1+\veps n_\ast) & 0 & 0 \\
\veps^2 (1+\veps n_\ast) & \veps(c-\veps  u_\ast)+\veps\mathsf{V}(1+\veps n_\ast) & 0 & 0 \\
0 & 0 & 0 & 0 \\
-\veps(1+\veps n_\ast) & -(c-\veps  u_\ast)-\mathsf{V}(1+\veps n_\ast) & 0 & 0
\end{pmatrix},
\]
where $a_{jk}$ is the $(j,k)$-entry of the matrix ${A_1}_\ast$.  Here, we note that 
\[
\frac{1-a_{14}}{\veps} = \frac{(c-\veps  u_\ast)^2- K - (1+ \veps n_\ast)}{\veps J}  = \frac{2\sqrt{1+K} + \veps - 2c u_\ast + \veps  u_\ast - n_\ast}{J}.
\]
Hence, we may define $A_\ast(\xi,\Lambda,\veps)$ for all $\veps\in [0,\veps_\ast]$, and using \eqref{pointesti2}, it is straightforward to check that  we have \eqref{ODEscaled2} at $\veps=0$.

\subsection{Proof of Lemma \ref{Fredequiv}}\label{Lem39}
Let $\mc{A}(\lambda):=d/dx - A(x,\lambda,\veps)$. For a closed subspace $\textrm{Ran}(\lambda-\mc{L})$ of a Hilbert space $\mc{H}$, we have 
\[
\mc{H}=\textrm{Ran}(\lambda-\mc{L}) \oplus \textrm{Ran}(\lambda-\mc{L})^\perp = \textrm{Ran}(\lambda-\mc{L}) \oplus \textrm{Ker}\left(\lambda - \mc{L} \right)^\ast \quad ( \,^\ast:= \textrm{Hermitian adjoint})
\]
and a similar decomposition holds for a closed subspace $\textrm{Ran}\mc{A}(\lambda)$. We claim that
\begin{enumerate}
\item[C1.] $\textrm{Ran}(\lambda -\mc{L})$ is closed if and only if $\textrm{Ran}(\mc{A}(\lambda))$  is closed;
\item[C2.] $\textrm{Ker}(\lambda -\mc{L})$ is isomorphic to $\textrm{Ker}(\mc{A}(\lambda))$ ;
\item[C3.] $\textrm{Ker}\left(\lambda - \mc{L} \right)^\ast$ is isomorphic to $\textrm{Ker}(\mc{A}(\lambda)^\ast)$. 
\end{enumerate}
Then, Lemma \ref{Fredequiv} follows by the definition of the Fredholm index of an operator.

C2 and C3 are already checked in subsection \ref{SS_RefEP}.\footnote{Note that $\mathcal{A}(\lambda)^\ast=-d/dx - \overline{A(x,\lambda,\veps)}^T = -d/dx - A^T(x,\overline{\lambda},\veps)$ from the form of the matrix $A$.}
We only prove the right direction of C1 since  the converse is easier to check. Also,  we only consider  the case of the unweighted $L^2$-space for simplicity.

 We suppose that $\textrm{Ran}(\lambda -\mc{L})$ is closed. For a sequence $\mathbf{f}_j=(f_j^1,f_j^2,f_j^3,f_j^4)^T \in \textrm{Ran}(\mc{A}(\lambda))$ such that $\mathbf{f}_j \to \mathbf{f}=(f^1,f^2,f^3,f^4)^T$ in $(L^2)^4$ as $j \to \infty$, let $\mathbf{y}_j:=(n_j,u_j,\phi_j,\psi_j)^T\in (H^1)^4$ be a solution of $\mc{A}(\lambda)\mathbf{y}_j=\mathbf{f}_j$ for each $j\in \mathbb{Z}$.

 We may decompose the last two components of $\mc{A}(\lambda)\mathbf{y}_j=\mathbf{f}_j$ (corresponding to the Poisson equation \eqref{Poisson_ODE}) into two parts:
\begin{equation}\label{CH_3PoissonDecomposition}
\left\{
\begin{array}{l l}
\partial_x\phi_j^f - \psi_j^f = f_j^3, \\ 
\partial_x\psi_j^f - e^{\phi_c}\phi_j^f = f_j^4,
\end{array} 
\right.
\quad
\left\{
\begin{array}{l l}
\partial_x(\phi_j-\phi_j^f) - (\psi_j - \psi_j^f) = 0, \\ 
\partial_x(\psi_j - \psi_j^f) - e^{\phi_c}(\phi_j-\phi_j^f) + n_j = 0.
\end{array} 
\right.
\end{equation}
Using the energy estimate, one can check that the solution $(\phi_j^f,\psi_j^f) \in (H^1)^2$ to the LHS of \eqref{CH_3PoissonDecomposition} exists for all $(f_j^3,f_j^4) \in (L^2)^2$, and $(\phi_j^f,\psi_j^f)$ converges to $(\phi^f,\psi^f)$ in $(H^1)^2$, where $(\phi^f,\psi^f)$ $\in$ $(H^1)^2$  satisfies
\begin{equation}\label{CH3PosDec}
\partial_x\phi^f - \psi = f^3, \quad 
\partial_x\psi^f - e^{\phi_c}\phi = f^4.
\end{equation}

We rewrite the first system of \eqref{CH_3PoissonDecomposition} and the system \eqref{CH3PosDec} as
\begin{equation}\label{fjfj}
\mathcal{A}(\lambda)\begin{pmatrix}
0 \\
0 \\
\phi_j^f \\
\psi_j^f
\end{pmatrix}
=
\begin{pmatrix}
L^{-1}\begin{pmatrix}
0 \\
\psi_j^f
\end{pmatrix} \\
f_j^3 \\
f_j^4
\end{pmatrix}
\quad \textrm{and}
\quad 
\mathcal{A}(\lambda)\begin{pmatrix}
0 \\
0 \\
\phi^f \\
\psi^f
\end{pmatrix}
=
\begin{pmatrix}
L^{-1}\begin{pmatrix}
0 \\
\psi^f
\end{pmatrix} \\
f^3 \\
f^4
\end{pmatrix},
\end{equation}
respectively.
Subtracting the LHS of \eqref{fjfj} from $\mc{A}(\lambda)\mathbf{y}_j=\mathbf{f}_j$, we have 
\[
\mc{A}(\lambda)\begin{pmatrix}
n_j \\
u_j \\
\phi_j-\phi_j^f \\
\psi_j-\psi_j^f
\end{pmatrix} = \begin{pmatrix}
\begin{pmatrix}
f_j^1 \\
f_j^2
\end{pmatrix}
-L^{-1}\begin{pmatrix}
0 \\
\psi_j^f
\end{pmatrix} \\
0 \\
0
\end{pmatrix},
\]
which implies that 
\[
(\lambda -\mc{L})\begin{pmatrix}
n_j\\
u_j
\end{pmatrix}
=L\left( \begin{pmatrix}
f_j^1 \\
f_j^2
\end{pmatrix}
-L^{-1}\begin{pmatrix}
0 \\
\psi_j^f
\end{pmatrix} \right) .
\]
Since  $\textrm{Ran}(\lambda -\mc{L})$ is closed, there is $(n,u)\in (H^1)^2$ such that 
\[
(\lambda -\mc{L})\begin{pmatrix}
n \\
u
\end{pmatrix} =L\left( \begin{pmatrix}
f^1 \\
f^2
\end{pmatrix}
-L^{-1}\begin{pmatrix}
0 \\
\psi^f
\end{pmatrix} \right),
\]
and this implies that 
\begin{equation}\label{fjfj2}
\mc{A}(\lambda)\begin{pmatrix}
n \\
u \\
\tilde{\phi}  \\
\tilde{\psi}
\end{pmatrix} = \begin{pmatrix}
\begin{pmatrix}
f^1 \\
f^2
\end{pmatrix}
-L^{-1}\begin{pmatrix}
0 \\
\psi^f
\end{pmatrix}  \\
0 \\
0
\end{pmatrix},
\end{equation}
where $(\widetilde{\phi},\widetilde{\psi})$ satisfies $\partial_x \widetilde{\phi} = \widetilde{\psi}$ and $ \partial_x\widetilde{\psi}  =  e^{\phi_c}\widetilde{\phi} - n$. Now adding \eqref{fjfj2} and the RHS of \eqref{fjfj}, we conclude that $\textrm{Ran}(\mc{A}(\lambda))$  is closed.

\subsection{Non-negativity of the matrix $S_1$} \label{NonnegaS}
We show that the symmetric matrix $S_1(x,\veps)$ defined in \eqref{S1sym} is non-negative for all sufficiently small $\veps>0$. 

Since $S_1$ is symmetric, it is enough to show that the eigenvalues of $S_1$,
\[
0, 0, 2\sqrt{K} R^{(1)}_{11} \pm \sqrt{2K^2( R^{(1)}_{12})^2 + 2( R^{(1)}_{21})^2},
\]
are non-negative.  Recalling the definition of the matrix $R^{(1)}$ (see  \eqref{R12}), $R_{11}^{(1)}$ is positive since
\[
\begin{split}
R_{11}^{(1)}
& = \frac{c-u_c}{J} - \frac{c}{c^2-K} \\
& = \frac{(c-u_c)(c^2-K) - c[(u_c-c)^2-K]}{J(c^2-K)} \\
& = \frac{u_c[c(c-u_c)+K]}{J(c^2-K)} > 0,
\end{split}
\]
where we have used the solitary wave identity \eqref{TravelEqB1} for the inequality.

 We check that $2\sqrt{K} R^{(1)}_{11} - \sqrt{2K^2( R^{(1)}_{12})^2 + 2( R^{(1)}_{21})^2}$ is positive. Since $R^{(1)}_{11}>0$, it is enough to show that 
\[
\begin{split}
  &  4K(R^{(1)}_{11})^2 - 2K^2(R^{(1)}_{12})^2 - 2(R^{(1)}_{21})^2  \\
  & =  \frac{2K}{J_1}\left[\underbrace{2\left((c-u_c)(c^2-K)-cJ \right)^2(1+n_c)^2}_{=:I_1=2(R^{(1)}_{11})^2J_1} \underbrace{- K\left((c^2-K)(1+n_c)-J \right)^2(1+n_c)^2}_{=:I_2}  \right. \\
& \left.\underbrace{-K\left(c^2-K-(1+n_c)J\right)^2}_{=:I_3}\right]
\end{split}
\]
is positive, where $J_1:=J^2(c^2-K)^2(1+n_c)^2>0$. Using the solitary wave identity \eqref{TravelEqB1} and the definition of $J=(c-u_c)^2-K$ (see \eqref{Useful_Id3}), we have
\begin{align*} 
 I_1 & = 2\left((c-u_c)(c^2-K)-c\left((c-u_c)^2-K \right) \right)^2(1+n_c)^2 \\
	 & = 2\left(c(c^2-K)-c\left(c(c-u_c)-K(1+n_c) \right) \right)^2 \\
     & = 2c^2 ( c u_c+K n_c  )^2, \\
      \\
 I_2 & = - K\Big((c^2-K)(1+n_c)-\left((c-u_c)^2-K \right) \Big)^2(1+n_c)^2 \\
     & = - K\Big((c^2-K)(1+n_c)^2-\big( c(c-u_c)-K (1+n_c) \big) \Big)^2 \\  
     & = - K\left(n_c(c^2-K)(2+ n_c)+\left(cu_c+K n_c\right) \right)^2 \\ 
     & = -K\left(cu_c+K n_c\right)^2 -Kn_c^2(c^2-K)^2(2+n_c)^2 - 2Kn_c(c^2-K)(2+n_c)\left(cu_c+K n_c\right) \\ 
     & = -K\left(cu_c+K n_c\right)^2 -4Kn_c^2(c^2-K)^2 - 4Kn_c(c^2-K)\left(cu_c+K n_c\right) \\
      & \quad  +  O\big(|n_c|^3 + |n_c|^2|u_c|\big) , \\
       \\
 I_3 & = -K\left(c^2-K-(1+n_c)\left((c-u_c)^2-K \right)\right)^2 \\
     & = -K\left(cu_c+K n_c\right)^2.
\end{align*}
Hence, 
\[
\begin{split}
I_1+I_2+I_3
&  =   2(c^2-K) ( c u_c+K n_c  )^2 -4Kn_c^2(c^2-K)^2 -  4Kn_c(c^2-K)\left(cu_c+K n_c\right)\\
& \quad  + O\big(|n_c|^3 + |n_c|^2|u_c|\big) \\
& =  2(c^2-K)( c u_c+K n_c  )\left( c u_c -K n_c \right) -4Kn_c^2(c^2-K)^2  \\
& \quad  + O\big(|n_c|^3 + |n_c|^2|u_c|\big) \\
& =   2(c^2-K)( c^2 u_c^2 + K^2 n_c^2 - 2K c^2 n_c^2  ) +O\big(|n_c|^3 + |n_c|^2|u_c|\big).
\end{split}
\]
Since $\frac{cn_c}{1+n_c}=u_c$ from \eqref{TravelEqB1}, we obtain that
\[
c^2 u_c^2 + K^2 n_c^2 - 2K c^2 n_c^2   = (c^2-K)^2n_c^2 \big( 1+O(|n_c|) \big).
\]
Therefore, we conclude that  $I_1+I_2+I_3>0$ for all sufficiently small $\veps>0$.

%

  \section*{Acknowledgments.}
B.K. was supported by Basic Science Research Program through the National Research Foundation of Korea (NRF) funded by the Ministry of science, ICT and future planning (NRF-2020R1A2C1A01009184)

 \end{document}